\documentclass[10pt]{article}

\title{
Global Jacobian and $\Gamma$-convergence in a two-dimensional Ginzburg-Landau model for boundary vortices
}
\author{{\Large Radu Ignat}\footnote{Institut de Math\'ematiques de Toulouse \& Institut Universitaire de France, UMR 5219, Universit\'e de Toulouse, CNRS, UPS
IMT, F-31062 Toulouse Cedex 9, France. Email: Radu.Ignat@math.univ-toulouse.fr} \and {\Large Matthias Kurzke}\footnote{School of Mathematical Sciences,
University of Nottingham, University Park, Nottingham, NG7 2RD, UK. Email: matthias.kurzke@nottingham.ac.uk}}

\usepackage[mathscr]{euscript}
\usepackage{pstricks} 
\usepackage{graphicx, pst-plot, pst-math} 
\usepackage{amsfonts, dsfont}
\usepackage{amsmath, latexsym}
\usepackage{amssymb,verbatim}
\usepackage{theorem}
\usepackage{todonotes}
\usepackage{color}
\usepackage[normalem]{ulem}

\newtheorem{lem}  {Lemma}[section]
\newtheorem{pro}[lem]    {Proposition}
\newtheorem{thm}[lem]    {Theorem}
\newtheorem{cor}[lem]    {Corollary}
\newtheorem{df}[lem]     {Definition}

\theorembodyfont{\rmfamily}
\newtheorem{rem}{Remark}[section]

\def\XXint#1#2#3{{\setbox0=\hbox{$#1{#2#3}{\int}$}
     \vcenter{\hbox{$#2#3$}}\kern-.5\wd0}}

\DeclareFontFamily{U}{matha}{\hyphenchar\font45}
\DeclareFontShape{U}{matha}{m}{n}{
      <5> <6> <7> <8> <9> <10> gen * matha
      <10.95> matha10 <12> <14.4> <17.28> <20.74> <24.88> matha12
      }{}
\DeclareSymbolFont{matha}{U}{matha}{m}{n}
\DeclareFontSubstitution{U}{matha}{m}{n}

\DeclareFontFamily{U}{mathx}{\hyphenchar\font45}
\DeclareFontShape{U}{mathx}{m}{n}{
      <5> <6> <7> <8> <9> <10>
      <10.95> <12> <14.4> <17.28> <20.74> <24.88>
      mathx10
      }{}
\DeclareSymbolFont{mathx}{U}{mathx}{m}{n}
\DeclareFontSubstitution{U}{mathx}{m}{n}
\DeclareMathDelimiter{\vvvert}{0}{matha}{"7E}{mathx}{"17}

\newcommand {\den} {{e}_\eta}
\newcommand {\ka}{\varkappa}

\newcommand{\ol}[1]{\overline{#1}}

\newcommand{\jacbd}{\mathcal{J}_{bd}}
\newcommand{\jaco}{\mathcal{J}}
\newcommand{\Ss}{\mathbb{S}}

\newcommand{\RR}{\mathbb{R}}
\newcommand{\R}{\mathbb{R}}
\newcommand{\C}{\mathbb{C}}

\newcommand{\eps}{\varepsilon}
\newcommand{\teps}{\tilde\eps}

\newcommand{\h}{{\mathcal{H}}}
\newcommand{\Hh}{{\mathcal{H}}}

\newcommand{\proof}[1]{\par\medskip\noindent{\bf Proof#1.}}
\newcommand{\qed}{\hfill$\square$}

\newcommand{\be}{\begin{equation}}
\newcommand{\ee}{\end{equation}}

\newcommand{\nd}{\noindent}

\newcommand{\NN}{\mathbb{N}}
\newcommand{\ZZ}{\mathbb{Z}}
\newcommand{\Om}{\Omega}
\newcommand{\dOm}{\partial\Omega}
\newcommand{\de}{\partial}
\newcommand{\dist}{\mathop{\rm dist \,}}

\newcommand{\jac}{\mathop{\rm jac \,}}

\newcommand{\f}{\varphi}

\renewcommand{\O}{\Omega}

\newcommand{\supp}{\operatorname{supp}}

\newcommand{\degr}{\operatorname{deg}}

\newcommand{\e}{\varepsilon}

\newcommand{\Ge}{{\cal G}_\eps}

\newcommand{\eee}{E_{\e,\eta}}

\newcommand{\inte}{\mathrm{int}}

\begin{document}

\maketitle

\begin{abstract}

In the theory of $2D$ Ginzburg-Landau vortices, the Jacobian plays a crucial role for the detection
of topological singularities. We introduce a related distributional quantity, called the global Jacobian that can detect both interior and boundary vortices for a $2D$ map $u$.
We point out several features of the global Jacobian, in particular, we prove an important stability 
property. This property allows us to study boundary vortices in a $2D$ Ginzburg-Landau model arising in thin ferromagnetic films, where a weak anchoring boundary energy penalising the normal component of $u$ at the boundary competes with the usual bulk potential energy. We prove an asymptotic expansion by $\Gamma$-convergence at the second order for this mixed boundary/interior energy in a regime where boundary vortices are preferred. More precisely, at the first order of the limiting expansion, the  energy is quantised and determined by the number of boundary vortices detected by the global Jacobian, while the second order term in the limiting energy expansion accounts for the interaction between the boundary vortices.

\medskip

\nd {\it AMS classification: } Primary: 35Q56, Secondary: 35B25, 49J45

\nd {\it Keywords: } Jacobian, stability, compactness, $\Gamma$-convergence, Ginzburg-Landau vortices, boundary vortices
\end{abstract}

\tableofcontents

\section{Introduction} 

For two small parameters $\eps>0$ and $\eta>0$, we consider the energy functional
\be
\label{eq:introenepseta}
E_{\eps,\eta}(u) = \int_{\Omega} |\nabla u|^2 \, dx + \frac1{\eta^2}\int_\Omega (1-|u|^2)^2 \, dx 
+ \frac{1}{2\pi\eps} \int_{\dOm} (u\cdot \nu)^2 \, d\h^1,
\ee
for every $2D$ map $u\in H^1(\Omega;\RR^2)$. The domain $\Omega\subset \R^2$ is a bounded, simply connected and $C^{1,1}$ regular domain (unless explicitly stated otherwise) with $\nu$ being the outer unit normal field on $\partial \Omega$. We 
always consider the following unit tangent field $$\tau=\nu^\perp=(-\nu_2, \nu_1) \quad \textrm{on} \quad \partial \Om,$$ so that $(\nu, \tau)$ forms an oriented frame on $\dOm$.
 
From a physical or modelling perspective, the functional \eqref{eq:introenepseta} has been used as a {somewhat} ad hoc model for thin ferromagnetic films, for example by Moser \cite{Moser:2003a}
and Cantero-\'Alvarez \cite{Cantero:2009a}, {highlighting an interplay between interior and boundary vortices}. In  \cite{IK_bdv}, we show explicit bounds that relate \eqref{eq:introenepseta} to 
{an effective micromagnetic energy} in a thin film regime  where boundary vortices are preferred. The results of the present article are essential in obtaining the $\Gamma$-convergence results for the full micromagnetic energy in that regime. A different regime corresponding to slightly larger films (where the nonlocality plays a more important role) was 
studied by Moser \cite{Moser:2004a}, who obtained convergence results only at the level of minimisers. We refer to \cite{IK_bdv} for a {thorough} discussion of the micromagnetic energy 
and the relevant thin-film regimes. 

From a purely mathematical point of view,  \eqref{eq:introenepseta} combines two penalisation terms leading 
to two well-known singularly perturbed problems that we explain in the following.

\medskip

\nd {\bf Ginzburg-Landau functional for interior vortices}. If we formally set $\eps=0$ in \eqref{eq:introenepseta}, then a finite energy configuration $u$ must be tangential to the boundary $\dOm$. Therefore, the following minimisation problem plays an essential role in our study for small $\eta>0$:
\be
\label{eq:introgl}
E^{GL}_\eta(u) = \int_{\Omega} |\nabla u|^2 \, dx + \frac1{\eta^2}\int_\Omega (1-|u|^2)^2 \, dx, \quad u\in H^1(\Omega;\RR^2),
\ee
within the boundary constraint that $u=\tau$  on $\dOm$. 
As $\Omega$ is a bounded simply connected $C^{1,1}$ regular domain, the tangent field $\tau$ has winding number $1$ on $\dOm$. This situation fits with the setting of the seminal book of Bethuel-Brezis-H\'elein 
\cite{BethuelBrezisHelein:1994a} who showed in particular that minimisers $u_\eta$ of \eqref{eq:introgl} with $u=\tau$ on $\dOm$ have an energy of leading order $2\pi |\log\eta|+ O(1)$ as $\eta\to 0$  and converge in various spaces to a singular $\Ss^1$-valued harmonic map having one point-singularity of topological degree $1$ (called {\it interior vortex point}). Moreover, for small $\eta>0$, minimisers satisfy $|u|\approx 1$ outside a single ``bad disc'' around the interior vortex point of
radius comparable to $\eta$, and the precise asymptotic behaviour of the minimal energy at the second order was determined in \cite{BethuelBrezisHelein:1994a} by introducing a novel notion of 
renormalised energy governing the location of the interior vortex point.

\medskip

\nd {\bf A weak anchoring energy for $\Ss^1$-valued maps}. If we formally set $\eta=0$ in \eqref{eq:introenepseta}, then a finite energy configuration must satisfy $|u|=1$ in $\Omega$. Therefore, 
we are interested in minimising the following weak anchoring energy for $\Ss^1$-valued maps with small $\eps>0$:
\be
\label{eq:introbv}
E^{KS}_\eps(u) = \int_{\Omega} |\nabla u|^2 \, dx + \frac{1}{2\pi\eps} \int_{\dOm} (u\cdot \nu)^2 \, d\h^1, \quad u\in H^1(\Omega;\Ss^1).
\ee
This model was first derived by Kohn-Slastikov \cite{KohnSlastiko:2005a} as a $\Gamma$-limit in a certain thin-film regime in micromagnetics. The asymptotic behaviour of the energy $E^{KS}_\eps$ at the minimal level when $\eps\to 0$ was studied by the second author \cite{Kurzke:2006b}: in particular, the minimisers have an energy of leading order $2\pi |\log\eps|+O(1)$ and converge to a $\Ss^1$-valued harmonic map with two boundary singularities. 
{Each of these singularities can be interpreted as carrying a ``half'' topological degree. }
For small $\eps>0$, minimisers satisfy $u\approx \pm \tau$ outside of two ``bad discs'' of radius comparable to $\eps$, 
and again, it is possible to precisely determine the asymptotics of the minimal energy at the second order.

Both \eqref{eq:introgl} and \eqref{eq:introbv} have also been studied from the point of view of $\Gamma$-convergence. A difficulty is that the diverging 
energies typically lead to a lack of compactness for the order parameter $u$. To overcome this problem, it was observed 
that instead of the map $u$, other quantities have much better compactness properties. In the case of the Ginzburg-Landau functional $E^{GL}_\eta$ for interior vortices, the natural 
quantity is the Jacobian determinant $\jac(u)=\det\nabla u$. It was shown that 
 for families $(u_\eta)_\eta$ with $E^{GL}_\eta(u_\eta)=O(|\log\eta|)$ as $\eta\to 0$, the Jacobians $(\jac(u_\eta))_\eta$ 
are precompact in  $(W^{1,\infty}_0(\Omega))^*$ and other dual spaces of functions that are zero on the boundary, see
Jerrard-Soner~\cite{JerrardSoner:2002b} or  Sandier-Serfaty
\cite{SandierSerfaty:2007a}.
The limits of the Jacobians are of the form $\pi \sum_k d_k \delta_{a_k}$ for some distinct points $a_k\in\Omega$, corresponding to interior vortex points, carrying the topological degrees $d_k\in \ZZ$.
As the Jacobian is controlled only in the dual space of functions that are zero on the boundary, there is no control over vortices escaping to the boundary.

For the weak anchoring energy $E^{KS}_\eps$ over $\Ss^1$-valued maps, the problem is slightly easier: as  a map $u_\eps$ with finite energy possesses a global lifting 
$u_\eps=e^{i\phi_\eps}$ with $\phi_\eps\in H^1(\Omega)$, then every family $(u_\eps)_\eps$ with $E^{KS}_\eps(u_\eps) =O(|\log\eps|)$ has the liftings $(\phi_\eps)_\eps$ (up to an additive constant) 
{precompact} in $L^p(\dOm)$ for every { $p\in [1, \infty)$}
with limits $\phi_*$ on $\partial \Om$ such that $ \partial_\tau \phi_*-\ka$ is a multiple of a sum of Dirac masses {on $\partial \Om$} (see \cite{Kurzke:2006a}). Here, $\partial_\tau$ denotes tangential 
differentiation and $\ka$ the curvature on $\dOm$. This approach relies strongly on the constraint $|u_\eps|=1$ in $\Omega$.

The energy \eqref{eq:introenepseta} allows for both types of topological phenomena (boundary and interior vortex), so we need a tool that captures these
singularities and does not require the existence of a global lifting. The natural tool is the notion of global Jacobian that we discuss in the next section. The $\Gamma$-convergence results for the energy \eqref{eq:introenepseta} are proved in this paper in the regime
\be
\label{regim_1}
|\log\eps|\ll |\log\eta|,
\ee 
i.e., interior vortices cost more energy than boundary vortices. 

\medskip

\nd {\bf Notation}. We always denote by $a_\eps\ll b_\eps$ or $a_\eps=o(b_\eps)$ if $\frac{a_\eps}{b_\eps}\to 0 $ as $\eps\to 0$, and 
similar $a_\eps\lesssim b_\eps$ or $a_\eps=O(b_\eps)$ if there exists $C>0$ such that $a_\eps\le C b_\eps$ for all {small} $\eps>0$. In the following, $\eps\to 0$ can mean both a sequence $\eps_k\to 0$ as well as the continuous parameter $\eps\to 0$. More precisely, in our $\Gamma$-convergence results, the limits $\eps \to 0$ in $\liminf$
(and $\limsup$) are understood in both cases sequence / family of the parameter $\eps$; only for the compactness result, 
we start with a fixed sequence $\eps_k\to 0$ and then 
we take further subsequences of this sequence.

\subsection{Global Jacobian} 
For a two-dimensional map $u\in H^1(\Omega;\R^2)$  defined in a Lipschitz bounded
domain $\Omega\subset \R^2$, we call {\it global Jacobian} of $u$ the following linear functional $\jaco(u):W^{1, \infty}(\Omega)\to \RR$ acting on Lipschitz test functions:
\be \label{eq:gj}
\left<\jaco(u), \zeta\right>:=-\int_{\Omega} u \times \nabla u\cdot \nabla^\perp \zeta\, dx, \quad \textrm{for every Lipschitz function }\, \zeta: \Omega\to \R.
\ee
Here $a\times b = a_1b_2-a_2b_1$ for $a,b\in \R^2$, $u\times \nabla u=(u\times \partial_{x_1} u, u\times \partial_{x_2} u)$ that belongs to $L^1(\Om; \R^2)$ for $u\in H^1(\Om; \R^2)$,  
$\nabla^\perp = (-\partial_{x_2}, \partial_{x_1})$ and $\left<\cdot, \cdot\right>$ stands for the  (algebraic) dual pairing
 between $(W^{1, \infty}(\Omega))^*$ and $W^{1, \infty}(\Omega)$. In particular, the global Jacobian has zero average, i.e., 
\be
\label{ave_0}
\left<\jaco(u), 1\right>=0.\ee
{\bf Relation with the interior and boundary Jacobian}. On the one hand, when applied to test functions $\zeta\in W^{1,\infty}(\Omega)$ vanishing at the boundary $\partial \Omega$, the global Jacobian $\jaco(u)$ reduces to twice the {\it interior 
Jacobian} $\jac(u)=\partial_{x_1} u \times \partial_{x_2} u \in L^1(\Omega)$ for $u\in H^1(\Omega;\R^2)$; indeed, integrating by parts, it follows 
\[
\left<\jaco(u), \zeta\right> =  \int_\Omega 2\jac(u) \zeta \, dx \quad  \textrm{ if } \, \zeta=0 \,  \textrm{ on } \, \partial \Omega.
\]
Therefore, the global Jacobian carries the topological information at the interior $\Omega$ and detects the interior vortices.

On the other hand, the global Jacobian also carries the topological information at the boundary $\partial \Omega$ and enables us to detect  boundary vortices; 
more precisely, we define the {\it boundary Jacobian} of $u\in H^1(\Omega;\R^2)$ to be the linear functional
 $\jacbd(u):W^{1, \infty}(\Omega)\to \RR$ given by
 \be
 \label{def_jac_bd}
 \left<\jacbd(u),\zeta\right>:=\left<\jaco(u),\zeta\right>-\int_\Omega 2 \jac(u)\zeta\, dx, \quad \textrm{for every Lipschitz function }\, \zeta: \Omega\to \R .\ee
 In fact, the functional $\jacbd(u)$ acts only on the boundary $\partial \Om$ (see Proposition \ref{pro:jacbd} below): in particular, if $u\in C^2(\bar \Omega; \R^2)$, then integration by parts yields 
 $$\left< \jacbd(u), \zeta\right>=-\int_{\partial \Om} u\times \partial_{\tau} u \zeta\,d\h^1 \quad \textrm{for every Lipschitz function }\, \zeta:\Omega\to \R,$$
 i.e., $\jacbd(u)=-u\times \partial_{\tau} u \, \h^1\llcorner \partial \Omega$. 
{While $\zeta$ is a priori only defined in $\Omega$, it has a unique Lipschitz extension to $\ol\Omega$, and we will tacitly use this extension in the following.}
In addition, for a $\Ss^1$-valued map $u$ given through a smooth lifting $\f\in C^2(\bar \Omega; \R)$, i.e., $u=(\cos \f, \sin \f)$ in $\Omega$, then the interior Jacobian $\jac(u)$ vanishes in $\Om$ so that the whole topological information is carried by the
tangential derivative of $\f$ at the boundary, i.e.,
 \be
 \label{jac_s1}
\jac(u)=0, \quad \jaco(u)=\jacbd(u)=-\partial_{\tau} \f \, \h^1\llcorner \partial \Omega \quad \textrm{and}\quad \left<\jacbd(u), 1\right>=0 \quad \textrm{ if } \, u=e^{i\f} \,  \textrm{ in } \, \Omega.
\ee

\subsection{Main results}

We show the following stability result for the global Jacobian. This is the generalisation of the well known estimate for the interior Jacobian (see e.g., Brezis-Nguyen \cite{BreNgu}, or Proposition \ref{pro:jac_int} below). 
\begin{thm}
\label{thm:teorem}
Let $\Omega\subset \R^2$ be a $C^{1,1}$ bounded domain
and $u, v\in H^1(\Omega; \R^2)$ such that
$$|v|\leq 1\quad \textrm{ a.e. in } \, \Om.$$ Then for every $\zeta \in W^{1,\infty}(\Omega)$, we have 
$$\left|\left<\jaco(u)-\jaco(v),\zeta\right>\right|\leq f\bigg(\|u-v\|_{L^2(\Om)}\big(\|\nabla u\|_{L^2(\Om)}+\|\nabla v\|_{L^2(\Om)} \big)\bigg)\|\nabla \zeta\|_{L^\infty(\Om)}$$
where the function $f$ is given by $f(t)=t+C\sqrt{t}$ with $C>0$ depending only on the geometry of $\Om$.
 \end{thm}
 Note that the above inequality can be interpreted as a stability property of the global Jacobian $\jaco(\cdot)$ in the strong $L^2$-topology of maps under a certain control on their $H^1$-seminorm (that eventually could blow up).
This allows us to make perturbations of $u$ that are small in $L^2$, but possibly large in $H^1$ without changing the 
global Jacobian much. 

Theorem \ref{thm:teorem} is an important tool in proving the compactness result in Theorem \ref{thm:firstocomp} of the global Jacobian for configurations $u_\eps$ satisfying the energetic bound $E_{\eps, \eta}(u_\eps)\leq C |\log \eps|$ in the regime \eqref{regim_1}.  In addition, 
we prove the compactness of the trace $u_\eps\big|_{\dOm}$ in the strong $L^p(\dOm)$-topology for every $p\geq 1$. This compactness result of $u_\eps\big|_{\dOm}$ is {\bf very surprising} in the context of Ginzburg-Landau type functionals where in general, no compactness of configurations $u_\eps$ is expected to occur. Moreover, under a more restrictive energetic regime, we prove that strong $L^p(\Om)$-convergence of $u_\eps$ inside $\Om$ does also occur, see Theorem \ref{thm:gammacforee}. 
The role of Theorem~\ref{thm:firstocomp} consists also in proving a lower bound of the energy at the first order that is quantised by the number of boundary vortices detected by the global Jacobian.

\begin{thm}\label{thm:firstocomp} 
Let $\Omega\subset \R^2$ be a bounded, simply connected, $C^{1,1}$ regular domain and $\ka$ be the curvature on $\dOm$. 
If $\eps\to 0$ and $\eta=\eta(\eps)\to 0$ satisfy $|\log\eps|\ll |\log\eta|$, then the following holds: Assume $u_\eps\in H^1(\Om;\R^2)$ satisfy
\[
\limsup_{\eps\to 0} \frac{1}{|\log\eps|}E_{\eps,\eta}(u_\eps) <\infty.
\]
\begin{itemize}
\item[i)] {\textbf{Compactness of global Jacobians and $L^p(\dOm)$-compactness of $u_\eps\big|_{\partial \Om}$}. 
For a subsequence,
 the global Jacobians $\jaco(u_\eps)$ 
 converge  to 
a measure $J\in\mathscr{M}(\overline{\O})$ on the closure $\overline{\O}$, in the sense that 
 \be
  \label{conv_lip}
  \sup_{|\nabla \zeta|\leq 1\textrm{ in } \Omega}\left|\left<\jaco(u_\eps)-J,\zeta\right>\right|\to 0 \quad \textrm{as }\, \eps\to 0,
  \ee
$J$ is supported on $\dOm$ and has the form 
\be
\label{newlab}
J=-\ka \, {\h^1\llcorner \dOm}+\pi \sum_{j=1}^N d_j \delta_{a_j} \quad \textrm{with} \quad  \sum_{j=1}^N d_j=2
\ee
for $N$ distinct boundary vortices $a_j \in \partial\O$ carrying the {non-zero} {multiplicities\footnote{We use 
integer ``multiplicities'' instead of half-integer ``degrees'' for the boundary vortices in this article.}}
$d_j\in\ZZ\setminus \{0\}$. Moreover, for a subsequence, the trace $u_\eps\big|_{\partial \Om}$ converges as $\eps\to 0$ in $L^p(\dOm)$ (for every $p\geq 1$) to $e^{i\phi_0}\in BV(\dOm; \{\pm \tau\})$ for a lifting $\phi_0$ of the tangent field $\pm \tau$ on $\dOm$ determined (up to a constant in $\pi \ZZ$) by 
$$\partial_\tau \phi_0 =\ka -\pi\sum_{j=1}^N d_j \delta_{a_j} \quad \textrm{ on} \quad \dOm.$$
}

\item[ii)] \textbf{Energy lower bound at the first order}. If $(u_\eps)$ satisfies the convergence assumption in $i)$ as the sequence / family $\eps\to 0$, then
the energy lower bound at the first order is the total mass of the measure $J+\ka {\h^1\llcorner \dOm}$ on $\dOm$: 
\footnote{Recall that $J+\ka  {\h^1\llcorner \dOm} = \pi \sum_{j=1}^N d_j \delta_{a_j}$.}
\[
\liminf_{\eps\to 0} \frac1{|\log\eps|}E_{\eps,\eta}(u_\eps) \ge\pi\sum_{j=1}^N |d_j|=\big|J+\ka  {{\h^1\llcorner \dOm}}\big|(\dOm).
\]
\end{itemize}
\end{thm}

Note that the limit Jacobian measure \eqref{newlab} lives on the boundary $\partial \Om$, having a diffuse part carried by the curvature $\ka$ and a singular part carried by (multiples of) Dirac masses at the boundary vortices $a_j$.  The convergence \eqref{conv_lip} is discussed in Section \ref{sec:JACO}. In particular, by \eqref{ave_0} and \eqref{conv_lip}, we have $\left<J,1\right>=0$; thus, 
combined with the Gau\ss{}-Bonnet formula, we have that $$\pi \sum_{j=1}^N d_j=\int_{\dOm}\ka\, d\h^1=2\pi.$$ This explains the constraint \eqref{newlab} on the sum of the multiplicities $(d_j)_j$. The $BV$ lifting $\phi_0$ on $\dOm$ is determined by $\partial_\tau \phi_0 =-J$ up to an additive constant that a-priori is arbitrary in $\R$; however, the restriction that the limit $e^{i\phi_0}$ is parallel with $\tau$ fixes this constant to be a multiple of $\pi$. 

Theorem \ref{thm:firstocomp} is carried out in the regime \eqref{regim_1}, so that the formation of boundary singularities is preferred over interior singularities.
In particular, we have that $\eta\ll \eps$, so the 
typical core size of an interior vortex is much smaller than the length scale of a boundary transition from a parallel to an 
antiparallel tangent direction $\pm \tau$ at $\dOm$. In this context, as interior vortices of non-zero winding number are expected to be absent, we prove in  Theorem~\ref{lem_approx} below that $u_\eps$ can be replaced by an $\Ss^1$-valued map
 without raising the energy by much and without affecting the convergence and limit of the global Jacobians (thanks to Theorem \ref{thm:teorem}). 
The $\Ss^1$-valued problem is studied in Section~\ref{sec:s1val} (in particular Theorem \ref{thm:GCforGeps}), and we improve results in the literature
 \cite{Kurzke:2006b, Kurzke:2006a, Kurzke:2007a} by giving simpler, more direct proofs and obtaining new and significantly stronger results for the second order energy expansion.
In particular, we adapt a co-area argument of Sandier \cite{Sandier:1998a} in the nonlocal context of \eqref{eq:introbv} 
(see the rewriting \eqref{def:feps} below) to show a new single multiplicity result
and use arguments inspired by 
Colliander-Jerrard \cite{ColliandJerrard:1999a} to obtain lower bounds using purely energy methods.
 Owing to our
 approximation Theorem~\ref{lem_approx}, these results can then be transferred to the study of \eqref{eq:introenepseta}. 

For the analysis of the asymptotic expansion at the second order,
we need to introduce a
 renormalised energy similar to that of Bethuel-Brezis-H\'elein \cite{BethuelBrezisHelein:1994a}
 that consists in eliminating the ``infinite" energy carried asymptotically in small disks around the boundary vortices.
\begin{df}\label{defi:renen} 
Let $\Omega\subset \R^2$ be a bounded, simply connected, $C^{1,1}$ regular domain and $\ka$ be the curvature on $\dOm$.  Consider
$\phi_0:\partial\Omega\to \R$ to be a
 $BV$ function such that  $e^{i\phi_0}\cdot \nu=0$ in $\partial\Omega\setminus\{a_1,\dots,a_N\}$ and
$$\partial_\tau \phi_0 =\ka  -\pi\sum_{j=1}^N d_j \delta_{a_j}\, \, \textrm{ on } \partial \Om \quad \textrm{with}\quad d_j\in \{\pm 1\} \textrm{ and }\sum_{j=1}^N d_j=2$$ for $N$ distinct points $a_j\in \dOm$ carrying the degrees $d_j\in \{\pm 1\}$. If $\phi_*$ is the harmonic extension to $\Omega$ of $\phi_0$, then the \emph{\bf renormalised energy} of $\{(a_j,d_j)\}$ is defined 
as
\begin{equation}\label{eq:renW}
W_\Omega(\{(a_j,d_j)\}) =  \lim_{\rho\to 0} \left( 
\int_{\Omega \setminus \bigcup_{j=1}^N B_\rho(a_j)} |\nabla \phi_*|^2 \,dx - N\pi \log\frac1\rho
\right),
\end{equation}
where $B_\rho(a_j)$ is the disk of radius $\rho$ centered at $a_j$. 
\end{df}

In Definition \ref{defi:renen}, $\phi_0$ is uniquely determined (up to an additive constant) and stands for a  
$BV$ lifting of some tangent unit vector field ${\pm\tau}$ on $\dOm$ with prescribed jumps at $a_j$ (see e.g., \cite{Ig_CV} for more details on $BV$ liftings). The difference with respect to the lifting in Theorem \ref{thm:firstocomp}, point {\it i)} consists in allowing here only jumps of $\pm \pi$ at the boundary vortices $a_j$.
Note also that the limit in \eqref{eq:renW} exists, see \cite{Kurzke:2006b}. The renormalised energy $W_\Omega(\{(a_j,d_j)\})$ can be computed in a $C^{1,1}$-domain $\Omega$, in particular, it depends on $\log |a_j-a_k|$ for every $j\neq k$ and on the curvature $\ka$ of $\partial \Om$, see \cite{IK_bdv} for details. In a disk $\Omega=B_R$,
the renormalised energy has a particularly simple form:
\be
\label{numb}
W_{B_R}(\{a_j,d_j\})=-2\pi \sum_{1\leq k< j\leq N} d_k d_j \log|a_k-a_j|, \quad a_j\in \partial B_R, \, d_j\in \{\pm 1\} \textrm{ and }\sum_{j=1}^N d_j=2.
\ee
In particular, if $N=2$, then $d_1=d_2=1$ and the renormalised energy achieves the minimum value $-2\pi \log 2R$ 
 for
 diametrically opposed singularities and has no other critical points than this rotationally 
symmetric family of minimisers. 

We have the following refinement of Theorem~\ref{thm:firstocomp} at the second order using the renormalised energy \eqref{eq:renW}:
\begin{thm}
\label{thm:gammacforee} 
Under the hypothesis of Theorem~\ref{thm:firstocomp}, we assume that the sequence / family $(u_\eps)$ satisfies the convergence at point $i)$ in Theorem~\ref{thm:firstocomp} as $\eps\to 0$. 
In addition, we assume the following sharper bound:
\be
\label{eq:sharpenbd}
\limsup_{\eps\to 0} \bigl(E_{\eps,\eta}(u_\eps) - 
|\log\eps| \pi\sum_{j=1}^N |d_j|\bigr) <\infty.
\ee
Then the following results hold:
\begin{itemize}
\item[i)] \textbf{Single multiplicity and second order lower bound}.
The multiplicities satisfy $d_j=\pm 1$ for $1\leq j\leq N$, { so $\sum_{j=1}^N |d_j|=N$}  and 
{there holds the finer energy bound}
\[
\liminf_{\eps\to 0} \bigl(E_{\eps,\eta}(u_\eps) - |\log\eps| \pi {N}\bigr) 
\ge W_\Omega(\{a_j,d_j\}) + {\gamma_0}{N},
\]
with $\gamma_0=\pi\log\frac{ e}{4\pi}$ a universal constant and $W_\Omega$ the renormalised
energy defined in \eqref{eq:renW}. 
\item[ii)] \textbf{Penalty bound.}
The penalty terms {are of order $O(1)$, i.e., }
\begin{equation}\label{eq:thpenub}
\limsup_{\eps\to 0}\left( \frac1{\eta^2} \int_\Omega (1-|u_\eps|^2)^2\, dx + \frac1{2\pi \eps} \int_{\partial\Omega} 
(u_\eps \cdot \nu)^2 \, d\h^1 \right)<\infty.
\end{equation}
\item[iii)] \textbf{Local energy lower bound.}
There are $\rho_0>0$, $\eps_0>0$ and $C>0$ such that 
the energy of $u_\eps$ near the singularities satisfies for all the {$\eps<\eps_0$ in the sequence / family } and $\rho<\rho_0$: 
\begin{equation}\label{eq:l2lob}
\left( \int_{\Omega\cap \bigcup_{j=1}^N B_\rho(a_j)} |\nabla u_\eps|^2 \, dx -\pi {N} \log \frac\rho\eps\right) > -C.
\end{equation}
\item[iv)] \textbf{$L^p(\Om)$-compactness of maps $u_\eps$.}
For any $q\in[1,2)$, 
the sequence /family $(u_\eps)_\eps$ is uniformly bounded in $W^{1,q}(\Omega;\R^2)$. {Moreover, for a subsequence, $u_\eps$ converges as $\eps\to 0$ strongly in $L^p(\Omega;\R^2)$ for any $p\geq 1$ to 
$e^{ i\hat\phi_0}$, where  $\hat\phi_0\in W^{1,q}(\Omega)$ is an extension (not necessarily harmonic) to $\Omega$ of the lifting $\phi_0\in BV(\dOm)$ determined in Theorem ~\ref{thm:firstocomp}, point i).  }

\end{itemize}
\end{thm}

\bigskip

Finally, we have a matching upper bound that complements  Theorems~\ref{thm:firstocomp} and~\ref{thm:gammacforee} to yield a full asymptotic expansion by $\Gamma$-convergence 
 at the second order for the energy \eqref{eq:introenepseta}.
\begin{thm}\label{thm:gcubintro}
Given any $N$ distinct points $a_j\in \partial \O$ with their multiplicity $d_j\in \ZZ\setminus \{0\}$
satisfying the constraint $\sum_{j=1}^N d_j=2$,
we can construct for every $\eps\in (0, \frac12)$, $u_\eps\in H^1(\Omega; \Ss^1)$ such that the global Jacobians $\jaco(u_\eps)$ converge to $J=-\ka {\h^1\llcorner \dOm}+\pi \sum_{j=1}^N d_j \delta_{a_j}$ as in \eqref{conv_lip}.
{Furthermore, $u_\eps$ converge strongly  to $e^{i\hat \phi}$ in $L^p(\Om)$ and $L^p(\dOm)$ for all $p\in[1,\infty)$,
where $\hat \phi$ is the harmonic extension in $\Om$ of a boundary lifting $\phi_0$ satisfying $e^{i\phi_0}\cdot \nu =0$ and $\partial_\tau \phi_0=\ka-\pi\sum_{j=1}^N d_j \delta_{a_j}$ on $\dOm$. The energies satisfy}
\[
\lim_{\eps\to 0} \frac{1}{|\log\eps|}E_{\eps,\eta}(u_\eps) = \pi \sum_{j=1}^N |d_j|.
\]
If furthermore $|d_j|=1$ for all $j=1,\dots,N$, then $u_\eps$ can be chosen such that
\[
\lim_{\eps\to 0}  (E_{\eps,\eta}(u_\eps)-\pi N|\log\eps|) = W_\Omega(\{a_j, d_j\})+ N\gamma_0.
\]
\end{thm}
\begin{rem} As a consequence of Theorems \ref{thm:gammacforee} and \ref{thm:gcubintro}, if $u_\eps$ are minimisers of $\eee$, then \eqref{eq:sharpenbd} is satisfied, and by standard properties of $\Gamma$-convergence we find a limit Jacobian
corresponding to  two singularities $a_1, a_2\in \dOm$, $a_1\neq a_2$ with multiplicity $1$
whose positions minimise the renormalised energy $W_\Omega(\{(a_1,1), (a_2,1)\})$. So if $\Omega$ is a disk, these two  (limit) singularities $a_1$ and $a_2$ are diametrically opposite to each other thanks to \eqref{numb}. 
As the minimisers of $W_\Om$ in a disk are not unique, the convergence of the Jacobians (and of the maps $u_\eps$) only needs to hold up to subsequences. 
\end{rem}

Some of our results were announced in \cite[Section 11]{Ignat_HDR}. 
We expect that our results can be extended to  situations 
where both interior and boundary vortices are present 
as long as sufficiently tight energy bounds 
 hold, and they generalise the results for minimisers of Moser \cite{Moser:2003a} (see also the case of boundary ``boojums" in
a liquid crystal model studied by Alama-Bronsard-Golovaty  \cite{boojum}). Even if boundary singularities are favourable compared to interior ones {in the regime \eqref{regim_1}}, 
{certain configurations with}
interior vortices are still conjectured to be local minimisers (see \cite{Cantero:2009a} for partial results). 
However, an extension of our method {will require} an approximation result that can be used in the presence of interior vortices, see Ignat-Otto\cite{IgnatOtto:2011a}. 
We also expect it is possible to extend our results on $2$-dimensional Riemannian manifolds with boundary, by following the approach of
Ignat-Jerrard \cite{IgJe, IgJeP}.

\section{Stability of the global Jacobian. Proof of Theorem \ref{thm:teorem}}
\label{sec:JACO}

In this section we discuss some properties of the global Jacobian $\jaco(u)$ introduced in \eqref{eq:gj} for a two-dimensional map $u\in H^1(\Om; \R^2)$ defined on a Lipschitz 
bounded domain $\Omega\subset \R^2$; in particular, we prove 
Theorem \ref{thm:teorem}. We recall that the global Jacobian $\jaco(u)$ is an element of the (algebraic) dual $(W^{1,\infty}(\Omega))^*$ of $W^{1,\infty}(\Omega)$. 
In order to speak about the continuity of this linear functional, some natural seminorms are considered on the space of Lipschitz functions $W^{1,\infty}(\Omega)$ and the subspace
$$W_0^{1,\infty}(\Omega)=\{\zeta \in W^{1, \infty}(\Om)\, :\, \zeta=0\,  \textrm{ on } \partial \Omega\}.$$ 
These seminorms lead to the following dual quantities that measure the global and the interior Jacobian: if $A\in (W^{1,\infty}(\Omega))^*$, we define
\begin{align*}
\| A \|_{(\mathrm{Lip}(\Omega))^*}&=\sup \left\{\left<A,\zeta\right> : \zeta \in W^{1,\infty}(\Omega),  |\nabla \zeta| \le 1 \right\}, \\
\| A \|_{(W^{1, \infty}(\Om))^*}&=\sup \left\{\left<A,\zeta\right> : \zeta \in W^{1,\infty}(\Omega), |\zeta|+ |\nabla \zeta| \le 1 \right\}, \\
\| A \|_{(W^{1, \infty}_0(\Om))^*}&=\sup \left\{\left<A,\zeta\right> : \zeta \in W_0^{1,\infty}(\Omega), |\zeta|+|\nabla \zeta| \le 1 \right\}.
\end{align*}
{\bf We write $\|A \|$ as a shorthand for $\|A\|_{(\mathrm{Lip}(\Omega))^*}$ and is the quantity we use in the next sections}. Note that by homogeneity,
\be
\label{hom}
\textrm{ $\|A\|<\infty$ implies $\left<A, 1\right>=0$.}
\ee
Clearly, we have for all $A\in (W^{1, \infty}(\Om))^*$:
$$\|A\|\ge \| A \|_{(W^{1, \infty}(\Om))^*}\ge \| A \|_{(W^{1, \infty}_0(\Om))^*}.$$  
In particular, 
\[
\|\jaco(u)\|\geq \|\jaco(u)\|_{(W_0^{1, \infty}(\Om))^*} 
 =2\|\jac(u)\|_{(W_0^{1, \infty}(\Om))^*}.
 \]
Identifying $\R^2$ with the complex plane, both operators $\jac(\cdot)$ and $\jaco(\cdot)$ are invariant under (complex) multiplication with a fixed unit length vector $a\in \Ss^1$ on $H^1(\Om; \R^2)$. While $\jac(\cdot)$ is invariant 
under addition of a fixed vector $a\in \R^2$,  $\jaco(\cdot)$ is not.\footnote{For example, consider $u$ of the form
$u=e^{i\f}$ with a smooth lifting $\f$ in $\bar \Omega$. Then \eqref{jac_s1} implies $\jaco(u)=\jacbd(u) = -\partial_\tau\f\, \h^1\llcorner \dOm$ and
$\jaco(u+1) = \jacbd(u+1)= -\partial_\tau(\f+\sin \f)\, \h^1\llcorner \dOm$ (because $\jac(u+1)=0$ where $1$ is identified with $(1,0)\in \R^2\simeq \C$); therefore, $\jaco(u)\neq \jaco(u+1)$ provided that $\partial_\tau (\sin\f)\neq 0$ at some point on $\partial \Omega$.}
Therefore, when estimating $\jac(u)$ for $u\in H^1(\Om; \R^2)$ on a bounded domain $\Omega$, we may replace $u$ by 
\be
\label{av_u}
\tilde u:=u-\frac1{|\Om|}\int_{\Om} u.\ee

We start by recalling the following stability inequality of the interior Jacobian\footnote{For the case of $BV$ maps, we refer the reader to the paper \cite{Ig_JacBV}.} (see e.g. Brezis-Nguyen \cite{BreNgu}) that represents a weaker form of Theorem \ref{thm:teorem}:

\begin{pro}
\label{pro:jac_int}
Let $\Omega\subset \R^2$ be a Lipschitz bounded domain and let $u, v\in H^1(\Omega; \R^2)$. Then
$$\|\jac(u)-\jac(v)\|_{(W_0^{1, \infty}(\Om))^*}\leq \frac 1 2 \|u-v\|_{L^2(\Om)}\big(\|\nabla u\|_{L^2(\Om)}+\|\nabla v\|_{L^2(\Om)} \big).$$
The above estimate can be improved by using $\|\tilde u-\tilde v\|_{L^2(\Om)}=\min_{a\in \R^2}  \|u-v-a\|_{L^2(\Om)}$ defined in \eqref{av_u} instead of $\|u-v\|_{L^2(\Om)}$.
\end{pro}

\proof{} First assume that $u$ and $v$ are smooth maps in $\Om$. Note that
\be
\label{ident12}
2\jac(u)=\nabla\times (u\times \nabla u) \quad \textrm{in}\quad \Omega.\ee
If $\zeta \in C^\infty_c(\Om)$, then integration by parts yields:
\begin{align}
\nonumber
2\bigg|\int_{\Om} \big(\jac(u)&-\jac(v)\big)\zeta \, dx\bigg|=\bigg|\int_{\Om} (u\times \nabla u-v\times \nabla v)\cdot \nabla^\perp \zeta\, dx\bigg|\\
\label{pg7}
&= \bigg|\int_{\Om} \bigg\{(u-v)\times (\nabla u+\nabla v)\cdot \nabla^\perp \zeta+\nabla(v\times u)\cdot \nabla^\perp \zeta\bigg\}\, dx\bigg|\\
\nonumber
&\leq \|u-v\|_{L^2(\Om)} (\|\nabla u\|_{L^2(\Om)}+\|\nabla v\|_{L^2(\Om)}) \|\nabla \zeta\|_{L^\infty(\Om)}.
\end{align}
The general case follows by a density argument: every test function $\zeta\in W^{1, \infty}(\Om)$ with $\zeta=0$ on $\partial \Om$ is approximated in $W^{1,1}(\Om)$ by $\zeta_n\in C^\infty_c(\Om)$ such that
$\|\nabla \zeta_n\|_{L^\infty(\Om)}\leq \|\nabla \zeta\|_{L^\infty(\Om)}$ (in particular $(\zeta_n)_n$ is uniformly bounded in $L^\infty(\Om)$), while the maps $u, v\in H^1(\Om; \R^2)$ are approximated in $H^1(\Om)$ by smooth maps $u_n$ and $v_n$ in $\Omega$ implying in particular that $\jac(u_n)\to \jac(u)$ and $\jac(v_n)\to \jac(v)$ in $L^1(\Om)$. Finally, passing at the limit $n\to \infty$ in the above inequality for $(u_n, v_n, \zeta_n)$, the conclusion is proved.
Note that for $u\in H^1$, \eqref{ident12} holds true in the distribution sense by the same density argument since $u\times \nabla u\in L^1$ (so, $\nabla\times (u\times \nabla u)\in (W_0^{1, \infty}(\Om))^*$) and $\jac(u)\in L^1$. The last statement of 
Proposition~\ref{pro:jac_int} follows from the  invariance of $\jac(\cdot)$ under addition of a fixed vector $a\in \R^2$.
\qed

\bigskip

In order to prove Theorem \ref{thm:teorem}, we need to investigate the stability properties of the boundary Jacobian defined in \eqref{def_jac_bd}. 
The following lemma proves that the boundary Jacobian is indeed a quantity living on the boundary $\partial \Om$ and we obtain a stability inequality for the boundary Jacobian in the strong ${H}^{1/2}(\partial \Om)$-topology, i.e., endowed by the norm
$\|u\|_{H^{1/2}(\partial \Om)}:=\|u\|_{L^2(\partial \Om)}+
\|u\|_{\dot{H}^{1/2}(\partial \Om)}$ with
$$\|u\|^2_{\dot{H}^{1/2}(\partial \Om)}:=\int_{\partial \Om} \int_{\partial \Om} \frac{|u(x)-u(y)|^2}{|x-y|^2}\, dxdy.$$ 

\begin{pro}
\label{pro:jacbd}
Let $\Om\subset \R^2$ be a Lipschitz bounded domain. Then for every $u\in H^1(\Om;\RR^2)$, the boundary Jacobian $\jacbd(u)$ of $u$ defined in \eqref{def_jac_bd} can be identified with the following linear functional acting on Lipschitz functions $W^{1, \infty}(\partial \Om)$ at the boundary $\partial \Om$:
\be
\label{forma1}
\zeta\in W^{1, \infty}(\partial \Om)\mapsto -(\zeta u\times \partial_\tau u)_{H^{1/2}(\dOm), H^{-1/2}(\dOm)}\ee
where the right hand side is interpreted as a dual (cross) product\footnote{Using a Lipschitz arc-length parametrisation $\{\gamma(\theta)\}_{\theta\in \Ss^1}$ of $\partial \Omega$, the RHS of \eqref{forma1}
becomes  up to sign $$\big((\zeta u)\times \partial_\tau u\big)_{H^{1/2}(\dOm), H^{-1/2}(\dOm)}
%=\int_{\Ss^1} (\tilde \zeta \tilde u)\times \partial_\theta \tilde u
=\big(\tilde \zeta \tilde u \times  \partial_\theta \tilde u\big)_{H^{1/2}(\Ss^1), H^{-1/2}(\Ss^1)}$$ where $\tilde u(\theta)=u(\gamma(\theta))\in H^{1/2}(\Ss^1)$ and $\tilde \zeta(\theta)=\zeta(\gamma(\theta))\in W^{1, \infty}(\Ss^1)$.
} between $\partial_\tau u\in {H}^{-1/2}(\partial \Om)=({H}^{1/2}(\partial \Om))^*$ and $\zeta u\in {H}^{1/2}(\partial \Om)$. Moreover, for every $u, v\in H^1(\Om, \RR^2)$ and 
$\zeta\in W^{1, \infty}(\Om)$, we have
\be \label{stab_ineq_bdry}
\left|\left<(\jacbd(u)-\jacbd(v)),\zeta\right>\right|\leq C \|u-v\|_{H^{1/2}(\partial \Om)}(\|u\|_{{H}^{1/2}(\partial \Om)}+
\|v\|_{{H}^{1/2}(\partial \Om)})\|\zeta\|_{W^{1, \infty}(\partial \Om)}, \ee
where $C>0$ is a constant depending only on $\Omega$ and $\|\zeta\|_{W^{1,\infty}(\partial \Om)}:=\|\zeta\|_{L^\infty(\partial \Om)}+\|\partial_\tau \zeta\|_{L^\infty(\partial\Om)}$.
\end{pro}

\proof{} 
First, we prove that for any $u\in H^1(\Om;\RR^2)$ the linear functional \eqref{forma1} is continuous on $W^{1, \infty}(\partial \Om)$ endowed with the norm $\|\cdot\|_{W^{1,\infty}(\partial \Om)}$. Indeed, we have for $\zeta\in W^{1, \infty}(\partial \Om)$: 
$$\bigg|\big(\zeta u\times \partial_\tau u \big)_{H^{1/2}(\dOm), H^{-1/2}(\dOm)}\bigg|\leq \|\zeta u \|_{{H}^{1/2}(\partial \Om)}  
\|\partial_\tau u \|_{{H}^{-1/2}(\partial \Om)}\leq C \|u \|^2_{{H}^{1/2}(\partial \Om)}\|\zeta\|_{W^{1,\infty}(\partial \Om)}$$
because $\|\partial_\tau u \|_{{H}^{-1/2}(\partial \Om)}\leq \| u \|_{{H}^{1/2}(\partial \Om)}$, 
$\|\zeta u \|_{L^2(\partial \Om)}\leq \|\zeta\|_{L^\infty(\partial \Om)} \|u\|_{L^2(\partial \Om)}$ and
\begin{align*}
\|\zeta u \|^2_{\dot {H}^{1/2}(\partial \Om)}&=\int_{\partial \Om} \int_{\partial \Om} \frac{|\zeta(x)u(x)-\zeta(y)u(y)|^2}{|x-y|^2}\, dxdy \\
&\leq C\|\zeta\|^2_{L^\infty(\partial \Om)}\|u \|^2_{\dot {H}^{1/2}(\partial \Om)}+C\|\partial_\tau \zeta\|^2_{L^\infty(\partial \Om)}\|u \|^2_{L^{2}(\partial \Om)}\\
&\leq C \|u \|^2_{{H}^{1/2}(\partial \Om)}\|\zeta\|^2_{W^{1,\infty}(\partial \Om)}
\end{align*}
where $C>0$ depends only on the geometry of $\Om$. Now let us check that the boundary Jacobian $\jacbd(u)$ coincides with \eqref{forma1} for a map $u\in H^1(\Om; \RR^2)$ and a test function 
$\zeta\in W^{1, \infty}(\Om)$. Indeed, if $u\in C^2(\bar\Omega)$, then integration by parts yields:
\begin{align*}
-\left<\jacbd(u),\zeta\right>
\stackrel{\eqref{def_jac_bd}}{=}\int_{\Om} u\times \nabla u\cdot \nabla^\perp\zeta \, dx+2\int_{\Om} \jac(u)\zeta\, dx
\stackrel{\eqref{ident12}}=\int_{\partial \Om} \zeta u\times \partial_\tau u \,d\h^1.\end{align*}
Since the trace operator is continuous from $H^1(\Om; \RR^2)$ to $H^{1/2}(\dOm; \RR^2)$ as well as the operator mapping
$u\in H^{1/2}(\dOm; \RR^2)\mapsto \partial_\tau u\in H^{-1/2}(\dOm;\RR^2)$, by the density of $C^2(\bar\Omega; \R^2)$ maps into $H^1(\Om; \RR^2)$, we conclude that the last identity also holds for general $u\in H^1(\Om;\RR^2)$ within the duality $(H^{1/2}(\dOm), H^{-1/2}(\dOm))$.
Finally, we prove \eqref{stab_ineq_bdry} for $u, v\in H^1(\Om;\RR^2)$ and 
$\zeta\in W^{1, \infty}(\Om)$. Indeed, using the same estimates as above, there exists a constant $C>0$ depending on $\Om$ such that
\begin{align*}
\left|\left<(\jacbd(u)-\jacbd(v)),\zeta\right>\right|
&=\left|\big(\zeta u\times \partial_\tau u-\zeta v\times \partial_\tau v\big)_{H^{1/2}(\dOm), H^{-1/2}(\dOm)} \right|\\
&= \bigg| \big(\underbrace{\zeta (u-v)\times \partial_\tau(u+v)}_{:=I} \, 
\underbrace{ -\zeta u \times \partial_\tau v+\zeta v\times \partial_\tau u}_{:=II}\big)_{H^{1/2}(\dOm), H^{-1/2}(\dOm)}\bigg|
\end{align*}
with
\begin{align*}
|\, I\, |&\leq \|\zeta (u-v) \|_{ {H}^{1/2}(\partial \Om)}  \|u+v \|_{{H}^{1/2}(\partial \Om)}\\
&\leq 
C \|u-v \|_{{H}^{1/2}(\partial \Om)}\bigg( \|u\|_{{H}^{1/2}(\partial \Om)}+\|v \|_{{H}^{1/2}(\partial \Om)} \bigg)\|\zeta\|_{W^{1,\infty}(\partial \Om)}
\end{align*}
and $II$ is interpreted as a duality between $\partial_\tau (v\times u)\in (W^{1, \infty}(\partial \Om))^*$ and $\zeta \in W^{1, \infty}(\partial \Om)$ which combined with $v\times u=v\times (u-v)\in L^1(\partial \Om)$ leads to
\begin{align*}
|\, II\, |=\left|\int_{\partial \Om} v\times u \partial_\tau \zeta\, d\h^1\right|&\leq \| u-v \|_{L^2(\partial \Om)}  \|v \|_{L^2(\partial \Om)}\|\zeta\|_{W^{1,\infty}(\partial \Om)}\\
&\leq C \|u-v \|_{{H}^{1/2}(\partial \Om)}\|v\|_{{H}^{1/2}(\partial \Om)}\|\zeta\|_{W^{1,\infty}(\partial \Om)}.
\end{align*}
Summing up, we conclude with \eqref{stab_ineq_bdry} which implies in particular for every $\zeta\in W^{1, \infty}(\Om)$:
\be
\label{weak_jac_stab}
\left|\left<\jacbd(u)-\jacbd(v), \zeta\right> \right|\leq C \|u-v\|_{H^{1}(\Om)}(\|u\|_{{H}^{1}(\Om)}+
\|v\|_{{H}^{1}(\Om)})\|\zeta\|_{W^{1, \infty}(\Om)}.\ee
\qed

\begin{rem}
\begin{enumerate}
\item[i)]
Note that the above inequality \eqref{weak_jac_stab} is weaker than the estimate in Theorem~\ref{thm:teorem} because it represents a stability
inequality for the boundary Jacobian in the strong $H^1(\Om)$-topology, while in Theorem~\ref{thm:teorem} only $L^2(\Om)$ closeness is required 
together with a slight control of the $\dot{H}^1(\Om)$-seminorm that may blow-up.

\item[ii)] For $u\in H^1(\Om; \R^2)$, while the interior Jacobian $\jac(u)$ is a measure in $\Om$, the global Jacobian is not in general a 
measure on $\bar \Om$ because $\jacbd$ is not in general a measure on $\partial \Om$. Indeed, if $|u|=1$ in a smooth simply connected domain $\Om$, then $\jac(u)=0$ in $\Om$
(because $\partial_{x_1} u$ and $\partial_{x_2} u$ are parallel vectors, both being orthogonal to $u$) and by the 
Bethuel-Zheng theorem in \cite{BethuelZheng:1988a}, we know that $u=e^{i\f}$ with a lifting $\f\in H^1(\Om)$ so that $\f\in H^{1/2}(\partial \Om)$. Then it follows by Proposition \ref{pro:jacbd} that for every $\zeta\in W^{1, \infty}(\Om)$, $\left<\jacbd(u), \zeta\right>=-(\zeta\partial_{\tau} \f)_{H^{1/2}(\dOm), H^{-1/2}(\dOm)} $ where $\partial_{\tau} \f$ belongs to $H^{-1/2}(\partial \Om)$ which clearly can be chosen not to be a measure on $\partial \Om$.

\end{enumerate}
\end{rem}

\bigskip

\noindent {\bf Proof of Theorem \ref{thm:teorem}}. Since $\Om$ is $C^{1,1}$, there exists $r_1:=r_1(\Omega)>0$ such that every point $x\in \Omega$ with $\dist(x, \partial \Om)<r_1$ has a unique orthogonal projection on the boundary $\partial \Om$, i.e., the crossing of two normal directions on $\partial \Om$ in the interior of $\Om$ happens at a distance larger than $r_1$ from the boundary. 

Assume for the moment that $u, v$ are smooth maps in $\bar \Omega$. Note that the inequality is trivial if $u$ and $v$ are equal or if they are both constant maps. Therefore, in the following, we can assume that 
$$\delta=\frac{\|u-v\|_{L^2(\Om)}}{\|\nabla u\|_{L^2(\Om)}+\|\nabla v\|_{L^2(\Om)}}
\in (0, \infty).$$ 
Let $\zeta$ be smooth in $\bar \Om$. In the following, {\bf we denote by $C>0$ a constant 
depending only on the geometry of $\Om$ that can change from line to line.}

\medskip

\noindent {\bf Case 1.} {\it Suppose that $\delta\geq r_1/4$.} In this case, we have 
\begin{align*}
&\bigg|\left<\jaco(u)-\jaco(v),\zeta\right> \bigg|  =\bigg|\int_{\Om} (u\times \nabla u-v\times \nabla v)
\cdot \nabla^\perp \zeta\, dx\bigg|\\
&= \bigg|\int_{\Om} \bigg\{(u-v)\times \nabla u \cdot \nabla^\perp \zeta+v \times (\nabla u-\nabla v)
\cdot \nabla^\perp \zeta\bigg\}\, dx\bigg|\\
&\leq \|v-u\|_{L^2(\Om)} \|\nabla u\|_{L^2(\Om)} \|\nabla \zeta\|_{L^\infty(\Om)}+
\|v\|_{L^2(\Om)}  (\|\nabla v\|_{L^2(\Om)}+ \|\nabla u\|_{L^2(\Om)})\|\nabla \zeta\|_{L^\infty(\Om)}.
\end{align*}
The conclusion follows by 
$$\|v\|_{L^2(\Om)}\leq \h^2(\Om)^{1/2}\leq C r_1^{1/2}\leq C\delta^{1/2},$$
where we used the hypothesis $|v|\leq 1$ in $\Om$ and the assumption $\delta\geq r_1/4$.

\medskip

\noindent {\bf Case 2.} {\it Suppose that $\delta\leq r_1/4$.} In this case, we denote by $$\Om_R=\{x\in\Omega\, :\, \dist(x,\partial \Om)<R\}$$ the region around the boundary $\partial \Om$ at a distance less than $R$. 
Then by averaging on the interval $(\delta, 2\delta)$,
the co-area formula yields the existence of some $R \in (\delta, 2\delta)$ such that:
\begin{align}
\nonumber
\delta \int_{\partial \Om_{R}\cap \Om}|v\times u|\, d\h^1=\int_\delta^{2\delta} \, dr \int_{\partial \Om_r\cap \Om} |v\times u| \, d\h^1&=
 \int_{\Om_{2\delta}\setminus \Om_{\delta}} |v\times u|\, dx\\
 \label{choiceR}
 &\leq {C}{\delta^{1/2}}  \|v-u\|_{L^2(\Om)}
 \end{align}
because $v\times u=v\times (u-v)$ and $|v|\leq 1$.
We estimate the desired quantity on $\Om\setminus \Om_{R}$:
\begin{align*}
I:=&\bigg|\int_{\Om\setminus \Om_{R}} (u\times \nabla u-v\times \nabla v)\cdot \nabla^\perp \zeta\, dx\bigg|\\
&\stackrel{\eqref{pg7}}{\leq}  \|\nabla \zeta\|_{L^\infty(\Om)}\int_{\Om\setminus \Om_{R}} |u-v|(|\nabla u|+|\nabla v|)\, dx+
\bigg| \int_{\Om\setminus \Om_{R}}\nabla(v\times u)\cdot \nabla^\perp \zeta\, dx\bigg|\end{align*}
where the integration by parts leads to
\begin{align*}
\bigg|\int_{\Om \setminus \Om_{R}} \nabla(v\times u)\cdot \nabla^\perp \zeta\, dx\bigg|
&=\bigg|\int_{\partial \Om_{R}\cap \Om} v\times u\, \partial_\tau \zeta\, d\h^1\bigg|\\
&\leq \|\nabla \zeta\|_{L^\infty(\Om)} \int_{\partial \Om_{R}\cap \Om}|v\times u|\, d\h^1\\
& \stackrel{\eqref{choiceR}}{\leq} C \|v-u\|^{1/2}_{L^2(\Om)}(\|\nabla v\|_{L^2(\Om)}+\|\nabla u\|_{L^2(\Om)})^{1/2}  \|\nabla \zeta\|_{L^\infty(\Om)}.
\end{align*}
Next we estimate the desired quantity on $\Om_{R}$:
\begin{align*}
II:=&\bigg|\int_{\Om_{R}} (u\times \nabla u-v\times \nabla v)\cdot \nabla^\perp \zeta\, dx\bigg|\\
&=\bigg|\int_{\Om_{R}} \bigg\{(u-v)\times \nabla u \cdot \nabla^\perp \zeta+v \times (\nabla u-\nabla v)\cdot \nabla^\perp \zeta\bigg\}\, dx\bigg|\\
&\leq \|\nabla \zeta\|_{L^\infty(\Om)}\int_{\Om_{R}} |v-u| |\nabla u| \, dx+\h^2(\Om_R)^{1/2} (\|\nabla v\|_{L^2(\Om)}+
 \|\nabla u\|_{L^2(\Om)})\|\nabla \zeta\|_{L^\infty(\Om)},
\end{align*}
with $\h^2(\Om_R)^{1/2}\leq C \delta^{1/2}$. Adding $I$ and $II$ we obtain the desired inequality. By a standard density argument (as in the proof of Proposition~\ref{pro:jac_int}), the inequality holds for general $H^1$-maps $u$ and $v$ and general Lipschitz test function $\zeta$.
\qed

\bigskip

As a consequence of Theorem \ref{thm:teorem} and Proposition \ref{pro:jac_int}, we have the following stability result for the global and interior Jacobian:

\begin{cor}
\label{coro} 
Let $\Omega\subset \R^2$ be a $C^{1,1}$ bounded domain and let $u_\eps, v_\eps:\Omega\to \R^2$ be two sequences / families of $H^1$-maps such that $|v_\eps|=1$ in $\Om$ and 
\be
\label{cond_control}
\|u_\eps-v_\eps\|_{L^2(\Om)}\big(\|\nabla u_\eps\|_{L^2(\Om)}+\|\nabla v_\eps\|_{L^2(\Om)} \big)\to 0 \quad \textrm{as}\quad \eps\to 0.
\ee
Then
\begin{enumerate}
\item[1.] (Stability of the interior Jacobian) We have $\jac(v_\eps)=0$ and  $\|\jac(u_\eps)\|_{(W_0^{1, \infty}(B_1))^*}\stackrel{\eps\to 0}{\to} 0;$
\item[2.] (Stability of the global Jacobian) We have
$\|\jaco(u_\eps)-\jaco(v_\eps)\|=\|\jaco(u_\eps)-\jacbd(v_\eps)\|\stackrel{\eps\to 0}{\to} 0.$
\end{enumerate}
\end{cor}

\bigskip

Let us show now that the boundary Jacobian $\jacbd$ is \emph{not} stable under the condition \eqref{cond_control} (recall that $\jacbd$ is stable in the strong $H^1(\Om)$-topology, see \eqref{weak_jac_stab}).

\begin{pro}
Let $P=(0,-\frac34)$, $\Om=B_{1/4}(P)$ be the disk of center $P$ and radius $1/4$. Then 
for every small $\eps>0$, there exists a map $u_\eps\in H^1(\Om; \R^2)$ such that \eqref{cond_control} holds
for $v_\eps=1$ in $\Om$, but
$$\left<\jacbd(u_\eps)-\jacbd(v_\eps),1\right>\nrightarrow 0 \quad \textrm{as}\quad \eps\to 0.$$
In particular, $\|\jacbd(u_\eps)-\jacbd(v_\eps)\|\nrightarrow 0$ as $\eps\to 0$.
\end{pro}

\proof{} Set $P_\pm=(0, \pm \frac 1 2)\in \R^2$ and $B^2$ be the unit disk of $\R^2$. Note that $P_-\in \partial \Omega$.

\medskip

\noindent {\bf Step 1}. {\it Construction of a function $U:B^2\to \overline{B^2}$}. First, we set 
$$G=B^2\setminus \bigg(\overline{B_{1/4}(P_-)}\cup \overline{B_{1/4}(P_+)}\bigg)$$ and we define $U:\partial G\to \Ss^1$ as follows: $U=1$ on $\partial B^2$, $U(P_\pm+\frac14e^{i\theta})=e^{\pm i\theta}$ for $\theta\in [0, 2\pi)$. Note that the topological degree of this smooth boundary data $U$ over $\partial G$ is zero. Therefore, we can smoothly 
extend $U:\bar G\to \Ss^1$ to the closure of the set $G$ (see  Bethuel-Brezis-H\'elein \cite{BethuelBrezisHelein:1994a} or Struwe \cite{Struwe:1994a}). 
Finally, we extend $U$ to the whole disk $B^2$ by setting $U(P_\pm+re^{i\theta})=4re^{\pm i\theta}$ for $0\leq r\le \frac14$ and $\theta\in [0, 2\pi)$. Then $U$ is continuous, $U\in H^1(B^2; \R^2)$ and $U$ has degree $\pm 1$ on the circles $\partial B_{\frac14}(P_\pm)$.

\medskip

\noindent {\bf Step 2}. {\it Construction of $u_\eps$ and $v_\eps$ on $\Om$.} For every $0<\eps<\frac14$,
set $v_\eps\equiv 1$ in $\Om$ and $u_\eps:\Om\to {\overline{B^2}}$ is defined as follows: 
$u_\eps(x)=U(\frac{x-P_-}{\eps})$ in $B_\eps(P_-)\cap \Om$ 
and $u_\eps=1$ in $\Om\setminus B_\eps(P_-)$.  In other words, $u_\eps$ has one interior vortex point going to the boundary $\dOm$ as $\eps\to 0$.
Then
\[
\int_{\Om}\big(|\nabla u_\eps|^2+|\nabla v_\eps|^2\big)\,dx \leq \int_{B^2} |\nabla U|^2\,dx.
\]
Since $\|u_\eps\|_{L^\infty}=\|v_\eps\|_{L^\infty}=1$ and
$\mathcal{L}^2(\{u_\eps\neq v_\eps\}) \le \pi\eps^2$, we deduce that  
\[
\int_{\Om}|u_\eps-v_\eps|^2\,dx \le C\eps^2;
\]
thus, \eqref{cond_control} holds in $\Om$ as $\eps\to 0$. We finally calculate the boundary Jacobians in the disk 
$\Om$ with boundary $\partial \Om$.
Clearly $\jacbd(v_\eps)=0$ and by Proposition \ref{pro:jacbd},
\[
\left< \jacbd(u_\eps), 1\right> = -\int_{\partial \Omega} u_\eps \times \partial_\tau u_\eps\,  d\h^1. 
\]
For $\eps$ sufficiently small, we have $|u_\eps|=1$ on $\partial \Om$, so we obtain 
\[
\left< \jacbd(u_\eps), 1\right> = -2\pi  \degr (u_\eps; \partial\Om)   = 2\pi,
\]
which clearly does not tend to zero as $\eps\to 0$. In particular, $\|\jacbd(u_\eps)-\jacbd(v_\eps)\|_{(W^{1, \infty}(\Om))^*}\nrightarrow 0$ as $\eps\to 0$ (note that 
$\|\jacbd(u_\eps)-\jacbd(v_\eps)\|_{(W^{1, \infty}(\Om))^*}<\infty$ by \eqref{weak_jac_stab}),
while $\|\jacbd(u_\eps)-\jacbd(v_\eps)\|=\infty$ { due to \eqref{hom}}. \qed

\section{Approximation by $\Ss^1$-valued  maps}
\label{sec:approx}
In this section we show that maps $u:\Om\to\R^2$ with energy  of order  $E_{\eps, \eta}(u)\leq C |\log \eps|$ can be approximated by suitable $\Ss^1$-valued maps $U:\Om\to\Ss^1$ in the regime $|\log\eps|\ll |\log\eta|$. The approximation is realised such that $u$ and $U$ are close energetically, and also in $L^2(\Om)$ and in $L^2(\partial\Om)$, and such that their global Jacobians
are close to each other. This is an essential step in the reduction of our model to the study of a simpler problem for $\Ss^1$-valued maps. Our result is based on some ideas introduced by Ignat-Otto \cite{IgnatOtto:2011a} (see also C\^ote-Ignat-Miot \cite{Cote:2014aa}) where the approximation argument was done locally; here the improvement consists in developing a global analysis of the configurations $u$, in particular at the boundary $\partial \Omega$.
  
\medskip

\noindent {\bf Notation}: If $G\subset \Omega$ and $u:\Omega\to \R^2$, we denote
$$\eee(u; G)=\int_G \big(|\nabla u|^2 + \frac{1}{\eta^2} (1-|u|^2)^2\big) \, dx + \frac{1}{2\pi {\eps}} \int_{\bar G\cap \partial\Omega} (u\cdot \nu)^2 \, d\h^1.$$

\begin{thm}
\label{lem_approx}
Let $\beta\in (\frac12,1)$ , $C>0$ and $\Om\subset \R^2$ be a simply connected $C^{1,1}$ bounded domain. 
We consider the {sequence / family}
$\e>0$ and $\eta=\eta(\e)>0$ satisfying $\eta(\eps)\to 0$ and $|\log\e|\ll |\log\eta|$ as $\eps\to 0$. Then there exist $\eps_0, c_0, \tilde C>0$ depending only on $\beta, C$ and $\Omega$ and { $0<\tilde \beta<\frac{1-\beta}6$} such that for  $\eps\leq \eps_0$ {in the sequence / family} and
every $u=u_\eps:\Om\to \RR^2$ 
with $\eee(u)\le C|\log\eps|$, 
we can construct a unit-length map $U=U_\eps:\Om\to \Ss^1$ 
such that \footnote{ In the following, $\lesssim$ denotes an upper bound with a constant depending only on $\beta, C$ and $\Omega$.}
 \be
\label{cond-approx}  \int_{\Om} |U-u|^2\, dx\lesssim
\eta^{2\beta} \eee(u), \quad \int_{\Om} (|\nabla U|^2 +|\nabla u|^2)\, dx\lesssim\eee(u),\ee
\be \label{aprox_bdry} \int_{\partial \Om} |U -u|^2\, d\h^1\lesssim
\eta^{\beta} \eee(u)\ee
and 
\be
\label{global_est}
\eee(U)\leq E_{ \eps, c_0\eta}(u) +\tilde{C}\eta^{ \tilde \beta }\bigg(E_{ \eps, c_0\eta}(u) +\sqrt{E_{ \eps, c_0\eta}(u)}\bigg).
\ee
As consequence, 
{for every $p\geq 1$, $\|U-u\|_{L^p(\dOm)}\to 0$, $\|U-u\|_{L^p(\Om)}\to 0$ as $\eps\to 0$, }
$$\|\jac(u)\|_{(W_0^{1,\infty})^*(\Om)}\lesssim \eta^\beta  \eee(u)\quad \textrm{ and }\quad
\|\jaco(U )-\jaco(u)\|_{(\mathrm{Lip}(\Omega))^*}\lesssim \sqrt{\eta^\beta \eee(u)}.$$
The map $U$ also satisfies the following local estimates: for any open set $G\subset\Omega$ { independent of $\eps$},
there exists a constant $\tilde{C}_G>0$ such that
\be
\label{eq:localapprox}
\eee(U; G_\eta)\leq E_{ \eps, c_0\eta}(u; G) +\tilde{C}_G\eta^{ \tilde \beta }\bigg(E_{ \eps, c_0\eta}(u; G) +\sqrt{E_{ \eps, c_0\eta}(u; G)}\bigg)
\ee
where \[
G_\eta=\{ x \in G : \dist(x, { \partial G\cap \Omega}) >3\eta^\beta\}.
\]
 
 \end{thm}

\proof{} We start by proving the result in the case of the unit disk $\Om=B^2$ and then we treat the general case of a simply connected $C^{1,1}$ domain $\Omega$.

\medskip

\noindent {\bf Step 1}. {\it Construction of a polar squared grid $\cal R$ in $B^2$.} We use the polar coordinates $(r,\theta)\in(0,1)\times [0,2\pi)$ corresponding to $x=(x_1,x_2)\in B^2$. For each (radial) shift
$R\in (0, \eta^{\beta})$, write
$$V_R:=\{x\in B^2\, :\, r=|x|\in (\eta^\beta, 1), \, r\equiv R \,\, (\hspace{-0.3cm} \mod
\eta^{\beta})\}$$ for the net of concentric circles at a distance
$\eta^{\beta}$ in $B^2$. By the mean value theorem, 
there exists $R\in (0,\eta^\beta)$
such that $$\int_{V_R} \den(u)\, d\h^1\leq \frac{1}{\eta^\beta} \int_{B^2} \den(u)\,
dx,$$
where $\den(\cdot)$ is the Ginzburg-Landau energy density:
\be
\label{ginlan}
\den(u) = |\nabla u|^2 + \frac1{\eta^2}(1-|u|^2)^2.
\ee 
If one repeats the above argument for the net of radii
lines at an angular distance $\eta^{\beta}$ in $B^2$, we obtain for some angle $\Theta\in (0, \eta^\beta)$
a net  $$\tilde V_\Theta:=\{x\in B^2\, :\, \theta=\arg x\in (\eta^\beta, 2\pi), \, \theta\equiv \Theta \,\, (\hspace{-0.3cm}\mod
\eta^{\beta})\}$$ such that
$$\int_{\tilde V_\Theta} { \den(u(x))\, d\h^1(x)}=\int_{\{r<1, \theta\equiv \Theta\}} {\den(u(r, \theta))\, rdr}\leq \frac{1}{\eta^\beta} \int_{B^2} \den(u)\,
dx.$$ 
Therefore, we obtain a polar squared grid
${\cal R}=V_R \cup \tilde V_\Theta$ of size {at most} $\eta^\beta$ such that \be
\label{condR}
 \int_{{\cal R}} \den(u)\,
d\h^1 \leq \frac{2}{\eta^\beta} \int_{B^2} \den(u)\,
dx\lesssim  \frac{\eee(u)}{\eta^\beta}.
\ee
We regroup the cells of $\cal R$ in order that each new cell has approximatively the same area of order $\sim\eta^{2\beta}$: the first new cell has the interior given by the disk $B(0,R+\eta^\beta)$ (by regrouping all the sectors of ${\cal R}$ of radius less than $R+\eta^\beta$ and containing the origin). Then for each annulus of $\cal R$ of the form $B(0,R+(k+1)\eta^\beta)\setminus B(0,R+k\eta^\beta)$ with  $k\geq 1$, we regroup the neighbouring cells of the angular sectors $(\Theta+j\eta^\beta, \Theta+(j+1)\eta^\beta)$ ($j\geq 1$) so that the length of the angular arc gets of order $\sim \eta^\beta$ and their area become of order $\sim \eta^{2\beta}$. (In the annuli close to the origin, many cells are regrouped, while in the annuli far away from the origin, no regrouping is needed). Therefore, from now on, we can assume that all cells of $\cal R$ (excepting the first one $B(0,R+\eta^\beta)$) are rather identical (all the four sides of the cell having the length of order $\sim\eta^\beta$). 
For any cell ${\cal C}\subset {\cal R}$ (which is one dimensional as a union of straight and circular segments) 
we denote by 
 $\inte({\cal C})$ the $2D$ region bounded by $\cal C$ and let
 $$\inte({\cal R})=\cup_{\cal C\subset \cal R} \inte(\cal C).$$ 
Therefore, we have that the closure $\overline{\inte({\cal R})}$ of $\inte({\cal R})$
is a disk strictly included in $B^2$ at a distance less than $\eta^\beta$ from the boundary $\partial B^2$.
{The cells we have constructed all satisfy uniform conditions on their geometry so we can apply Proposition~\ref{thm_main_GL} with uniform constants.}

\medskip

\noindent {\bf Step { 2}}. {\it An approximating $\Ss^1$-valued map $\hat U$ for $u$ inside $\inte({\cal R})$}. 
In the interior $\inte({\cal C})$ of a polar squared cell ${\cal C}$ of $\cal R$ having each side of length $\sim\eta^\beta$, we
define
$w=w_\eps \in H^1(\inte({\cal C}),\R^2)$ (depending on $\eps$ through $\eta=\eta(\eps)$) be  a minimiser of 
\be
\label{min_u_gl}
\min_{w=u\, \textrm{on}\, {\cal C}} \, \,  \int_{\inte({\cal C})} \den(w)\, dx.\ee Putting together all the cells, $w$ is now defined in the whole $\inte(\cal R)$. 
We apply Proposition \ref{thm_main_GL} below (for $\kappa:=C|\log \eps|\ll |\log \eta|$): Since \eqref{condR} holds (in particular, \eqref{asum1} holds for $\den$ on the domain ${\cal D}_\eta$ and $\partial {\cal D}_\eta$), we have the existence of $0<\tilde \beta<\frac{1-\beta}6$ such that
\be
\label{delta}
\sup_{\inte(\cal R)}\||w|^2-1\|\lesssim\eta^{\tilde \beta}=:\delta\ll 1.
\ee
In particular, $|u|\geq 1/2$ on $\cal R$ and $u$ has vanishing degree on each cell, i.e., $\degr(u, {\cal C})=0$. 
The same conclusion holds for the central cell of interior $B(0,R+\eta^\beta)$. Therefore, we can define
$$\hat U:=\frac{w}{|w|} \quad \textrm{in} \quad \inte({\cal R}).$$
Then $|w|^2 |\nabla \hat  U|^2\leq |\nabla w|^2$ and we deduce for small $\eps>0$:
\begin{align}
\nonumber
\int_{\inte(\cal R)} |\nabla \hat U|^2\, dx  \stackrel{\eqref{delta}}{\leq} (1+2\delta)&\int_{\inte(\cal R)} |\nabla w|^2\, dx \leq (1+2\delta)\int_{\inte(\cal R)} \den(w)\, dx \\
\label{esti_harmon} &\stackrel{\eqref{min_u_gl}}{\leq} (1+2\delta)\int_{\inte(\cal R)} \den(u)\, dx \leq (1+2\delta) \int_{B^2} \den(u)\, dx.
\end{align}
For the local estimates inside a set $G\subset B^2$, we set $\hat {\cal R} $ be the union of cells
$\cal C\subset {\cal R}$ such that $\inte ({\cal C})\subset G$ and by the same notation as above, we call
$\inte({\hat{\cal R}})=\cup_{\cal C\subset \hat{\cal R}} \inte(\cal C)$. Then we have 
$\inte(\hat{\cal R})\subset G$ and
we conclude as above 
\[
\int_{\inte({\hat{\cal R}})} |\nabla \hat U|^2\, dx  \leq (1+2\delta)\int_{\inte(\hat{\cal R})} |\nabla w|^2\, dx \leq (1+2\delta) 
\int_G  \den(u)\, dx.
\]

\medskip

\noindent {\bf Step { 3}}. {\it Our approximating $\Ss^1$-valued map $U$ of $u$ in $B^2$}. 
We have defined $\hat U$ in $\inte({\cal R})\subset \subset B^2$. However, 
we have that 
 $B^2=(1+O(\eta^\beta)) \inte({\cal R}).$ For simplicity of notation, we assume  in the following that 
 \be
\label{conventi}
 \overline{B^2}=(1+\eta^\beta) \overline{\inte({\cal R})} \quad \textrm{and} \quad 
 U(\tilde x):= \hat U(x), \, \,  \tilde x=(1+\eta^\beta)x   \, \, \textrm{ for every } x\in \inte({\cal R}) \ee
and our goal is to prove that $U:B^2\to \Ss^1$ is indeed the desired approximating map of the given $u$. 
We also set $\tilde u:(1+\eta^\beta)B^2\to \R^2$ by $\tilde u(\tilde x)=u(x)$ for every $x\in B^2$. 

\medskip

\noindent {\bf Step 4}. {\it Estimate the $L^2$-norm of gradients.}
We have that
\[
\int_{B^2 } |\nabla U|^2 \, d\tilde x = \int_{\inte({\cal R})} |\nabla \hat U|^2\, dx\stackrel{\eqref{esti_harmon}}{\leq} (1+2\delta)\int_{B^2} \den(u)\, dx.
\]
Combined with $\displaystyle \|\nabla u\|_{L^2(\Omega)}^2 \leq \int_{B^2} \den(u)\, dx\leq \eee(u)$, 
the second estimate in
\eqref{cond-approx} follows.
For the local estimate, {we have
 $G_\eta \subset (1+\eta^\beta) \inte(\hat{\cal R})\subset B^2$ by convention \eqref{conventi} } so that by Step 2 it follows
 $$\|\nabla U\|^2_{L^2(G_\eta)}\leq \|\nabla \hat U\|^2_{L^2(\inte(\hat{\cal R}))} \leq (1+2\delta) 
\int_G  \den(u)\, dx.$$ 
 
\noindent {\bf Step 5}. {\it Estimate $\|\hat U-u\|_{L^2(\inte{\cal R})}$.}
By
Poincar\'e's inequality, we have for each cell ${\cal C}\subset {\cal R}$: \be \label{eve1} \int_{ \inte({\cal C})}
\bigg|{ \hat U}-\int_{{\cal C}}\hspace{-3.7mm}- \,\,\,\,\,{ \hat U}
\bigg|^2\, dx\lesssim \eta^{2\beta} \int_{ \inte({\cal C})} |\nabla
{\hat U}|^2\, dx 
\ee 
and \be
\label{eve2} \int_{ \inte({\cal C})} \bigg|u-\int_{{\cal
C}}\hspace{-3.7mm}- \,\,\,\,\,u \bigg|^2\, dx\lesssim
\eta^{2\beta} \int_{ \inte({\cal C})} |\nabla u|^2\, dx, \ee where $\int_{{\cal
C}}\hspace{-3.9mm}-=\frac1{\h^1(\cal C)}\int_{\cal C}$ is the average on the cell $\cal C$. 
As $\rho:=|u|\geq \frac12$ on $\cal R$, we can set $v=\frac{u}\rho$ on ${\cal R}$ with $|v|=1$; therefore, we have $v={\hat U}$ on ${\cal R}$ and by Jensen's inequality, we
estimate
\begin{align}
\label{eve3}\int_{ \inte({\cal C})} \bigg|\int_{ {\cal
C}}\hspace{-3.7mm}- \,\,\,\,\,(\hat U-u) \, d\h^1\bigg|^2\, dx&=\int_{\inte({\cal C})} 
\bigg|\int_{ {\cal C}}\hspace{-3.7mm}-
\,\,\,\,\,(v-\rho v) \, d\h^1 \bigg|^2\, dx\\
\nonumber&\lesssim \eta^{2\beta} \int_{ {\cal
C}}\hspace{-3.7mm}-
\,\,\,\,\,(1-\rho)^2\, d\h^1\lesssim \eta^{\beta} \int_{{\cal C}}(1-\rho^2){^2}\, d\h^1
\lesssim   \eta^{\beta+2} \int_{{\cal C}}
\den(u)\, d\h^1.
\end{align}
Summing \eqref{eve1}, \eqref{eve2} and \eqref{eve3} over all
the cells ${\cal C}$ of the grid ${\cal R}$, by \eqref{condR} and \eqref{esti_harmon}, we
obtain that $$ \int_{\inte({\cal R})} |\hat U-u|^2\, dx\lesssim
\eta^{2\beta}\int_{B^2} \den(u)\, dx. $$

\noindent{\bf Step 6}. {\it The $L^2$-estimate of $U-u$ in $B^2$}.
From \eqref{conventi} and the previous step, we clearly have that 
\[
\| U-\tilde u\|^2_{L^2(B^2)}\lesssim \eta^{2\beta}\int_{B^2} \den(u)\, dx. 
\]
Hence it remains to show that the $L^2$ norm of $u-\tilde u$ satisfies the same estimate.  
We compute
\[
\int_{B^2} |u(x)-\tilde u(x)|^2 \, dx = \int_{B^2} |u(x) - u( \frac{x}{1+\eta^\beta}) |^2 \, dx .
\]
We set $\lambda(t) = (1-t)+\frac{t}{1+\eta^\beta}$ for $t\in [0,1]$. Then $\frac{1}{1+\eta^\beta}\le \lambda(t) \le 1$,
$|\lambda'(t)| = 1-\frac1{1+\eta^\beta}=O(\eta^\beta)$ and 
\[
|u(x)-\tilde u(x)| = \left|\int_0^1 \lambda'(t) x\cdot \nabla u (\lambda(t) x) \, dt\right| 
\] 
so integrating on $B^2$, we obtain
\[
\int_{B^2} |u(x)-\tilde u(x)|^2 \, dx \lesssim \eta^{2\beta}\int_0^1 \int_{B^2} |x\cdot \nabla u(\lambda(t) x)|{^2} \, dt \, dx. 
\]
Changing variables $y=\lambda(t) x$ and using Fubini, we see that 
\[
\int_{B^2} |u(x)-\tilde u(x)|^2 \, dx \lesssim \eta^{2\beta}\int_0^1 \frac1{|\lambda(t)|{^4}}\int_{B^2} |y\cdot \nabla u(y)|{^2} \, dy \, dt\lesssim \eta^{2\beta} \int_{B^2} \den(u)\, dx
\] 
as claimed. This proves the first inequality in \eqref{cond-approx}. 

\medskip

\noindent {\bf Step 7}. {\it The $L^2$-estimate of $U-u$ at the boundary $\partial B^2$ and $\overline{G_\eta}\cap \partial B^2$.}
Let $R_0\in (0,1)$ be the largest radius such that $\partial B(0, R_0)\subset {\cal R}$.
By the convention \eqref{conventi}, we have chosen 
$${ R_0=\frac1{1+\eta^\beta}}$$ and we have defined $U$ in terms of $\hat U$. 
Since $v=\hat U$, $|v|=1$ and $u=\rho v$ on ${\cal R}$, we have
$$
\int_{\partial B(0, R_0) } |\hat U-u|^2\, d\h^1=\int_{\partial B(0, R_0) } (1-\rho)^2\, d\h^1
\leq \eta^{2} \int_{\cal R}\den(u)\, d\h^1
\stackrel{\eqref{condR}}{\lesssim} \eta^{2-\beta} \int_{B^2} \den(u)\, dx,
$$
\begin{align*}
\int_{\partial B(0, R_0) } |u(x)-u(\frac{x}{R_0})|^2\, d\h^1(x)&\stackrel{\eqref{conventi}}{=}R_0\int_{\partial B^2 } |u(R_0\tilde x)-u(\tilde{x})|^2\, d\h^1(\tilde x)\\
&\leq R_0
\int_0^{2\pi} \bigg(\int_{R_0}^1 |\partial_r u|(re^{i\theta})\, dr\bigg)^2\, d\theta\\
&\lesssim \eta^\beta \int_{B^2\setminus B(0,R_0)} |\nabla u|^2\, dx\lesssim  \eta^{\beta}\int_{B^2} \den(u)\, dx.
 \end{align*}
Combining these inequalities, we conclude
$$\int_{\partial B^2 } |u(\tilde x)- U(\tilde x)|^2\, d\h^1(\tilde x) = { (1+\eta^\beta)} \int_{\partial B(0, R_0) } |u(\frac{x}{R_0})-{\hat{U}(x)}|^2\, d\h^1(x) \lesssim \eta^{\beta}\int_{B^2} \den(u)\, dx.
$$
For the local estimate at the boundary $\overline{G_\eta}\cap \partial B^2$,  we have as before for $\eps$ small: $$\int_{G_\eta\cap \partial B(0, R_0) } |u(x)-u(\frac{x}{R_0})|^2\, d\h^1(x) \lesssim \eta^\beta \int_{G} |\nabla u|^2\, dx$$ because by the definition of $G_\eta$, we know that for every $x\in G_\eta\cap \partial B(0, R_0)$, the open segment $(x, \frac{x}{R_0})\subset G$. It remains to prove that 
$$\int_{G_\eta\cap \partial B(0, R_0) } |\hat U-u|^2\, d\h^1=\int_{G_\eta\cap \partial B(0, R_0) } (1-\rho)^2\, d\h^1 \lesssim \eta \int_{G}\den(u)\, dx.$$ Indeed, we consider the covering
$G_\eta\cap \partial B(0, R_0)\subset \cup {\cal C}\subset G$ and for each cell $\cal C$ we consider the function $w$ constructed at Step 2. For simplicity of notation, we write such a cell $\cal C$ to be the sector $(R_0-\eta^\beta, R_0)\times (\Theta, \Theta+\eta^\beta)$ in the polar coordinates. By averaging in the radial coordinates, one can find an arc ${\cal L}_*=\{r_*\}\times (\Theta, \Theta+\eta^\beta)$ with $r_*\in (R_0-\eta^\beta,R_0)$ such that   
\be
\label{eq43}
\int_{{\cal L}_*} (1-|w|)^2\, d\h^1\leq \frac1{\eta^\beta} \int_{\inte({\cal C})} (1-|w|)^2\, dx\leq \eta^{2-\beta} \int_{\inte({\cal C})}\den(w)\, dx.\ee
Then 
\begin{align*}
 \int_{{\cal C}\cap \partial B(0, R_0)} (1-|w|)^2\, d\h^1&\leq  \int_{{\cal L}_*} (1-|w|)^2\, d\h^1
 +2\int_{\inte({\cal C})} (1-|w|)|\partial_r w|\, dx\\
 &\stackrel{\eqref{eq43}}{\lesssim} \eta \int_{\inte({\cal C})}\den(w)\, dx \lesssim\eta \int_{\inte({\cal C})}\den(u)\, dx.
 \end{align*}
Summing up over cells $\cal C$ covering $G_\eta\cap \partial B(0, R_0)$, we conclude that
$$\int_{\overline{G_\eta}\cap \partial B^2 } |u(\tilde x)- U(\tilde x)|^2\, d\h^1\lesssim \eta^\beta \int_{G}\den(u)\, dx.$$

\medskip

 \noindent {\bf Step 8}. {\it Estimate of the global / interior Jacobian and $L^p$-estimates of $U-u$ in $\Om$ and $\dOm$}. The estimates of the global / interior are consequences of \eqref{cond-approx}, Theorem \ref{thm:teorem} and Proposition \ref{pro:jac_int}. For the $L^p$-estimates of $U-u$ in $\Om$, we use \eqref{cond-approx}, the Gagliardo-Nirenberg interpolation inequality for $p>2$ \footnote{For every $p\in (2, \infty)$, there exists $C>0$ such that 
 $\|f\|_{L^p}\leq C (\|\nabla f\|_{L^2}+\|f\|_{L^2})^{1-2/p} \|f\|_{L^2}^{2/p}$ for every $f\in H^1(\Om)$.} or simply, the H\"older inequality for $p\leq 2$ as well as $\eta\leq \eps$ (due to the regime \eqref{regim_1}), in particular, $\eta^\sigma |\log \eps|\to 0$ as $\eps\to 0$ for every $\sigma>0$. The same argument applies for the estimate $L^p(\dOm)$ of $U-u$.

\medskip

\noindent {\bf Step 9}. {\it Estimate of the energy of $U$ in $B^2$ and $G_\eta$}. Using that $a^2\leq b^2+2|a-b|$ for every $a\in [-1,1]$ and $b\in \R$, then Step 7 and Cauchy-Schwarz yield
\begin{align}
\nonumber \frac1{2\pi \eps}\int_{\partial B^2} (U\cdot \nu)^2\, d\h^1&\leq \frac1{2\pi \eps}\int_{\partial B^2} (u\cdot \nu)^2\, d\h^1+\frac1{\pi \eps} \int_{\partial B^2} |(U-u)\cdot \nu|\, d\h^1\\
\label{c}&\leq  \frac1{2\pi \eps}\int_{\partial B^2} (u\cdot \nu)^2\, d\h^1+\frac{c}{\eps} \|U-u\|_{L^2(\partial B^2)}\\
\nonumber&\leq  \frac1{2\pi \eps}\int_{\partial B^2} (u\cdot \nu)^2\, d\h^1+\frac{c \eta^{\beta/2}}{\eps} \sqrt{\int_{B^2}\den(u)\, dx}, 
\end{align}
for some $c>0$. Since $|\log  \eps|\ll |\log  \eta|$, we can choose $\eps_0>0$ (depending on $\beta$) such that
$\frac{\eta^{\beta/2}}{\eps} \leq \delta$ for every $\eps\leq \eps_0$ where $\delta$ is defined in \eqref{delta}. 
(Here, the assumption 
$\beta>\frac12$ is essential.) 

By Step 4, we obtain
$$\eee(U)\leq \eee(u)+\tilde{C} \delta \bigg(\sqrt{\int_{B^2}\den(u)\, dx}+\int_{B^2}\den(u)\, dx\bigg),$$
for some constant $\tilde{C}>0$. The local estimate \eqref{eq:localapprox} (with $c_0=1$) follows by the same argument, the constant $c$ in \eqref{c} depending only on the length of $\partial G\cap \partial B^2$.

\medskip

\noindent {\bf Step 10}. {\it The general case of a simply connected $C^{1,1}$ domain $\Omega$}.
By the Kellogg-Warschawski theorem (see Pommerenke \cite[Theorem 3.5]{Pommerenke:1992aa}), there exists a conformal map $\Psi\in C^{1,\alpha}{ (\bar \Omega; \bar B^2)}$ that transforms $\Omega$ and $\partial \Omega$ in $B^2$ and $\partial B^2$ respectively, for every $\alpha\in(0,1)$. Since the Jacobian $\jac(\Psi)$ is bounded above and below by some positive constants, 
the corresponding energy on $B^2$ is bounded (above and below) by $E_{\tilde \eps, \tilde \eta}$ where $\eps\sim \tilde \eps$ and $\eta\sim \tilde \eta$. Therefore, \eqref{cond-approx} and \eqref{aprox_bdry} (as well as the estimates for the interior / global Jacobian) follow immediately because the prefactor in those inequalities is not essential. However, as the prefactor is essential for the global / local estimates \eqref{global_est} and \eqref{eq:localapprox}, we note that our argument in 
Steps 1-9 is based only on the control of the Ginzburg-Landau density 
$\den$ and therefore, the estimates \eqref{global_est} and \eqref{eq:localapprox} hold true by changing $\eta$ by $\tilde{\eta}=c_0\eta$.
\qed

\bigskip

In the previous proof, we used the following global uniform estimate for solutions of the standard Ginzburg-Landau equation in a cell, 
which was obtained in \cite{IKL19} (with slightly different notation, using $\eps$ instead of $\eta$). Let ${\cal D}\subset \R^2$ be a Lipschitz bounded domain. For a sequence / family $\eta\to 0$ and $\beta\in (0,1)$, we consider ${\cal D}_\eta:=\eta^\beta {\cal D}$ a cell of size $\eta^\beta$, $\tau$ be a unit tangent vector field a.e. on $\partial {\cal D}_\eta$ and  a boundary data $g_\eta\in H^1(\partial {\cal D}_\eta; \R^2)$. For every $u\in H^1({\cal D}_\eta;\R^2)$, we recall the Ginzburg-Landau energy density
$\den(u)$ defined in \eqref{ginlan}.

\begin{pro}[\cite{IKL19}, Corollary 2] \label{thm_main_GL}
For { a sequence / family} $\eta \to 0$, let $u_\eta \in H^1({\cal D}_\eta;\R^2)$ be a minimiser of 
$$\min_{u=g_\eta\, \, \partial {\cal D}_\eta} \int_{{\cal D}_\eta} \den(u)\, dx.$$
Let $\kappa=\kappa(\eta)\ll |\log \eta|$ as $\eta\to 0$. Assume that
\be
\label{asum1}
\int_{\partial {\cal D}_\eta} |\partial_\tau g_\eta|^2+\frac1{\eta^2}(1-|g_\eta|^2)^2\, d\Hh^1\leq \frac{\kappa}{\eta^\beta}\quad \textrm{and} \quad \int_{{\cal D}_\eta} \den(u_\eta)\, dx\leq \kappa.\ee
Then there exists $0<\tilde \beta<\frac{1-\beta}6$ such that for {the {members of the} sequence / family {with}} $0<\eta\leq \eta_0$,
$$\sup_{{\cal D}_\eta}\bigg|{ |u_\eta|^2}-1\bigg|\leq C \eta^{\tilde \beta},$$ 
where $C>0$ and $\eta_0>0$ depend only on the geometry of $\cal D$. 
In particular, $\degr(g_\eta; \partial {\cal D}_\eta)=0$.
\end{pro}

\section{Second order $\Gamma$-convergence in the case of $\Ss^1$-valued maps}
\label{sec:s1val}

In this section we start with the setting of $\Ss^1$-valued maps motivated by the previous section, and perform a $\Gamma$-development at second order of $\eee$ restricted to such $\Ss^1$-valued  maps. The main benefit is seen in the following lifting argument, which simplifies the analysis and geometry 
of the problem:

\begin{lem}\label{lem:lifting} 
Let $\Omega\subset \R^2$ be a bounded, simply connected and $C^{1,1}$ regular domain.
If $u\in H^1(\Om;\Ss^1)$ then there exists a lifting $\phi\in H^1(\Om; \R)$ with $u=e^{i\phi}$ and $\phi$ is unique up to an additive constant in $2\pi \ZZ$. Furthermore, for every small $\eps>0$ and $\eta>0$,
\begin{equation}\label{eq:Geps}
\eee(u)\stackrel{\eqref{eq:introbv}}{=}E^{KS}_\eps(u)=\int_{\Omega} |\nabla \phi|^2\,dx + \frac{1}{2\pi\eps} 
\int_{\partial\Omega} \sin^2(\phi-g) \, d\h^1=:\Ge(\phi), 
\end{equation}
where $g$ is a lifting of the unit tangent vector field $\tau$ at $\partial \Om$, i.e., 
\be
\label{g}
e^{ig}=\tau =i\nu \quad \textrm{ on }\quad \partial \Om
\ee 
and $g$ is continuous except at a point of $\partial \Om$. 
\end{lem}
\proof{}
For the existence and uniqueness of the lifting $\phi$ of $u$ in $\Om$, we refer to Bethuel-Zheng \cite{BethuelZheng:1988a}. 
For the existence of $g$, we note that $\tau$ has winding number $1$ around $\dOm$ as $\Om$ is simply connected, and hence
no continuous $g:\dOm\to \R$ with $e^{ig}=\tau$ can exist. However, if $\dOm$ is $C^{1,1}$, we can choose $g$ 
to be locally Lipschitz except at one point of $\dOm$ where it jumps by $-2\pi$ (see e.g. Ignat \cite{Ig_CV} for the theory of $BV$ liftings). Clearly, the curvature $\ka$ of $\dOm$ is 
given by the absolutely continuous part of the derivative of $g$ (as a $BV$ function), i.e., 
$$\ka=(\partial_\tau g)_{ac} \quad \textrm{and} \quad \int_{\dOm} \ka \,  d\h^1 = 2\pi$$ which is in fact the Gau\ss{}-Bonnet formula for the boundary of a simply connected domain.
As $|\nabla u|=|\nabla \phi|$ in $\Om$ and $u\cdot \nu=\sin(g-\phi)$ on $\partial \Om$, the equality of $\eee(u)$ and $\Ge(\phi)$  is straightforward.

\qed

The functional $\Ge$ has been studied before: compactness and a first order $\Gamma$-convergence result
were established  by the second author in \cite{Kurzke:2006a}, while the second order lower bound was shown for in the restricted case of minimizers in 
 \cite{Kurzke:2006b}  and under a stronger a priori single multiplicity assumption in  \cite{Kurzke:2007a} (which is true for critical points, see \cite{BEK19}).
We use a different approach here that leads to new and significantly improved results and  proofs:
For the first order compactness,  unlike the proof in \cite{Kurzke:2006a}, 
 our new approach  incorporates 
ideas of Garroni-M\"uller \cite{GarroniMuller:2006a} so that it does not 
require the fairly elaborate rearrangement inequality for \emph{functions}
from Garsia-Rodemich \cite{GarsiaRodemich:1974a}, but instead
uses a much more straightforward rearrangement inequality for \emph{sets} from 
Alberti-Bouchitt\'e-Seppecher~\cite{AlbertiBouchittSeppeche:1994a}. 

For our more precise second order results, we employ a new method, adapting a co-area argument of Sandier \cite{Sandier:1998a}, see
Proposition \ref{prop:lolobd} below. We can avoid the use of a ``ball construction'' by directly working with the one-dimensional nonlocal energy (see \eqref{def:feps} below), and directly obtain some 
single multiplicity results  from the energy bounds.  A further central new step is Proposition~\ref{prop:onevortexGC}, a comparison argument inspired by 
Colliander-Jerrard \cite{ColliandJerrard:1999a} that yields the second order lower bounds
by purely energy methods. We can thus completely 
avoid the  PDE arguments used in \cite{Kurzke:2007a} or \cite{CCK19}.
We also find new strong compactness results on the level of the functions (in $\Om$) that are
in addition to the typical compactness of Jacobians for  Ginzburg-Landau theory. These results are essential to
show compactness of the magnetisation in a dimension reduction argument in our work \cite{IK_bdv}.

We now state the main compactness and $\Gamma$-convergence results for $\Ge$ defined at \eqref{eq:Geps}. The proof requires several steps and is completed at the end of this section. Recall that for our compactness results, we often label sequences with a continuous parameter $\eps$, which means that we start with a fixed sequence $\eps_k\to 0$ and then 
take further subsequences of this sequence. 

\begin{thm}\label{thm:GCforGeps}
Let $\Omega\subset \R^2$ be a bounded, simply connected and $C^{1,1}$ regular domain.

\nd {\bf 1. $L^p(\dOm)$-compactness and first order lower bound}. Let $(\phi_\eps)_\eps$ be a {sequence / family} in $H^1(\Omega; \R)$ such that 
\[
\limsup_{\eps\to 0} \frac{1}{|\log\eps|}\Ge(\phi_\eps)<\infty.
\]
Then there is a {sequence/ family} $(z_\eps)_\eps$ of integers such that 
$(\phi_\eps-\pi z_\eps)_\eps$ is bounded in $L^p(\dOm)$ for $1\le p<\infty$.
Moreover, for a subsequence, we have that 
$(\phi_\eps-\pi z_\eps)_\eps$ converges {\bf strongly} in $L^p(\dOm)$ to a limit $\phi_0$ such that 
$\phi_0-g \in BV(\dOm;\pi\ZZ)$ with $g$ given in \eqref{g} and  
$$\partial_\tau\phi_0=\ka -\pi\sum_{j=1}^N d_j \delta_{a_j}, \quad a_j\in \dOm \textrm{ distinct points}, \, d_j\in\ZZ\setminus\{0\}
\quad \textrm{with} \, \sum_{j=1}^N d_j=2$$ and $\partial_\tau \phi_\eps\to \partial_\tau \phi_0$ in $W^{-1,p}(\dOm)$ for every $1\le p<\infty$.
Furthermore, we have the following first order lower bound
$$
\liminf_{\eps\to 0}\frac{1}{|\log\eps|}\Ge(\phi_\eps) \ge |\partial_\tau \phi_0-\ka|(\dOm)= \pi \sum_{j=1}^N |d_j|.
$$
\nd {\bf 2. $W^{1,q}(\Om)$ weak compactness and second order lower bound}.
Let $(\phi_\eps)_\eps$ be a {sequence / family} in $H^1(\Omega; \R)$ satisfying the convergence at point $\bf 1.$ with the limit $\phi_0$ on $\dOm$ as $\eps\to 0$. If additionally we assume that 
\begin{equation}\label{eq:Gbetterub}
\limsup_{\eps\to 0} \left(\Ge(\phi_\eps)- \pi |\log\eps|\sum_{j=1}^N |d_j|\right)<\infty,
\end{equation}
then $d_j\in \{\pm 1\}$ for all $j=1,\dots,N$, 
$(\nabla\phi_\eps)_\eps$ converges weakly (for a subsequence) in $L^q(\Omega;\R^2)$ for any $q\in [1,2)$ to $\nabla 
\hat\phi_0$, where  $\hat\phi_0 \in W^{1,q}(\Omega)$ is an extension (not necessarily harmonic) of $\phi_0$ to $\Omega$.
 The following second order lower bound holds for the {sequence / family} $\eps\to 0$:
\begin{equation}\label{eq:Gepslowerbd}
\liminf_{\eps\to 0} \left(\Ge(\phi_\eps)- \pi N |\log\eps| \right) \ge W_\Omega(\{(a_j,d_j)\})+N\gamma_0,
\end{equation}
where $W_\Omega$ is the renormalised energy defined in \eqref{eq:renW} and $\gamma_0=\pi\log\frac{e}{4\pi}$.

\medskip

\nd {\bf 3. Upper bound construction:}
Let $\phi_0:\partial\Omega\to \R$ be such that  $\partial_\tau \phi_0 =\ka -\pi\sum_{j=1}^N d_j \delta_{a_j}$, {$d_j\in\ZZ\setminus\{0\}$} with $\sum_j d_j=2$, 
$e^{i\phi_0}\cdot \nu=0$ in $\partial\Omega\setminus\{a_1,\dots,a_N\}$.
Then {for every $\eps>0$ small,} there exists 
 $\hat\phi_\eps \in H^1(\Omega; \R)$ such that $\hat\phi_\eps\to \phi_0$ in ${L^p(\partial\Omega)}$ for every $p\in [1, \infty)$  and
  \begin{equation}\label{eq:gammaG_ub1}
\limsup_{\eps\to 0} \frac1{|\log\eps|}\Ge(\hat\phi_\eps) = \pi \sum_{j=1}^N |d_j|.
\end{equation}
 If in addition $d_j=\pm1$ for all $j$, then we have additionally
 \begin{equation}\label{eq:gammaG_ub}
\limsup_{\eps\to 0} \left(\Ge(\hat\phi_\eps) - N \pi\log\frac1\eps \right) = W(\{a_j, d_j\})+ N\gamma_0.
\end{equation}
\end{thm}

Our first steps towards the analysis of $\Ge$ are flattening the boundary and getting rid of the effect of $g$. 
For the first order in the
energy expansion, this can be done as in Alberti-Bouchitt\'e-Seppecher \cite{AlbertiBouchittSeppeche:1998a}, by locally flattening
the boundary with maps of small isometry defect, requiring only $C^1$ smoothness of the boundary.
In order to obtain slightly more precise estimates, we use a locally conformal flattening, requiring $C^{1,1}$ smoothness. We introduce the following notation for half disks and intervals centred at the origin, and use it throughout this section:
$$B_r^+=\{(x_1,x_2)\in \RR^2: |x|<r, x_2>0\}\quad \textrm{and}\quad I_r=(-r,r), \, r>0,$$ where $I_r$ is the straight part of the boundary of $B_r^+$.
We also denote by $$\R^2_+=\R\times (0, \infty).$$ The localisation lemma is proved in the following:
\begin{lem}\label{lem:conformalflat}
Let $\Omega\subset\R^2$ be a simply connected $C^{1,1}$ domain. There exist constants $c_1=c_1(\Omega)>0$ and $r_0=r_0(\Omega)\in {(0,1)}$ such that 
for any $a\in\dOm$, we can find a $C^1$ map 
$\Psi_a: \overline{B_{2r_0}^+}\to \overline{\Omega}$ with the following properties:
\begin{enumerate}
\item[(a)] $\Psi_a:{\overline{B^{+}_{r_0(1+c_1 r_0 \log\frac1{r_0})}}}\to\Psi_a({\overline{B^{+}_{r_0(1+c_1 r_0\log\frac1{r_0})}}})$ is a diffeomorphism with $\Psi_a(0)=a$;

\item[(b)] For any $\phi\in H^1(\Omega; \R)$, setting $\psi=\phi\circ \Psi_a$, we have for any $r<r_0$:
\[
\int_{B^+_{r(1-c_1 r\log\frac1{r})}} |\nabla \psi|^2\,dx \le \int_{B_r(a)\cap\Omega} |\nabla \phi|^2 \,dx\le \int_{B^+_{r(1+c_1 r\log\frac1{r})}} |\nabla \psi|^2\,dx
\]
\end{enumerate}
and 
\[
(1-c_1r\log\frac1{r})\int_{I_{r(1-c_1r\log\frac1{r})}} \sin^2 \psi \,d\h^1 \le \int_{\dOm\cap B_r(a)} \sin^2 \phi \,d\h^1 \le (1+c_1r\log\frac1{r})\int_{I_{r(1+c_1r\log\frac1{r})}} \sin^2 \psi \,d\h^1.
\]
\end{lem}
\proof{}
For a point $a\in\dOm$ with the unit tangent vector $\tau_a$ at $\dOm$, the Riemann mapping theorem 
 yields existence of a conformal map $\Psi_a:\RR^2_+\to\Omega$ such that $\Psi_a(0)=a$ and $\Psi_a'(0)=\tau_a$ where $\Psi'_a$ denotes the complex differential of $\Psi_a$. By the Kellogg-Warschawski theorem (see Pommerenke~\cite[Theorem 3.5]{Pommerenke:1992aa}), it follows that
$\Psi_a'$ extends to a Dini continuous map up to the boundary $\partial \R^2_+=\R\times\{0\}$.
  Near the origin, it has a modulus of continuity $\omega(\delta)=C\delta \log\frac1\delta$ {for $\delta>0$ small, where $C>0$ denotes here and in the following a constant depending only on $\Om$ that can change from line to line}.
In particular,
\[
|\Psi'_a(z)-\tau_a|\le C|z|\log\frac1{|z|}, \quad \textrm{ for $|z|$ small.}
\]
By complex integration, we deduce
$|\Psi_a(z)-a-\tau_a z|\le C |z|^2\log\frac1{|z|}$ for $|z|$ small. This implies that for $r<r_0$ sufficiently small,
\[
\Psi_a(B^+_{r(1-c_1 r\log\frac1r)}) \subset B_r(a)\cap\Omega \subset \Psi_a(B^+_{r(1+c_1 r\log\frac1r)}).
\]
Together with conformal invariance of the Dirichlet integral this implies the first part of claim {\it (b)}. The second part follows from the same inclusion together with our bounds on $|\Psi'_a(z)-\tau_a|$. \qed

\medskip

For an open set $G\subset \R^2_+$ and $\psi:G\to \R$, we define the localised functionals
\be
\label{funct_fhat}
\hat F^{(g)}_\eps(\psi;G) := \int_{G} 
|\nabla \psi|^2 \,dx+ \frac{1}{2\pi\eps} \int_{\overline{G}\cap(\R\times\{0\})} \sin^2\big(\psi(\cdot, 0)-g\big)\,dx_1,
\ee
where $g$ stands here for a function defined on $\overline{G}\cap(\R\times\{0\})$. Usually, this is 
the lifting of the tangent {vector field }defined in \eqref{g}, composed with the change of 
variables in Lemma~\ref{lem:conformalflat}.

Usually, we integrate over sets of the type $G=B_r^+$ or $G=B_r^+\setminus B_s^+$, where the corresponding 
boundary integral is over one or two intervals.
We can compare $\hat F_\eps^{(g)}$ and the special case $\hat F_\eps^{(0)}$ of zero boundary $g$ by subtracting a suitable harmonic extension:

\begin{lem}\label{lem:lem7}
Let $g$ be a Lipschitz function in $C^{0,1}((-1,1))$ and $\psi:B_1^+\to \R$. For every $r\in(0,1)$, we define
$\tilde g_r:\R\to \R$ by  
\[
\tilde g_r(x_1)=\begin{cases} g(x_1) &\quad\text{ if } |x_1|\le r,\\
g(\frac {r^2}{x_1}) &\quad\text{ if } |x_1|>r
\end{cases}
\]
and let $\hat g_r:\RR^2_+\to \R$ be the unique bounded harmonic extension of $\tilde g_r$ to $\RR^2_+$. Then we have for every $1{\leq} s<\infty$ and $r\in (0,1)$, with a  constant $C$ depending only on $s$ and the  Lipschitz constant $\|g'\|_{L^\infty}$ of $g$:
\begin{itemize}
\item[(i)] $\|\nabla \hat g_r\|_{L^s(B_r^+)}\le C r^{\frac2s};$ 
\item[(ii)] $\left\|\partial_{x_2} \hat g_r(\cdot, 0) \right\|_{L^s(I_r)}\le Cr^{\frac{1}{s}}$;
\item[(iii)]  $\nu \cdot \nabla \hat g_r=0$ on $\partial B_r(0)\cap \RR^2_+$ and $\nu$ is the unit outer normal vector to $\partial B_r(0)\cap \RR^2_+$.
\end{itemize}
If we set
\[
A(\psi;r)=\left|\hat F^{(g)}_\eps(\psi;B_r^+) - \hat F_\eps^{(0)}(\psi-\hat g_r;B_r^+) \right|, \quad r\in (0,1),
\]
then for every $p\in(1,\infty)$ and $r\in (0,1)$, we have, with constants $C$ depending on $p$ and the Lipschitz constant of $g$, 
\begin{equation}\label{eq:35}
A(\psi;r) \le C \|\nabla \psi\|_{L^p(B_r^+)} r^{2-\frac2p} +Cr^2;
\end{equation}
in particular, for $p=2$, 
\begin{equation}\label{eq:l7roughbd}
A(\psi;r) \le Cr \left(1+\sqrt{\hat F^{(g)}_\eps(\psi;B_r^+)}\right).
\end{equation} 
Furthermore, 
\begin{equation}\label{eq:36}
A(\psi;r)\le  C\|\psi(\cdot, 0)\|_{L^p(I_r)}  r^{1-\frac1{p}} + Cr^2, \quad p\in(1,\infty).
\end{equation}
\end{lem}
\begin{proof}{}
In order to prove $(i)$ and $(ii)$, we start by  noting that $\| \tilde g'_r\|_{L^s(\R)}\le  Cr^{\frac1s}\| g'\|_{L^\infty(-1,1)}$ for every $s\in [1, \infty]$ and $r\in (0,1)$ for some universal constant $C>0$ (with the convention that $1/\infty=0$). It is known that $x_1\mapsto \partial_{x_2} \hat g_r(x_1, 0)$ represents the Dirichlet-to-Neumann operator applied to $\tilde g_r$ that is given by the Hilbert transform $H$ of the derivative $\tilde g'_r$. As $H:L^s(\R)\to L^s(\R)$ is a bounded linear operator for $s\in (1, \infty)$, the estimates on $\| \tilde g'_r\|_{L^s(\R)}$ yield $(ii)$; {for the case $s=1$ we use the H\"older inequality and the embedding $L^2(I_r)\subset L^1(I_r)$ }. As $\partial_{x_j} \hat g_r$ is harmonic in $\R^2_+$ for $j=1,2$, the standard theory of harmonic functions, see e.g. 
Axler et al. \cite[Theorem 7.6]{Axler:2001aa}, implies (also for $s=1$) that 
$\|\partial_{x_j} \hat g_r(\cdot, x_2)\|_{L^s(\R)}\leq C \|\partial_{x_j} \hat g_r(\cdot, 0)\|_{L^s(\R)}$ for every $x_2>0$. Integrating on the strip $\R\times (0,r)$, we deduce the desired estimate in $(i)$. 
For proving $(iii)$, note that $\hat g_r(x)=\hat g_r(\frac{xr^2}{|x|^2})$ in $\R^2_+$ because uniqueness of the bounded harmonic extension implies invariance under the inversion at the circle $\partial B_r(0)$ (satisfied by the boundary data $\tilde g_r$). Then differentiating in radial direction and comparing both sides on the circle $\partial B_r(0)$
yield the claim $(iii)$.

For the claims on $A(\psi;r)$ note that
\[
A(\psi;r) = \left |\int_{B_r^+} \left( |\nabla \psi|^2 - |\nabla (\psi-\hat g_r)|^2\right) \, dx \right| = \left| \int_{B_r^+} 2 \nabla \psi \cdot \nabla \hat g_r -|\nabla \hat g_r|^2 \, dx \right|
\]
Now $\int_{B_r^+} |\nabla \hat g_r|^2 \, dx=O(r^2)$ by $(i)$. Furthermore, we have by H\"older's inequality and $(i)$ applied with $\frac1s=1-\frac1p$
\[
\left| \int_{B_r^+} 2 \nabla \psi \cdot \nabla \hat g_r  \, dx \right|\le Cr^{2-\frac2p} \left(\int_{B^+_r} |\nabla \psi|^p \, dx \right)^{\frac1p} ,
\]
which yields \eqref{eq:35}. For the final claim \eqref{eq:36}, integration by parts and H\"older's inequality applied with $\frac1s=1-\frac1p$, combined with $(ii)$ and $(iii)$ 
imply
\[
\left|\int_{B_r^+} \nabla \psi \cdot \nabla \hat g_r  \, dx \right|= \left|\int_{I_r}  \psi \partial_{x_2} \hat g_r \, d\h^1 \right| \le C \left(\int_{I_r} |\psi|^p\, d\h^1\right)^{\frac1p} r^{1-\frac1{p}}.
\]
\qed
\end{proof}

After reducing the study of $\hat F_\eps^{(g)}$ defined at \eqref{funct_fhat} to the special case $\hat F_\eps^{(0)}$ thanks to the above lemma,
we further reduce the analysis of the two-dimensional energy functional $\hat F_\eps^{(0)}$ to a one-dimensional (nonlocal) functional defined for functions $\f:I\to \R$ for an interval $I\subset \R$:
\be
\label{def:feps}
F_\eps(\f;I) = \frac{1}{2\pi} \int_{I\times I}
\left| \frac{\f(s)-\f(t)}{s-t}
\right|^2 \, ds dt + \frac{1}{2\pi \eps} \int_{I} \sin^2 \f \, dt.
\ee

\begin{lem}\label{lem:ABS2d1d}
If $\psi:B_r^+\to \R$ is an $H^1$ function in $B^+_r$ for some $r>0$, then
\[\hat F_\eps^{(0)}(\psi; B_r^+) \ge F_\eps\big(\psi(\cdot, 0); I_r\big)\]
where the RHS is given by the trace $\psi(\cdot, 0)$ of $\psi$ on the interval $I_r=\partial B^+_r\cap (\R\times \{0\})$.
\end{lem}
\begin{proof}{}
For the half-space $\R^2_+$ (corresponding to $r=\infty$), the Dirichlet integral in $\hat F_\eps^{(0)}$ and the nonlocal functional in $F_\eps$ can be compared using a standard Fourier space argument:
$$\int_{\R^2_+} |\nabla \psi|^2 \, dx\geq \|\psi(\cdot, 0)\|^2_{\dot{H}^{1/2}(\R)}=\frac{1}{2\pi} \int_{\R\times \R}
\left| \frac{\psi(x_1, 0)-\psi(\tilde x_1, 0)}{x_1-\tilde x_1}
\right|^2 \, dx_1 d\tilde x_1.$$
The bounded domain version in $B_r^+$ can be deduced by inversion at $\partial B_r$ as in Lemma \ref{lem:lem7} 
(see Alberti-Bouchitt\'e-Seppecher  \cite[Corollary 6.4]{AlbertiBouchittSeppeche:1998a} for details). The constant $1$ in the above inequality is optimal (see e.g. \cite[Remark 6.5]{AlbertiBouchittSeppeche:1998a}). \qed
\end{proof}

\medskip

The following rearrangement inequality is essential in the proof of the compactness result for the functional $F_\eps$: 

\begin{lem}\label{lem:intbdABS}
Let $I\subset \RR$ be a { bounded} interval and $A,B\subset I$ be measurable sets of positive measure with
$A\cap B = \emptyset$. 
Set $P=I\setminus(A\cup B)$. Then
\begin{equation}\label{eq:rearr}
\int_A \int_B \frac{1}{|s-t|^2} \, dsdt 
\ge 
\log \frac{(|I|-|A|)(|I|-|B|)}{|I|(|I|-|A|-|B|)}\ge \log \frac{|B|}{|I|}+ \log\frac{|A|}{|I|}{ -}\log\frac{|I|-|A|-|B|}{|I|}.
\end{equation}
If additionally $|B|{\geq}c|I|$ for some $c\in(0,1)$, we have
\begin{equation}\label{eq:rearr2}
\int_A \int_B \frac{1}{|s-t|^2} \, dsdt 
\ge \log \left( 
1+ \frac{c|A|}{|P|}
\right).
\end{equation}
\end{lem}
\begin{proof}{}
By a simple rearrangement lemma 
(see \cite[Lemme 2]{AlbertiBouchittSeppeche:1994a}), 
\[
\int_A \int_B \frac{1}{|s-t|^2} \, dsdt 
\ge 
\int_0^{|A|}\int_{|I|-|B|}^{|I|}
\frac{1}{|s-t|^2} \, dsdt
=\log \frac{(|I|-|A|)(|I|-|B|)}{|I|(|I|-|A|-|B|)}
\]
and the last part of \eqref{eq:rearr} follows using that $|I|-|A|\geq |B|$ and $|I|-|B|\geq |A|$. 
We note that
\[
\frac{(|I|-|A|)(|I|-|B|)}{|I|(|I|-|A|-|B|}
= 1+ \frac{|B|}{|I|} \frac{|A|}{|I|-|A|-|B|}
\ge 1+c\frac{|A|}{|P|}
\]
so \eqref{eq:rearr2} now follows by the monotonicity of the logarithm.\qed
\end{proof}

\medskip

Now we prove a first compactness result for functional $F_\eps$ in \eqref{def:feps} in the weak $L^p$ topology:

\begin{pro}\label{prop:orliczbd}
Let $I\subset\RR$ be a {bounded open} interval and  $M>0$. Then there exists $\eps_0>0$ such that for every
sequence / family $(\f_\eps)_\eps$ of functions such that the functional $F_\eps$ defined in \eqref{def:feps} satisfies 
$F_\eps(\f_\eps;I) \le M|\log\eps| $, 
there exists a sequence / family $(k_\eps)_\eps$ of integers such that
$(\f_\eps-\pi k_\eps)_{\eps\in (0, \eps_0)}$ is bounded in $L^p(I)$ for every $p\in [1,\infty)$.
\end{pro}
\begin{proof}{}
We assume without loss of generality that $|I|=1$ (otherwise, one rescales by the length of the interval $I$ which implies only a change of the parameter $\eps$ in the functional $F_\eps$ as the nonlocal part of $F_\eps$ is scaling invariant). We denote
$$a\wedge b=\min(a,b)\quad \textrm{and} \quad a\vee b=\max(a,b), \quad a,b\in \R.$$

\medskip

\nd {\bf A particular case.} We assume that $|\{\f_\eps<0\}|>\frac14$ for every small $\eps>0$. We want to prove that the positive part $(\f_\eps)_+=\f_\eps\vee 0$ is uniformly bounded in $L^p(I)$ for every $p\in [1,\infty)$ as $\eps\to 0$.
For every $\eps$, we use the truncations of $\f_\eps$ between $k\pi$ and $(k+1)\pi$ given by
 \be
 \label{trunc1}
 T_k \f_\eps= (\f_\eps\wedge (k+1)\pi)\vee k\pi, \quad \textrm{ for every } k\in\ZZ.
 \ee Fix a small $\gamma>0$. We consider the following sets
$$A_k^\eps=\{ T_k  \f_\eps>(k+1)\pi-\gamma\}, \quad B_k^\eps = \{T_k \f_\eps<k\pi+\gamma\}, \quad k\geq 0$$
together with  
$$\alpha_k^\eps=|A_k^\eps| \quad \textrm{ and }\quad \rho_k^\eps=1-|A_k^\eps\cup B_k^\eps|.$$
Note that $\{\f_\eps<0\}\subset B_k^\eps$ for $k\geq 0$ so that $|B_k^\eps|/|I|>1/4$ (by the assumption of this case); also $(\alpha_k^\eps)_{k\geq 0}$ is a non-increasing sequence and we have the estimate 
\be
\label{esti_rho}
M|\log \eps|\geq F_\eps(\f_\eps;I)\geq \frac1{2\pi \eps} \int_{I\setminus (A_k^\eps\cup B_k^\eps)} \sin^2 \f_\eps\, dt\geq \frac{\rho_k^\eps}{C(\gamma)  \eps},
\ee
i.e., $\rho_k^\eps\le C(\gamma) M \eps|\log\eps|$.
Moreover, if $\alpha_k^\eps>0$, then  
$\rho_k^\eps>0$ (otherwise, we would have $|A_k^\eps|+|B_k^\eps|=1$ for the nontrivial partition $A_k^\eps\neq \emptyset$ and $B_k^\eps\neq \emptyset$ of $I$ which leads to a contradiction with the fact that $H^{1/2}$-functions have no jump discontinuities) and by 
Lemma \ref{lem:intbdABS}, we obtain:
\be
\label{inter1}
M\ge \frac{1}{|\log\eps|}F_\eps(T_k \f_\eps; I) \ge \frac{2}{2\pi |\log\eps|}
\int_{A_k^\eps} \int_{B_k^\eps} \frac{|\f_\eps(s)-\f_\eps(t)|^2}{|s-t|^2} \, dsdt\ge  
\frac{(\pi-2\gamma)^2\log(1+ \frac{\alpha_k^\eps}{4\rho_k^\eps})}{\pi|\log\eps|}.
\ee
Now we decompose the set of non-negative integers:
$$\NN={\cal K}_\eps\cup {\cal N}_\eps, \, {\cal K}_\eps:=\{k\geq 0\, :\, \alpha_k^\eps<\eps^{\frac13} \}, \,
 {\cal N}_\eps:=\{k\geq 0\, :\, \alpha_k^\eps\geq \eps^{\frac13} \}.$$ 

\medskip

\nd {\bf Subcase i).} Assume that ${\cal N}_\eps\neq \emptyset$. Note that for $\eps\leq \eps(M, \gamma)$, we have for every
$k\in {\cal N}_\eps$ that $\log (1+ \frac{\alpha_k^\eps}{4\rho_k^\eps})>\frac13|\log\eps|$ because $0<\rho_k^\eps\le C(\gamma) M \eps|\log\eps|$. Let $k_0^\eps
=\sup {\cal N}_\eps\in\NN\cup \{+\infty\}$.
 We claim that $(k_0^\eps)_\eps$ is uniformly bounded in $\eps$. Indeed, as $(\alpha_k^\eps)_{k\geq 0}$ 
is non-increasing, we know that ${\cal N}_\eps=\{0, 1, \dots, k_0^\eps\}$, i.e., $\alpha_k^\eps \geq \eps^{\frac13}$ for every $0\leq k\leq k_0^\eps$, so that
$$M\ge \frac{1}{|\log\eps|}\sum_{k=0}^{k_0^\eps} F_\eps(T_k \f_\eps; I)\stackrel{\eqref{inter1}}{\geq} 
\sum_{k=0}^{k^\eps_0}\frac{(\pi-2\gamma)^2
\log (1+ \frac{\alpha_k^\eps}{4\rho_k^\eps})}
{\pi |\log\eps|} \geq \frac{k_0^\eps}3$$ 
which proves our claim.
Let $k_0=\limsup_{\eps\in (0, \eps_0]}
k_0^\eps<\infty$. In particular, for $\eps\leq \eps_0$, $\alpha_{k_0+1}^\eps<\eps^{1/3}$.
Now the one-dimensional Moser-Trudinger inequality
(see Taylor \cite[Proposition 4.2]{Taylor:1997a}; compare \cite[Lemma~2.10]{Kurzke:2006a})
 implies the existence of constants $c_1,c_2>0$ such that 
\[
\int_{\{\f_\eps>{ (k_0+2)}\pi-\gamma\}} \exp \frac{c_1 (\f_\eps-{ (k_0+2)}\pi+\gamma)^2}{
M|\log\eps|}\, dt\le c_2 \alpha_{k_0+1}^\eps (\leq {c_2} \eps^{1/3}),
\]
so that for every ${ k\geq k_0+1}$, by definition of $\alpha_k^\eps$ it follows
\[
\alpha_k^\eps \exp \left(c_1\pi^2\frac{(k-k_0-1)^2}{M|\log\eps|} \right)\leq \int_{\{\f_\eps>(k_0+2)\pi-\gamma\}} \exp \frac{c_1 (\f_\eps-(k_0+2)\pi+\gamma)^2}{
M|\log\eps|}\le c_2 \exp ( -\frac13|\log\eps|),
\]
yielding for $k\ge k_0+1$
\begin{equation}\label{eq:orlak}
\alpha_k^\eps\leq  c_2 \exp \left( -\frac13|\log\eps|-c_1\pi^2\frac{(k-k_0-1)^2}{M|\log\eps|}\right)\leq c_2 \exp\left( -\frac{2\pi\sqrt{c_1}(k-k_0-1)}{\sqrt{3M}}\right), 
\ee
where we used $a^2+b^2\geq 2ab$ in the argument of the exponential.
Therefore, we obtain for the positive part of $\f_\eps$ and $p\in [1, \infty)$:
\begin{align*}
\int_I |(\f_\eps)_+|^p \, dt&=\sum_{k\geq 0} \int_{\{k\pi<\f_\eps<(k+1)\pi\}} |T_k\f_\eps|^p \, dt   
\le \pi^p+\sum_{k\geq 1} \int_{A_{k-1}^\eps} |T_k\f_\eps|^p \, dt\\
& {\leq} 
\pi^p+C\sum_{k=0}^\infty (k+1)^p \alpha_k^\eps\\
&\stackrel{\eqref{eq:orlak}}\leq \tilde C\sum_{k=0}^{k_0} (k+1)^p+\tilde C\sum_{k\geq k_0+1} (k+1)^p 
\exp\left( -\frac{2\pi\sqrt{c_1}(k-k_0-1)}{\sqrt{3M}}\right),
\end{align*}
which is  bounded independently of $\eps$, yielding the claimed $L^p$ bound of $(\f_\eps)_+$. 

\medskip

\nd {\bf Subcase ii).} Assume that ${\cal N}_\eps=\emptyset$, i.e., $\alpha_k^\eps<\eps^{\frac13}$ for every $k\geq 0$. Then by \eqref{eq:orlak}, we deduce that $\alpha_k^\eps$ satisfies an exponential decay for every $k\geq 1$ and the same argument as in Subcase i) yields the $L^p$ bound of $(\f_\eps)_+$. 

\medskip

\nd {\bf The general case}. For a measurable function $\f_\eps:I\to \R$ with $|I|=1$, there exists an integer $k_\eps$ such\footnote{One can consider the the smallest $k_\eps\in \ZZ$ such that $|\{\f_\eps<k_\eps\pi \}|>\frac14$.} that $|\{\f_\eps<k_\eps\pi \}|>\frac14$ and $|\{\f_\eps>(k_\eps-1)\pi \}|>\frac14$. 
By considering $\tilde \f_\eps:=\f_\eps-k_\eps \pi$, we deduce that $F_\eps(\tilde \f_\eps;I)=F_\eps(\f_\eps;I)$ and by the particular case discussed before, we have that the positive parts of the sequence / family
$(\tilde \f_\eps)_\eps$ are bounded in $L^p$.
The same argument yields that the sequence / family of positive parts of $-(\pi+\tilde \f_\eps)$, i.e., $(-\pi-\tilde \f_\eps)_+=(\tilde \f_\eps+\pi)_-$, is also bounded in $L^p$. Together these bounds yield the $L^p$ bound of $(\tilde \f_\eps)_\eps$. 
\qed
\end{proof}

\begin{rem}
From \eqref{eq:orlak} we can actually deduce a bound not just in $L^p$, but in a certain Orlicz space.
 The type of Orlicz space ($e^{cL}$ with a constant
 of order $\frac{1}{\sqrt{M}}$) is 
essentially optimal by 
an example presented in \cite{Kurzke:2006a}.
\end{rem}

\medskip

We can improve now the result in Proposition \ref{prop:orliczbd} by showing the compactness in {\bf strong} $L^p$ topology and derive a first order lower bound for the functional $F_\eps$ defined in \eqref{def:feps}.

\begin{pro}\label{prop:strcomp}
 Let $I\subset \R$ be a bounded open interval
and let {$(\f_\eps)_\eps$ be a sequence / family of functions such that $F_\eps(\f_\eps;I)\le M|\log\eps|$ as $\eps\to 0$  for 
some fixed $M>0$. Then for a subsequence $\eps\to 0$ (still denoted $(\f_\eps)$), there exists a sequence $(k_\eps)_\eps$ of integers such that 
$\f_\eps-k_\eps \pi\to \f$ {\bf strongly} in $L^p(I)$ for every $p\in [1, \infty)$, where $\f$ {is a piecewise constant function in }$BV(I;\pi \ZZ)$. Furthermore, every sequence / family $(\f_\eps)$ satisfying the above convergence as $\eps\to 0$ yields the following energy lower bound at first order:}\footnote{In the following, we use the $BV$-seminorm of a function $f:\Omega\to \R$: $\|\f\|_{BV(\Omega)}=|D\f|(\Omega)$.}
\begin{equation}\label{eq:abslibd}
\liminf_{\eps\to 0}  \frac{1}{|\log\eps|}F_\eps(\f_\eps;I)
\ge  \|\f\|_{BV}=\sum_{t \in S(\f)} |\f(t+)-\f(t-)|,
\end{equation}
where $S(\f)$ denotes the finite set of jumps of $\f$ and $\f(t\pm)\in \pi \ZZ$ the traces of $\f$ at a jump $t$.
\end{pro}
\begin{proof}{} We may assume that $I=(0,1)$ (by the same argument as in the proof of Proposition \ref{prop:orliczbd}). We start by treating a particular case and then we prove the general case.

\medskip

\nd {\bf A particular case.} Assume that
$\f_\eps$ takes values into $[0,\pi]$ for every $\eps$. We can then follow
the argument of Alberti-Bouchitt\'e-Seppecher \cite{AlbertiBouchittSeppeche:1994a}:
Since $(\f_\eps)$ is uniformly bounded, then for a subsequence, we can assume that
$\f_\eps$ is weakly$^*$ convergent in $L^\infty(I)$ to a function $\f:I\to [0, \pi]$. 
By the fundamental theorem of Young measures (see Ball \cite{Ball:1989aa} or M\"uller \cite{Muller:1999a}), there exists a
family of probability measures $\{\mu_t\}_{t\in I}$ (depending measurably on $t\in I$) over the range $[0,\pi]$
such that for any continuous test function
$\zeta\in C^0([0,\pi]\times[0,1])$,
\[
\int_0^1 \zeta(\f_\eps(t),t)\, dt \to \int_0^1 \int_0^\pi \zeta(z,t)\,d \mu_t(z)\,dt \quad \textrm{ as } \eps\to 0.
\]
Choosing $\zeta(z,t)=\sin^2 z$ for every $z\in [0,\pi]$ and $t\in I$, since $F_\eps(\f_\eps;I)\le M|\log\eps|$, it follows that 
\[
0 = \lim_{\eps\to 0} \int_0^1\sin^2 \f_\eps(t)\, dt = 
\int_0^1 \int_0^\pi \sin^2 z \,d \mu_t(z)
dt,
\]
and since $\sin^2 z> 0$ for $z\in (0, \pi)$, it follows that $\supp \mu_t \subset \{0,\pi\}$ 
for almost every $t$, and we can write 
$\mu_t = \theta(t) \delta_0 + (1-\theta(t))\delta_\pi$ for 
some {measurable} function $\theta:I\to [0,1]$. 

\medskip

\nd {\it Claim: For a.e. $t\in I$, $\theta(t)\in \{0,1\}$, i.e., $\mu_t$ is a Dirac measure.}

\medskip 

\nd To prove the claim, we first set
\[
S=  I\setminus \left\{ t_0\in I \, : \, \lim_{r\to 0}\frac{1}{2r} \int_{t_0-r}^{t_0+r} \theta(t)\, dt \text{ exists and belongs to } 
\{0,1\}  \right\}.
\]
Setting $I_r(t):=(t-r,t+r)\subset I$ for $t\in I$ and small $r>0$, the above definition implies for $t_0\in S$ that there exist $\delta>0$ and a decreasing sequence $r_k\to 0$ such that for all $k$,
\be
\label{new_def}
\frac1{2r_k} \int_{I_{r_k}(t_0)} \theta(s)\, ds\in(\delta,1-\delta).
\ee
Indeed, the function $r\mapsto \frac1{2r} \int_{I_{r}(t_0)} \theta(s)\, ds$ is continuous for small $r>0$, which implies that 
$J_{t_0}:=\big[\liminf_{r\to 0} \frac1{2r} \int_{I_{r}(t_0)} \theta(s)\, ds, \limsup_{r\to 0} \frac1{2r} \int_{I_{r}(t_0)} \theta(s)\, ds\big]$ is a closed interval $\subset [0,1]$ that is not reduced to $\{0\}$ or $\{1\}$ for $t_0\in S$. Therefore, there exists $\delta>0$ such that 
$J_{t_0}\cap (\delta, 1-\delta)\neq \emptyset$ which yields \eqref{new_def}.

\medskip

\nd {\bf Step 1}. {\it We show that for every $t_0\in S$ and  
any $\gamma\in (0, \pi)$,
we have that there exists a decreasing sequence $r_k\to 0$ such that}
\begin{equation}\label{eq:I1NEW}
\liminf_{k\to\infty}\liminf_{\eps \to 0} \frac{|I_{r_k}(t_0)\cap \{\f_\eps<\gamma \}|}{|I_{r_k}(t_0)|}>0
\end{equation}
{\it and}
\begin{equation}\label{eq:I2NEW}
\liminf_{k\to\infty}\liminf_{\eps \to 0} \frac{|I_{r_k}(t_0)\cap \{\f_\eps>\pi-\gamma \}|}{|I_{r_k}(t_0)|}>0.
\end{equation}
Indeed, let $t_0\in S$ with $\delta>0$ and $r_k\to 0$ satisfying \eqref{new_def}. We choose $\gamma_1$ and $\gamma_2$ such that $0<\gamma_1<\gamma_2<\pi$ and we consider a test function $\zeta=\zeta(z)$ such that $\zeta$ is continuous on $[0,1]$, 
$\zeta=1$ on $[0,\gamma_1]$, $\zeta=0$ on $[\gamma_2,\pi]$ and $0<\zeta<1$ on $(\gamma_1,\gamma_2)$. Then 
\[
\int_{I_r(t_0)} \zeta(\f_\eps)\, dt \to \int_{I_r(t_0)} \int_0^\pi\zeta(z) d\mu_t(z)\, dt = \int_{I_r(t_0)} \theta(t) \, dt \quad 
\textrm{as } \eps\to 0,
\]
 because $\{\mu_t\}_{t\in I_r(t_0)}$ is also the family of Young measures of the restriction $(\f_\eps\big|_{I_r(t_0)})_\eps$ for every small $r$. As $\int_{I_r(t_0)} \zeta(\f_\eps)\, dt\leq |I_r(t_0)\cap \{\f_\eps<\gamma_2\}|$ we deduce that
$$\int_{I_r(t_0)} \theta(t) \, dt \leq \liminf_{\eps\to 0} |I_r(t_0)\cap \{\f_\eps<\gamma_2\}|.$$
Setting $r:=r_k$ and $\gamma:=\gamma_2$, after dividing by $2r_k$ and passing to $\liminf$ as $k\to\infty$, the desired inequality \eqref{eq:I1NEW} holds true. The proof of  \eqref{eq:I2NEW} is analogous.

\medskip

\nd {\bf Step 2}.{\it  We show that the set $S$ is finite}. 
For that, let $(I_j)$, $1\le j\le J$ 
be a finite family of disjoint open  intervals inside $I$ such that $I_j\cap S\neq \emptyset$ for every $j$.
For some $\gamma\in (0, \frac\pi2)$, 
we consider the sets
$A_j^\eps=I_j\cap \{\f_\eps<\gamma\}$ and 
$B_j^\eps=I_j\cap \{\f_\eps>\pi-\gamma\}$. For every $j$, there exists $t_j\in I_j\cap S$ so that by \eqref{eq:I1NEW} and \eqref{eq:I2NEW}, there exist $r_j, \tilde r_j>0$ small satisfying $I_{r_j}(t_j), I_{\tilde r_j}(t_j) \subset I_j$ and
$$\liminf_{\eps\to 0} |A_j^\eps|\geq \liminf_{\eps \to 0} {|I_{r_j}(t_j)\cap \{\f_\eps<\gamma \}|}>0, \quad 
\liminf_{\eps\to 0} |B_j^\eps|\geq \liminf_{\eps \to 0} {|I_{\tilde r_j}(t_j)\cap \{\f_\eps>\pi-\gamma \}|}>0. $$
Furthermore,
since $\sin^2 z\ge c(\gamma)>0$ for every $z\in (\gamma,\pi-\gamma)$, we deduce by \eqref{esti_rho} that 
$|I_j\setminus(A_j^\eps \cup B_j^\eps)|\le  C(\gamma, M)\eps|\log\eps|$. Applying the rearrangement result \eqref{eq:rearr} with $|A_j^\eps|\le |I_j|-|B_j^\eps|$, we obtain
\[
\int_{A_j^\eps} \int_{B_j^\eps}
\frac{1}{|t-s|^2} \, dt ds \ge 
\log\frac{|B_j^\eps|}{|I_j|} \frac{|A_j^\eps|}{|I_j|}
-\log\frac{|I_j|-|A_j^\eps|-|B_j^\eps|}{|I_j|}=\log (a_j^\eps b_j^\eps)-\log \rho_j^\eps,
\]
within the notation $a_j^\eps:=\frac{|A_j^\eps|}{|I_j|}$, $b_j^\eps:=\frac{|B_j^\eps|}{|I_j|}$ and 
$\rho_j^\eps:=\frac{|I_j\setminus(A_j^\eps \cup B_j^\eps)|}{|I_j|} = 1-(a_j^\eps+b_j^\eps)$. Combined with the argument in \eqref{esti_rho} and \eqref{inter1}, we obtain
\[
 \frac{1}{|\log\eps|}F_\eps( \f_\eps;I_j) \ge \frac{(\pi-2\gamma)^2}{\pi|\log\eps|} \left( 
\log (a_j^\eps b_j^\eps)-\log\rho_j^\eps+ \frac{|I_j|c(\gamma)}{2(\pi-2\gamma)^2 \eps} 
\rho_j^\eps\right)  .
\]
Using $\liminf_{\eps\to 0} a_j^\eps b_j^\eps >0$ and\footnote{\label{foot1}If $K>1$ and $f(\rho)=K\rho-\log \rho$ for $\rho\in (0,1)$, then the minimum of $f$ is achieved at $\rho_K=1/K$ and $f(\rho)\geq f(\rho_K)=\log K+1$.}
$-\log\rho_j^\eps + K\rho_j^\eps \ge \log K+1$
for $K=\frac{|I_j|c(\gamma)}{2(\pi-2\gamma)^2\eps}\gg 1$, then summing over $j$ we conclude 
that
\begin{equation}\label{eq:29a}
M\ge\sum_{j=1}^J \liminf_{\eps\to 0}\frac{1}{|\log\eps|} F_\eps( \f_\eps;I_j) \ge J \liminf_{\eps\to 0}  \frac{(\pi-2\gamma)^2}
{\pi |\log\eps|} 
\left(|\log\eps|+\tilde c
\right)=\frac{(\pi-2\gamma)^2}
{\pi} J
\end{equation}
for a constant $\tilde c$ depending on $\gamma$ and the product $\Pi_j |I_j|$.
Therefore, $J$ is bounded by $M$ (up to a constant), hence $S$ must be a finite set. 

\medskip

\nd {\it Proof of Claim:} By the above considerations, we deduce that we can choose a representative $\theta$ defined on $I$ such that for every $t\in I\setminus S$, 
$$\theta(t)=\lim_{r\to 0} \frac{1}{2r}\int_{t-r}^{t+r}\theta(s) \, ds \in \{0,1\}.$$
If $t_1<t_2$ are two consecutive points in the (finite) set $S$, then $\theta$ satisfies the above condition for every $t\in (t_1,t_2)$ which implies that either $\theta\equiv 0$, or $\theta\equiv 1$ in the interval $(t_1,t_2)$. In other words, $\theta$ is a piecewise constant functions with values into $\{0,1\}$ whose jump points belong to $S$ (i.e., $\theta$ is a characteristic function of a finite union of disjoint open intervals, so $\theta\in BV$). 
In particular, this shows that $\mu_t$ is a Dirac measure for almost every $t$, finishing the proof of the Claim.

It now follows that $\f_\eps\to \f$ in $L^1(I)$ by a well known property of Young measures (see
Valadier \cite[Theorem 9]{Valadier:1994aa}). Moreover, since $\f(t)=\int_0^\pi z\, d\mu_t(z)=(1-\theta(t))\pi$ for a.e. $t\in I$, 
we find a representative $\f$ that is piecewise constant with values into $\{0,\pi\}$ almost everywhere,  and the jump points of $\f$ are those of $\theta$, hence included in $S$. By the finiteness of $S$ we obtain $\f\in BV(I;\{0,\pi\})$. Since both $\f$ and all $\f_\e$ are bounded, we obtain the convergence $\f_\eps\to \f$ in $L^p(I)$ for $1\le p<\infty$.

From \eqref{eq:29a} we find 
\[
\liminf_{\e\to 0} \frac1{|\log\e|} F_\e(\f_\e;I) \ge \frac{(\pi-2\gamma)^2}{\pi} \Hh^0(S),
\]
and letting $\gamma\to 0$ we find that $\liminf_{\gamma\to 0}\frac{(\pi-2\gamma)^2}{\pi} \Hh^0(S)\geq \|\f\|_{BV(I)}$, finishing the proof in our assumed particular case.

\medskip

\nd {\bf General case.} To recover the general case, we use a truncation argument similar to that of  Garroni-M\"uller \cite{GarroniMuller:2006a}.
From Proposition \ref{prop:orliczbd}, we find integers $k_\e$ such that $\f_\e-k_\e \pi$ is bounded in $L^p(I)$ for every $p\geq 1$. We may assume that $k_\e=0$ and (choosing a subsequence) 
$\f_\e \rightharpoonup \f$ in $L^p(I)$ for any $p\geq 1$. 
For every $k\in\ZZ$, the particular case above applied to the truncations $T_k\f_\e$ defined at \eqref{trunc1} yields for further subsequences,
$T_k \f_\e\to f_k$ in $L^1$ for some $f_k\in BV(I;\{k\pi,(k+1)\pi\})$
and 
\[
\liminf_{\e\to 0} \frac1{|\log\e|} F_\e(T_k \f_\e;I) \ge \|f_k\|_{BV}.
\]
For any positive integer $M$, we now consider another truncation operator:
\[
T^M \psi = \big(\psi\vee (-M\pi)\big)\wedge M\pi.
\]
Adding up the above pieces $T_k \f_\e$ in the set $\{k\pi\leq \f_\eps\leq (k+1)\pi\}$ for $k=-M, \dots, M-1$, we obtain the existence of $\f^M\in BV(I;\pi\ZZ)$ such that 
\[
T^M \f_\e \to \f^M\quad\text{in $L^1(I)$}.
\]
As in \cite{GarroniMuller:2006a}, we note that $|\f_\e|\le \frac{|\f_\e|^2}{M}$ on $\{|\f_\e|\geq M\}$. The uniform 
$L^2$-bound and weak lower semicontinuity of the $L^1$ norm then yield for every $M>0$:
\[
\|\f^M-\f\|_{L^1} \le \liminf_{\e\to 0} \|T^M \f_\e - \f_\e\|_{L^1} = \liminf_{\e\to 0} \int_{\{|\f_\e|>M\}} \left( |\f_\e|-M \right)\, dx \le \frac CM.
\]
Splitting $\f_\e-\f = \f_\e - T^M \f_\e + T^M \f_\e - \f^M + \f^M-\f$, we find that 
\[
\|\f_\e - \f\|_{L^1} \le \frac CM + \|T^M \f_\e - \f^M\|_{L^1} + \frac CM.
\]
{As $T^M \f_\e \to \f^M$ in $L^1$ and $M$ is arbitrary, we obtain $\f_\e\to \f$ in $L^1(I)$. As $(\f_\eps)$ is uniformly bounded in any $L^p$, by interpolation, we obtain $\f_\e\to \f$ in $L^p(I)$ for all $p\in[1,\infty)$. We also obtain that $T_k \f=f_k$.
By super-additivity in $F_\eps$, Fatou's lemma and the lower bound from the particular case 
\[
\infty>\liminf_{\e\to 0} \frac1{|\log \e|}F_\eps(\f_\eps;I ) \ge \liminf_{\e\to 0} \frac1{|\log \e|} \sum_k F_\eps(T_k \f_\eps;I)
\ge  \sum_k \|f_k\|_{BV}.
\]
As $f_k$ are piecewise constant with the only possible jumps of size $\pi$, we find that $\|f_k \|_{BV} = 0$ for all but finitely many $k$. Since $\f=\sum_k T_k \f=\sum_k f_k$, we deduce that $\f\in BV(I;\pi\ZZ)$, and 
using additivity of the $BV$ seminorm of $f_k$ taking the values $\{k\pi, (k+1)\pi\}$, we finally obtain \eqref{eq:abslibd}.}
\qed
\end{proof}

\medskip

In Proposition~\ref{prop:strcomp}, the lower bounds for $F_\eps$ are accurate up to $o(|\log\eps|)$.
In the following, we improve the error to $O(1)$ by means of a  co-area argument inspired by 
the work of Sandier \cite{Sandier:1998a} on the Ginzburg-Landau energy  (a different method was found by Jerrard \cite{Jerrard}). To this end, we need to compare
the nonlocal energy of a (scalar) function to that of a ${\Ss^0}\approx\{0,\pi\}$-valued variant of the same function
(the corresponding step in Sandier's argument compares the Dirichlet energy of a complex valued function $u$ with 
a $\Ss^1$-valued variant given by $\frac{u}{|u|}$).

\begin{lem}\label{lem:bdviaTheta}
Let $I$ be a bounded interval and  $\f\in H^{1/2}(I)$with $0\le \f\le \pi$ and define $\hat \f:I\to\{0,\pi\}$ by
\[
\hat \f= \begin{cases}
0 & \quad \f<\frac\pi 2, \\
\pi & \quad \f\ge \frac\pi 2.
\end{cases}
\]
For $0\le \gamma \le \frac\pi 2 $ we let 
$E_\gamma= \{s\in I\, : \, |\f(s)-\frac\pi 2 |>\frac\pi2- \gamma\}$.
Let 
\[
\Theta_\f(\gamma) = \int_{E_\gamma} \int_{E_\gamma} \left|\frac{\hat \f(s)-\hat \f(t)}{s-t}\right|^2 ds dt, 
\quad \gamma\in [0, \frac \pi2],
\]
then $\Theta_\f$ is a nondecreasing function and
\begin{equation}
\int_I \int_I \left|\frac{ \f(s)- \f(t)}{s-t}\right|^2 ds dt\ge \int_0^{\frac\pi 2} (1-\frac{2\gamma}{\pi})^2 d\Theta_\f(\gamma)\, 
\end{equation}
where $d\Theta_\f$ denotes the measure corresponding to the (distributional) derivative of $\Theta_\f$.
\end{lem}

\begin{proof}{}
Let 
\[
\tilde\Theta_\f(\gamma) = \int_{E_\gamma} \int_{E_\gamma} \left|\frac{\f(s)-\f(t)}{s-t}\right|^2 \, ds dt, \quad \gamma\in [0, \frac \pi2].
\]
{As $(E_\gamma)_\gamma$ is nondecreasing in $\gamma$ (with respect to inclusion),} then $\Theta_\f$ and $\tilde \Theta_\f$ are nondecreasing functions. For $\pi\geq \gamma>\tilde \gamma\geq 0$, we have that 
\[
\tilde \Theta_\f(\gamma)-\tilde \Theta_\f(\tilde \gamma) \ge \left(\frac{\pi-2\gamma}{\pi} \right)^2 \left( \Theta_\f(\gamma)-\Theta_\f(\tilde \gamma)\right).
\]
Letting $\tilde \gamma\to \gamma$, we see that the  distributional derivatives satisfy (as measures on $[0,\frac\pi2]$)
\[
d \tilde \Theta_\f(\gamma) \ge \frac{(\pi -2\gamma)^2}{\pi^2} d \Theta_\f(\gamma)
\]
and since $\tilde \Theta_\f(0)=0$, we obtain by integrating over $\gamma\in (0, \frac\pi2)$:
\[
\int_I \int_I \left|\frac{ \f(s)- \f(t)}{s-t}\right|^2 \, ds dt\ge \tilde \Theta_\f(\frac\pi 2) \ge \int_0^{\frac\pi 2} (1-\frac{2\gamma}{\pi})^2 \, d \Theta_\f(\gamma).
\]
\qed
\end{proof}

Now we show a more precise lower bound of $F_\eps$ up to an error $O(1)$:

\begin{pro}\label{prop:lolobd}
There is a universal constant $M_0>0$ such that the following holds.
Assume $\ell\in \ZZ$ and 
 $\f_\eps\to \ell \f_*$ in $L^1((-1,1))$ {for a sequence / family} $\eps\to 0$, where 
$\f_*(x)={\pi}$ for $x\in (-1,0)$ and $\f_*(x)=0$ for $x\in (0,1)$.
Then for every $r\in (0,1)$ we have
\begin{equation}\label{eq:lolobd1}
\liminf_{\eps\to 0} \bigg(F_\eps\big(\f_\eps; (-r,r)\big) - \pi |\ell| \log\frac{r}{\eps}\bigg) \ge -|\ell| M_0.
\end{equation}
\end{pro}
\proof{} 
Without loss of generality, we can assume $\ell\neq 0$ and  
$$F_\eps\big(\f_\eps; (-r,r)\big) \leq \pi |\ell| \log\frac{r}{\eps} \quad \textrm{for every } r\in (0,1) \textrm{ and } \eps<r \textrm{ small}$$ (otherwise the conclusion is obvious). We consider first the case of a single limit transition layer (i.e., $\ell=1$) and then we deduce the general case.

\medskip

\nd {\bf The particular case of $\ell=1$.} Without loss of generality we may assume $0\le \f_\eps\le \pi$, by replacing $\f_\eps$ with $(\f_\eps\vee 0)\wedge \pi$ that keeps the same limit and decreases the energy functional. 
By Lemma~\ref{lem:bdviaTheta}, using the notation of $\Theta_{\f_\eps}$ for $\f_\eps$ inside the interval $I_r$, we estimate
\[
F_\eps(\f_\eps;I_r) \ge \frac1{2\pi}\int_0^{\frac\pi 2} (1-\frac{2\gamma}{\pi})^2 \, d\Theta_{\f_\eps}(\gamma)+\frac{1}{2\pi\eps} \int_{I_r} \sin^2 \f_\eps \, dt, \quad r\in (0,1).
\]
Now for every $\gamma\in[0,\frac\pi 2]$ we have
\[
\int_{I_r} \sin^2 \f_\eps \,dt \ge \int_{\{\gamma{\leq}\f_\eps{\leq}\pi-\gamma \}} \sin^2 \f_\eps dx \ge  \left| {\{\gamma{\leq}\f_\eps{\leq}\pi-\gamma \}}\right| \sin^2 \gamma,
\]
{where the sets are understood as intersected with $I_r$.}
Averaging over $\gamma\in [0,\frac\pi 2]$ this yields
\[
\int_{I_r} \sin^2 \f_\eps \,dt \ge\frac2\pi \int_{0}^{\frac\pi 2} \left| {\{\gamma{\leq}\f_\eps{\leq}\pi-\gamma \}}\right| \sin^2 \gamma\,d\gamma.
\]
Integrating by parts, as $\Theta_{\f_\eps}(0)=0$, we have that
\[
\int_0^{\frac\pi 2} (1-\frac{2\gamma}{\pi})^2\, d \Theta_{\f_\eps}(\gamma)=\int_0^{\frac\pi 2} \frac4\pi(1-\frac{2\gamma}{\pi})\Theta_{\f_\eps}(\gamma)\,d\gamma,
\]
so we obtain
\be
\label{eusitu}
F_\eps(\f_\eps;I_r) \ge\frac1{2\pi}\int_0^{\frac\pi 2} \bigg(\frac4\pi(1-\frac{2\gamma}{\pi})\Theta_{\f_\eps}(\gamma)+\frac2{\pi\eps}\sin^2 \gamma \left| {\{\gamma{\leq}\f_\eps{\leq}\pi-\gamma \}}\right|\bigg)\,d\gamma, \quad r\in (0,1).
\ee
We set for $\gamma\in (0, \frac\pi 2)$:
$$a_\gamma^\eps=\frac{|\{t\in I_r\,:\, \f_\eps(t)<\gamma\}|}{2r}, \,  b_\gamma^\eps=\frac{|\{t\in I_r\, :\, \f_\eps(t)>\pi-\gamma\}|}{2r},\, c_\gamma^\eps=\frac{|\{t\in I_r\, :\, \gamma\le \f_\eps(t)\le \pi-\gamma\}|}{2r}.$$ Since the {integrand} in the RHS in \eqref{eusitu} is nonnegative, we use in the following only the restriction to $\gamma\in (\eps^{1/3}, \frac\pi 2)$ (which is enough to deduce the desired lower bound for $F_\eps(\f_\eps;I_r)$). This choice is motivated by the fact that
$c_{\eps^{1/3}}^\eps \le \frac{C \eps|\log\eps|}{2r\sin^2 (\eps^{1/3})} \to 0$ as $\eps\to 0$ (following from \eqref{esti_rho}); combined with the assumption $\f_\eps \to \f_*$ in $L^1(I_r)$ and the fact that
$a_{\eps^{1/3}}^\eps+b_{\eps^{1/3}}^\eps+c_{\eps^{1/3}}^\eps=1$, 
we deduce that 
$a_{\eps^{1/3}}^\eps\to \frac 12$ and $b_{\eps^{1/3}}^\eps\to \frac 12$ as $\eps\to 0$. 
Using  \eqref{eq:rearr}, we have for every $\gamma\in (\eps^{1/3}, \frac\pi 2)$:
\begin{align*}
\Theta_{\f_\eps}(\gamma) 
&\ge 2\pi^2\int_{\{\f_\eps<\gamma\}} \int_{\{\f_\eps>\pi -\gamma\}}  \frac1{|s-t|^2}\, ds \, dt \ge 2\pi^2\left(  \log a_\gamma^\eps +\log b_\gamma^\eps - \log c_\gamma^\eps\right) 
\end{align*}
so
\[
F_\eps(\f_\eps;I_r) \ge 2\int_{\eps^{1/3}}^{\frac\pi 2}\left( 2(1-\frac{2\gamma}{\pi}) \left( \log a_\gamma^\eps + \log b_\gamma^\eps - \log c_\gamma^\eps\right) + 
\frac{r}{\pi^2\eps} c_\gamma^\eps \sin^2 \gamma \right) \,d\gamma.
\]
{For every fixed $\gamma \in (0, \frac\pi 2)$}, as $a_{\gamma}^\eps\geq a_{\eps^{1/3}}^\eps$  {for $\eps\leq \eps_\gamma$}, we deduce that 
$\liminf_{\eps\to 0}a_\gamma^\eps \geq \frac 12$; idem, $\liminf_{\eps\to 0}b_\gamma^\eps \geq \frac 12$ for every $\gamma \in ({0}, \frac\pi 2)$.
Using footnote \ref{foot1}, for every $\gamma \in (\eps^{1/3}, \frac\pi 2)$ and $\eps\leq \eps_r$, we consider $K_\gamma=\frac{r\sin^2\gamma}{2\pi \eps(\pi-2\gamma)}>1$  and we obtain that
$-\log c_\gamma^\eps+K_\gamma c_\gamma^\eps\geq \log K_\gamma +1$ yielding for $\eps>0$ small enough
\[
F_\eps(\f_\eps;I_r) \ge -C+2\int_{\eps^{1/3}}^{\frac\pi2} \left(2(1-\frac{2\gamma}{\pi}) +2(1-\frac{2\gamma}{\pi})  
\log K_\gamma \right)\, d\gamma.
\]
As \[
2\int_0^{\frac\pi 2}2(1-\frac{2\gamma}{\pi}) \log \frac{r}{\eps} \, d\gamma = \pi \log \frac r\eps,
\]
$\eps^{1/3}\log \frac r\eps\to 0$ as $\eps\to 0$ and 
\[
\int_0^{\frac\pi 2} (1-\frac{2\gamma}{\pi})  \log\frac{\sin^2 \gamma }{2\pi(\pi-{2\gamma}) }\, d\gamma <\infty,
\]
we conclude to the existence of $M_0>0$ with 
\[
\liminf_{\eps\to 0} \bigg(F_\eps(\f_\eps;I_r) - \pi \log\frac{r}{\eps}\bigg)\ge -M_0.
\]

\medskip

\nd {\bf The general case of $\ell\in \ZZ$.}
For the higher-multiplicity statement, we may assume $\ell>0$ (otherwise, replace $\f_\eps$ with $-\f_\eps$) and decompose $\f_\eps=\sum_{j=0}^{\ell-1} \f_\eps^{(j)}$, where
\[
\f_\eps^{(j)} =(\f_\eps\vee j\pi)\wedge (j+1)\pi-j\pi.
\]
Using $(\f_\eps^{(j)}(t)-\f_\eps^{(j)}(s))(\f_\eps^{(k)}(t)-\f_\eps^{(k)}(s))\ge 0$ for every $t,s\in I_r$ and the $\pi$-periodicity of $\sin^2$,  we easily deduce that
\[
F_\eps(\f_\eps;I_r) \ge \sum_{j=0}^{\ell-1} F_\eps(\f_\eps^{(j)};I_r).
\]
As $\f_\eps^{(j)}\to \f_*$ in $L^1(I_r)$ as $\eps\to 0$ for $0\leq j\leq \ell-1$, we can use the case $\ell=1$ on every $\f_\eps^{(j)}$ and conclude with \eqref{eq:lolobd1} for $\ell$ general. \qed

\medskip

In the following two corollaries, we show that  \eqref{eq:lolobd1} holds without the $\liminf$, for sufficiently small $r$ and $\eps$.

\begin{cor}\label{cor:cor1pl}
Under the assumptions of Proposition~\ref{prop:lolobd}, consider sequences $r=r_k\to 0$, $\eps=\eps_k\to 0$ with $\frac{r_k}{\eps_k}\to\infty$ and $\f_k=\f_{\eps_k}$.
Then
\[
\liminf_{k\to\infty} \big(F_{\eps_k} (\f_k;I_{r_k})-\pi|\ell| \log \frac{r_k}{\eps_k}\big) \ge -|\ell|M_0.
\]
\end{cor}
\begin{proof}{}
Set $\hat \eps_k = \frac{\eps_k}{r_k}$ and $\hat \f_k(x) =\f_k (\frac x{r_k})$. Then $F_{\hat \e_k}(\hat \f_k;I_{1}) = F_{\e_k}(\f_k;I_{r_k})$ and 
$\hat \f_k \to \ell \f_*$ in $L^1(I_1)$.
By Proposition~\ref{prop:lolobd}, it follows that 
\[
\liminf_{k\to\infty} \big(F_{\eps_k} (\f_k;I_{r_k})-\pi|\ell| \log \frac{r_k}{\eps_k} \big)  = \liminf_{k\to\infty} \big(F_{\hat \e_k}(\hat \f_k;I_{1})-\pi |\ell| \log \frac1{\hat \e_k}\big) \ge -|\ell|M_0.
\] \qed
\end{proof}

\begin{cor}\label{cor:lob3}
There exist constants $M_2>0$, $\eps_0>0$, $r_0\in(0,1)$ such that  for every sequence / family $(\f_\eps)_{\eps\to 0}$ converging to $\f_*$ as in Proposition~\ref{prop:lolobd} and for all $r,\eps>0$ with $\eps<\eps_0$, $r<r_0$, 
the following holds:
\begin{equation}
\label{eq:lolob3}
F_\eps\big(\f_\eps; (-r,r)\big) - \pi |\ell| \log\frac{r}{\eps} \ge -M_2 |\ell|.
\end{equation}
\end{cor}
\begin{proof}{}
First, note that it is enough to show the existence of a universal constant $K>0$ such that the conclusion holds true in the restricted case $r>K\eps$. Indeed, the other case $r\le K\eps$ follows because then $\log\frac{r}{K\eps}\le 0$ and hence
\[
F_\eps\big(\f_\eps; (-r,r)\big)  \ge 0 \ge \pi |\ell| \log\frac{r}{K\eps} {=} \pi |\ell| \log\frac{r}{\eps}-\pi |\ell| \log K,
\]
so \eqref{eq:lolob3} is true up to replacing $M_2$ with $\max(M_2,\log K)$. 

For the existence of the constant $K$, we argue by contradiction. Assume that for $M_2=n$, $K=n$, $\eps_0=\frac1{n^3}$ and $r_0=\frac1n$ there exist a sequence $(\f_{\eps_n})_{n\to \infty}$ converging to $\f_*$ as in Proposition~\ref{prop:lolobd} and $\eps_n<\frac1{n^3}$ and $r_n\in(K\eps_n, \frac1n)$ 
with 
\[
F_{\eps_n}(\f_{\eps_n};I_{r_n})-\pi|\ell| \log \frac{r_n}{\eps_n} <-n|\ell|,
\]
then $\frac{r_n}{\eps_n}\to\infty$ but 
\[
\liminf_{n\to\infty} \left(F_{\eps_n}(\f_{\eps_n};I_{r_n})-\pi|\ell| \log \frac{r_n}{\eps_n} \right) =-\infty
\]
in contradiction to Corollary~\ref{cor:cor1pl}.
\qed
\end{proof}

\medskip

We also need the following simple but powerful lemma, a variant of an observation by del Pino and Felmer \cite{Pino:1997aa}.

\begin{lem}
\label{lem:ubpen}
For every $M_3>0$, there is $M_4:=2(M_2+M_3+\pi{\log 2}) >0$ (with $M_2$ given in \eqref{eq:lolob3}) such that
for every sequence / family $\f_\eps\to \f_*$ in $L^1((-1,1))$ as $\eps\to 0$, where $\f_*(x)=\pi$ for $x\in (-1,0)$ and $\f_*(x)=0$ for $x\in (0,1)$, that satisfies
\[
F_\eps(\f_\eps; (-r,r))\le \pi \log\frac r\eps+M_3 \quad \textrm{for every } r\in (0,1) \textrm{ and } \eps \textrm{ small},
\]
then 
\[
\limsup_{\eps\to 0} \frac{1}{2\pi\eps} \int_{-r}^r \sin^2 \f_\eps \, d\h^1 \le M_4 \quad \textrm{for every } r\in (0,1).
\]
\end{lem}
\proof{}
Let $\f_\eps\to \f_*$ in $L^1((-1,1))$. Denoting  $\tilde \f_{2\eps}:=\f_{\eps}$, we have that $\tilde \f_{2\eps} \to \f_*$
in $L^1((-1,1))$. Hence for small $\eps$, we apply Corollary~\ref{cor:lob3} for $\tilde \f_{2\eps}$ on $I_r$:
\[
F_{2\eps}(\f_{\eps};I_r)=F_{2\eps}(\tilde \f_{2\eps};I_r) \ge \pi \log\frac{r}{2\eps} -M_2
\]
so
\[
\frac{1}{4\pi \eps} \int_{I_r} \sin^2 \f_{\eps}\, d\h^1 =
F_{\eps}(\f_{\eps};I_r)-F_{2\eps}(\f_{\eps};I_r) \le \pi \log\frac r{\eps}+M_3-\pi \log\frac{r}{2\eps} +M_2,
\]
for every $r\in (0,1)$ and $\eps$ small.
\qed

\medskip

For the second order lower bound of the two-dimensional functional $\hat F_\eps^{(0)}$ defined at 
\eqref{funct_fhat},
we need the following result comparing some optimal profile problems. To simplify notation we {\bf skip $^{(0)}$ in $\hat F_\eps^{(0)}$}, i.e., we denote for an open set $G\subset \R^2_+$ and $\psi:G\to \R$ the localised functional
\be
\label{funct_sans}
\hat F_\eps(\psi;G) := \int_{G} 
|\nabla \psi|^2 \,dxdy+ \frac{1}{2\pi\eps} \int_{\overline{G}\cap(\R\times\{0\})} \sin^2\big(\psi(\cdot, 0)\big)\,dx.
\ee

\begin{lem}\label{lem:cabre}
We set $\phi^*(x,y)=\arg(x+iy)$ and $\phi^*_{\eps}(x,y)=\arg(x+i(y+2\pi \eps))$ for $(x,y)\in \R^2_+$. Setting $I_r=(-r,r)$ for $r>0$,
\[
\gamma_1 = \liminf_{\eps\to 0} \big(\inf_{\psi=\phi^* \text{on $\partial B_r^+\setminus I_r$}}\hat F_\eps(\psi;B_r^+) -\pi \log\frac{r}{\eps}\big)
\]
and 
\[
\gamma_2 = \lim_{r\to 0}\liminf_{\eps\to 0} \big(\inf_{\psi=\phi^*_\eps \text{on $\partial B_r^+\setminus I_r$}}\hat F_\eps(\psi;B_r^+) -\pi \log\frac{r}{\eps}\big),
\] then these limits 
are equal (in particular, $\gamma_1$ is independent of $r$), and moreover,
\[
\gamma_1=\gamma_2=\gamma_0:=\pi + \pi\log\frac{1}{4\pi} = \pi\log\frac{e}{4\pi}.
\]
\end{lem}
\proof{}
We remark that in the definition of $\gamma_1$, we can scale out $r$ if we replace $r$ by $1$ and $\eps$ by $\eps/r$ without changing the result, 
so the 
limit is in fact independent of $r$, i.e., $\gamma_1$ is independent of $r$. The harmonic  function $\phi^*_\eps$ is Peierls' solution of the Euler-Lagrange equations for $\hat F_\eps$ (see Toland~\cite{Toland:1997a}).

\medskip

\nd {\bf Step 1}. {\it We show that $\gamma_1 =\gamma_2$.}
For that, we construct comparison functions $\phi_\eps$ on $B^+_{r(1+r)}\setminus B_r$ for some $r>0$ that satisfy
$\phi_\eps=\phi^*_\eps$ on the half-circle $\partial B^+_{r(1+r)}\setminus I_{r(r+1)}$ and $\phi_\eps=\phi^*$ on $\partial B^+_{r}\setminus I_{r}$. 
For example, we can choose an interpolation function such as 
\[
\phi_\eps(x,y)=\arg(x+i (y+2\pi\eps \frac{\sqrt{x^2+y^2}-r}{r^2})), \quad (x,y)\in B^+_{r(1+r)}\setminus B_r.
\]
As both the argument function and the function multiplied by $\e$ are smooth away from $0$, it is straightforward to see that
$$\lim_{\eps \to 0} \int_{B^+_{r(1+r)}\setminus B_r} |\nabla \phi_\eps|^2 \, dxdy=
\int_{B^+_{r(1+r)}\setminus B_r} |\nabla \arg(x+iy)|^2 \, dx dy=\int_0^\pi \int_r^{r(r+1)} \frac1 s \, ds d\theta=\pi \log(1+r)$$
and $\sin^2 \phi_\eps(x,0)\le {\sin^2 \phi_\eps(r(r+1),0) }\leq C(\frac \eps r)^2 $ for $x\in I_{r(1+r)}\setminus I_r$,
 so  letting first $\eps\to 0$ and then $r\to 0$
 it follows that $\gamma_2\le \gamma_1$. The opposite inequality follows from a similar interpolation argument.

\medskip

\nd {\bf Step 2}. {\it We compute that $\gamma_1 =\gamma_2=\pi\log\frac{e}{4\pi}$.}
To identify the limit, we use a result of Cabr\'e and Sol\`a-Morales \cite[Lemma 3.1]{CabreSola-Mor:2005a} that states
that  $\phi^*_\eps$  is not only a critical point of $\hat F_\eps$, but actually the minimiser of $\hat F_\eps$ with respect to 
its own boundary conditions, i.e., $\phi^*_\eps$ is the minimiser inside the limit $\gamma_2$. Therefore, we compute explicitly the energy of $\phi^*_\eps$. First, note that by rescaling 
$\psi(z):=\phi_\eps^*(2\pi \eps z)$ for $z=(x,y)\in \R^2_+$, we have that
$\psi(x,y)={\frac\pi 2-\arctan\frac{x}{y+1}}$. For $R=\frac{r}{2\pi \eps}$ we then have
\[
\int_{B_R^+} |\nabla \psi|^2 \, dxdy = \int_{B_r^+} |\nabla \phi_\eps^*|^2 \, dxdy 
\]
and 
\[
\int_{-R}^R \sin^2 \psi\, dx = \frac{1}{2\pi \eps} \int_{-r}^r\sin^2 \phi_\eps^* \, dx.
\]
By direct calculation, $|\nabla \psi(x,y)|^2 = \frac{1}{x^2 + (y+1)^2}$ and changing variables we obtain
\[
\int_{B_R^+}\frac{1}{x^2+(y+1)^2} dx dy = \int_{B_R(0,1) \cap \{y>1\}} \frac{1}{x^2+y^2} dxdy.
\]
Setting $A_R= B_R^+\cap \{ y>1\}=\{(x,y): x^2 + y^2<R^2, y>1\}$, we clearly have for $R>1$:
\[
A_R\subset B_R(0,1){\cap \{y>1\}}\subset A_{R+1}. 
\]
Using polar coordinates $x=s\cos \theta$, $y=s\sin\theta$ in $A_R$, we have that $y>1$ corresponds to $\sin\theta>\frac1s$ and $s>1$ (as $s>y$) so
\[
\int_{A_R} \frac{1}{x^2+y^2}\,  dxdy = \int_1^R \int_{\arcsin \frac1s}^{\pi -\arcsin\frac1s}\frac 1s  \, d\theta ds.
\] 
Evaluating the $\theta$-integral and changing variables $s=\frac1{\sin t}$ we see 
\[
\int_{A_R} \frac{1}{x^2+y^2} \, dxdy  = \int_1^R\left( \frac\pi s-\frac{2\arcsin \frac1s}{s}\right) \, ds= \pi \log R
-2\int_{\arcsin \frac1R}^{\frac\pi 2} t \cot t\,  dt .
\]
We note that $\int_0^{\arcsin\frac1R} t \cot t \, dt =O(\frac1R)$ as $R$ is large and integrate by parts:
\[
\int_0^{\frac\pi 2} t\cot t dt = \int_0^{\frac\pi 2} t \frac{d}{dt}(\log \sin t)\, dt = -\int_0^{\frac\pi 2} \log\sin t \, dt=\frac\pi 2 \log 2,
\]
where the final equality is a standard integral, see Gradshteyn-Ryzhik \cite[3.747]{GradshteRyzhik:2007a}.

We thus have $\int_{A_R} |\nabla \psi|^2 dxdy = \pi \log R -\pi \log 2-O(\frac1R)$ and so, using $\log(R+1)-\log R=O(\frac1R)$ as $R$ is large 
that
\[
\int_{B_R^+} |\nabla \psi|^2 dxdy = \pi \log R -\pi \log 2-O(\frac1R).
\] 
For the boundary term, we calculate
\begin{displaymath}
  \int_{-R}^R \sin^2\psi\, dx = \int_{-R}^R \frac{1}{1+x^2} \, dx = 
2\arctan R = \pi- O(\frac1R)
\end{displaymath}
as $R\to\infty$.
Putting everything together we see that 
\[
\hat F_\eps(\phi_\eps^*;B_r^+) = \pi \log \frac{r}{2\pi \eps} - \pi \log 2 + \pi -O({\frac  \eps r}) = \pi \log\frac r \eps + \pi + \pi \log \frac{1}{4\pi}-O(\frac  \eps r),
\]
and passing to the limit $\eps\to 0$ and then $r\to 0$ we obtain that $\gamma_2=\gamma_0$ as claimed.
\qed

\medskip

Lemma~\ref{lem:cabre} clearly applies to boundary vortices of multiplicities $\pm 1$ by suitable sign change. 
For higher {multiplicity} transitions, we have the following result.

\begin{lem}\label{lem:ubhigher}
Let $d>0$ be an integer and set $\phi^*_d(x,y)=d\arg(x+iy)$ for every $(x,y)\in \R^2$.
{For every small $r>0$ and $\eps<e^{-1/r^2}$},
  there exists $\phi_{d,\eps}:B_r^+\to \R$ such that $\phi_{d,\eps}=\phi^*_d$ on $\partial B_r$ and 
\[
\hat F_\eps(\phi_{d,\eps};B_r^+) \le \pi d \log\frac r\eps+ Cd^2 (1+|\log r| + \log | \log \eps|) 
\]
{where $C>0$ is independent of $r$ and $\eps$.}
\end{lem}
\begin{proof}{}
{The idea of the proof is to replace a near-jump of $d\pi$ at $0$ by $d$ near-jumps of height $\pi$ at points $x_\eps^j$ that all converge to $0$ and to estimate their interaction energy.}
Set $a_\eps=\frac1{|\log\eps|}$ and $x_\eps^j=j a_\eps$, $j=1,\dots, d$. 
With the interpolation function
\[
f(x,y)=\begin{cases}
1 & \textrm{ if }\sqrt{x^2+y^2}<{r(1-r)}\\
{\frac{r-\sqrt{x^2+y^2}}{r^2}}  & \textrm{ if } {r(1-r)\le\sqrt{x^2+y^2}\le r}\\
0 & \textrm{ if } \sqrt{x^2+y^2}>{r}
\end{cases}
\]
we set
\[
\phi_{d,\eps} = \sum_{j=1}^d \arg\big(x-f(x,y)x_\eps^j+i(y+2\pi \eps f(x,y)) \big). 
\]
As in the {proof of Lemma~\ref{lem:cabre}, the interpolation function does not contribute much to the energy, in fact }
\[
{\int_{B_r^+\setminus B_{r(1-r)}} |\nabla \phi_{d,\eps}|^2 \, dxdy + \frac1\e \int_{I_r\setminus I_{r(1-r)}} \sin^2 \phi_{d,\eps}\, dx\le Cd^2 \left(|\log(1-r)|+
\frac{\e}{r^2}\right) \le Cd^2.}
\]
{It suffices to compute the energy of $\phi_{d,\eps}$ in $B_\rho^+$ for $\rho=r(1-r)$, where $f\equiv 1$.} For that, we note that 
 $\phi(x,y)=\arg(x+i(y+2\pi \eps))$ and $\psi(x,y)=\log | x+i(y+2\pi \eps)|$ are (up to sign) harmonic conjugates. Then
\[
|\nabla \phi_{d,\eps}|^2 = \sum_{j=1}^d  |\nabla \psi(\cdot-x_\eps^j)|^2
+\sum_{j\neq k} \nabla \psi(\cdot-x_\eps^j)\cdot \nabla \psi(\cdot-x_\eps^k).
\]
The integral over $B_\rho^+$ of the first sum is bounded by 
\[
d\pi \log\frac{\rho}{\eps}+O(1),
\]
while for the second part we compute
\[
\int_{B_\rho^+} \nabla \psi(\cdot-x_\eps^j)\cdot \nabla \psi(\cdot-x_\eps^k) \, dx 
= \int_{\partial B_\rho^+} \psi(\cdot-x_\eps^j)\frac{\partial}{\partial\nu} \psi(\cdot-x_\eps^k).
\]
The integrals over $\partial B_\rho\cap \R^2_+$ are estimated by $\pi|\log \rho|+O(a_\e^2)/\rho^2\le \pi |\log \rho|+C$,
while the integrals over the 
straight part are of the form
\[
\int_{-\rho}^\rho \frac12\log(x-x^j_\eps)^2 \frac{2\pi \eps}{(x-x^k_\eps)^2+(2\pi\eps)^2}\, dx. 
\]
Extending the integration interval to $(-\infty,\infty)$ provides an upper bound (up to the contribution of the region where the logarithm is negative, which is 
bounded by $C|\log \rho|$ since $\eps<\rho^2$). The remaining integral,
\[
\int_{-\infty}^\infty \frac12\log(x-x^j_\eps)^2 \frac{2\pi \eps}{(x-x^k_\eps)^2+(2\pi\eps)^2}\, dx 
\]
can be evaluated using the residue theorem:
The function can be extended to the upper half plane as 
\[
\left(\log |z-x^j_\eps|+i \arg(z-x^j_\e) \right)\frac{2\pi \eps}{(z-x^k_\eps)^2+(2\pi\eps)^2}
\]
for a branch of the argument that is smooth on the upper half plane. Integrating over $\partial (B_R^+\setminus B_s(x_j^\eps))$ and letting $s\to 0$ and $R\to\infty$,
we find that 
the only singularity in the contour is a simple pole at $z=x^k_\eps+i2\pi \eps$,
and we obtain after taking real parts
\[
\int_{-\infty}^\infty \frac12\log(x-x^j_\eps)^2 \frac{2\pi \eps}{(x-x^k_\eps)^2+(2\pi\eps)^2}\, dx =2\pi \log \big((j-k)^2a_\eps^2 + (2\pi\eps)^2\big).
\]

From $\sin^2(x+y)\le 2(\sin^2 x+\sin^2 y)$ we see that
 $\sin^2(\phi_{d,\eps})\le C{d}\sum_{j=1}^d \sin^2 (\phi(\cdot-x^j_\eps))$ and using the calculation in Step 2 of the previous lemma, the boundary term contributes only
{by} a constant, and adding up we arrive at the conclusion of the lemma {since $|\log \rho-\log r |=|\log(1-r)|\le C$.}\footnote{The proof sketched above is fully local. A nonlocal proof {of a less precise estimate} is given in \cite{Kurzke:2006a}.}
\qed
\end{proof}

\medskip

Now we show a precise estimate which is the central step in the $\Gamma$-expansion beyond the leading (logarithmic) order proved at \eqref{eq:abslibd}. This is based on an
argument that is new in the context of boundary vortices, inspired by the work of  Colliander-Jerrard \cite{ColliandJerrard:1999a} for interior vortices.
A different proof of the same result (due to   Alicandro-Ponsiglione \cite{Alicandro:2014aa}) uses a dyadic decomposition argument; we expect that such an approach can also be 
used here.

\begin{pro}\label{prop:onevortexGC}
Let $\rho>0$ and $\phi_\eps\in H^1(B_\rho^+)$ {be a sequence / family} with $\phi_\eps(x,0)\to \phi^*(x,0)=\pi {\bf 1}_{\{x<0\}}(x) $ in $L^1(I_\rho)$. 
 For the functional \eqref{funct_sans}, we have the following second order lower bound:
\[
\liminf_{\eps\to 0}\left( \hat F_\eps(\phi_\eps;B_\rho^+)-\pi \log\frac \rho\eps \right)\ge \gamma_0,
\]
where $\gamma_0=\pi\log\frac{e}{4\pi}$.
\end{pro}
\proof{} {As our statement is about the $\liminf$, it is enough to consider sequences in the following.}
First note the invariance of the desired estimate with respect to rescaling in $\rho$. Therefore, it suffices to consider the case $\rho=1$. Let $\delta\in (0, \pi)$. We let $C_j$ denote generic positive constants {\bf independent of $\eps$, $\phi_\eps$ and $\delta$.}
We may assume that $\phi_\eps$ are $C^1$ smooth in $\overline{B_{1}^+}$, since for any $\eta>0$ and $\phi_\eps\in H^1(B_{1}^+)$ there exists $\hat \phi_\eps\in C^1(\overline{B_{1}^+})$ with 
$\left|\hat F_\eps(\hat \phi_\eps;B_{1}^+) - \hat F_\eps(\phi_\eps;B_{1}^+)\right| \le\eta$.

Second, we may assume that there is $C_1>0$ such that 
\[
\hat F_\eps(\phi_\eps;B_1^+)\le\pi \log\frac 1\eps+\gamma_0+C_1
\]
(otherwise the desired estimate is trivially satisfied).

\medskip

\nd  {\bf Step 1.} {\it Finding a radius $\rho_*=\rho_*(\delta)\in (0, \frac12)$ such that for all $\eps$ along a sequence, $\phi_\eps(\rho_*e^{i\theta})$ has similar properties on $\partial B_{\rho_*}^+$ to that of the limit function $\phi^*(\rho_*e^{i\theta})=\theta$ where $\theta$ is the polar angle. } For that, we start by recalling from Proposition \ref{prop:lolobd} that for every $r_0\in(0,\frac12)$:
\[
\liminf_{\eps\to 0} \left( \hat F_\eps(\phi_\eps;B_{r_0}^+) - \pi \log\frac{r_0}{\eps} \right)\ge -M_0.
\]
Combining the two estimates, we obtain for a constant $C_2>0$ independent of $r_0$:
\[
\limsup_{\eps\to 0}\hat F_\eps(\phi_\eps; B_{1}^+\setminus B_{r_0}) \le \pi \log \frac{1}{r_0} + C_2,
\]
and reducing the domain of integration and setting $C_3=C_2-\pi \log\frac12>0$ we can write
\be
\label{123}
\limsup_{\eps\to 0}\hat F_\eps(\phi_\eps; B_{1/2}^+\setminus B_{r_0}^+) \le \pi \log \frac{1}{2r_0} +  C_3.
\ee
For $s\in (0, \frac12)$, we introduce
\begin{equation}\label{eq:kleinfeps}
f_\eps(s):= \int_{\partial B^+_s\setminus I_s} |\nabla \phi_\eps|^2 d\h^1 + \frac{1}{2\pi\eps}\int_{\partial I_s} \sin^2 
\phi_\eps(\cdot,0) \, d\h^0,
\end{equation}
{so that $\hat F_\eps(\phi_\eps; B_{r}^+)=\int_0^r f_\eps(s) \, ds$},
as well as the sets
$$
A_\eps=\big\{ s \in (0, \frac12)\, : \, f_\eps(s)\le \frac{\pi + \delta}{s}\big\}
\, \textrm{ and } \,
G_\eps = \big\{ s\in (0, 1)\, : \, |\phi_\eps(s,0)| + |\phi_\eps(-s,0)-\pi|<\frac14\big\},
$$
We fix $r_0$ such that
$$r_0=r_0(\delta)\le \frac12\exp(-\frac{2C_3}{\delta})$$ 
and 
\be
\label{epsi}
\textrm{\bf from now on, $\eps$ is small, i.e., 
$\eps<r_0$}.
\ee
Thus, we have  $\delta \log\frac{1}{2r_0} \ge 2 C_3$. 
The aim of this step is to show that 
$$[r_0,\frac12]\cap A_\eps
\cap G_\eps\neq\emptyset$$ 
{(any point $\rho_*$ in this intersection can be used as the desired radius in the claim of Step 1).}
To do so, we estimate 
$a_\eps=|[r_0,\frac12]\cap A_\eps|$ 
as follows: as $s\mapsto \frac1s$ is decreasing in $(0,\frac12)$, we may estimate
\[
\pi \log\frac{1}{2r_0} + C_3 \stackrel{\eqref{123}}{\ge} \int_{r_0}^\frac12 {f_\eps}(s) \, ds \ge  (\pi + \delta) \int_{r_0+a_\eps}^{1/2} \frac1s \, ds = (\pi+\delta) \log \frac{1}{2(r_0+a_\eps)}.
\]
Using our choice of $r_0$, it follows that
$$
-C_3 \ge C_3-\delta \log\frac{1}{2r_0}\geq (\pi+\delta) \log\frac{r_0}{r_0+a_\eps},
$$
so for every $0<\delta<\pi$, we can estimate
$$
a_\eps\ge r_0\left(e^{\frac{C_3}{\pi+\delta}}-1\right) \geq r_0 C_5, \quad C_5:=e^{\frac{C_3}{2\pi}}-1>0.
$$

Choosing a sequence 
$\eps_n\to 0$, 
we have $\big|G_{\eps_n}\cap[r_0,\frac12]\big|\to \frac12-r_0$ and hence 
 (using Fatou's lemma) 
that $\left|[r_0,\frac12]\cap \limsup_{n\to\infty} (A_{\eps_n}
\cap G_{\eps_n})\right|>0$. 
In particular there is a radius $\rho_*=\rho_*(\delta)\in[r_0,\frac12]$ that lies in infinitely many sets $A_{\eps_n}
\cap G_{\eps_n}$. In particular, $\rho_*>\eps$.

\medskip

\noindent{\bf Step 2}. {\it We show that $\phi_\eps(\rho_*e^{i\theta})$ is close to $\phi^*(\rho_*e^{i\theta})=\theta$
in $L^2(\partial B^+_{\rho_*})$.} Indeed, setting $$w_\eps(\theta):=\phi_\eps(\rho_* e^{i\theta})-\theta,$$ where $\theta$ is the polar angle, we have $|w_\eps(\theta=0)|, |w_\eps(\theta=\pi)| <\frac14$ (since $\rho_*\in G_\eps$). Since $x/\sin x$ is increasing on $(0, \frac14)$, there exists  $C_7>0$ such that
\be
\label{33}
|w_\eps|\le \frac{1/4}{\sin\frac14} |\sin w_\eps|\le 
\frac{1/4}{\sin\frac14} \sqrt{\frac{4\pi^2\eps}{\rho_*}}=:C_7\sqrt{\frac\eps{\rho_*}} \quad \textrm{at } \theta\in \{0, \pi\}
\ee
(since $\rho_*\in A_\eps$)  so
\begin{align}
\nonumber
\int_0^\pi |\partial_\theta w_\eps(\theta)|^2  d\theta &= \int_0^\pi \left( |\partial_\theta \phi_\eps(\rho_*e^{i\theta})|^2 + 2 \partial_\theta \phi_\eps (\rho_*e^{i\theta}) + 1\right) \,d\theta \\
\nonumber
&= \int_0^\pi \left( |\partial_\theta \phi_\eps(\rho_*e^{i\theta})|^2 -1+ 2 \partial_\theta w_\eps  \right)\, d\theta \\
\nonumber
& \le   \delta + 2 \int_0^\pi\partial_\theta w_\eps\,d\theta = \delta + 2 (w_\eps(\pi)-w_\eps(0))\\
\label{22}
&\le \delta + 4C_7\sqrt{\frac\eps{\rho_*}}.
\end{align}
In particular, for a suitably chosen $C_8>0$, we obtain thanks to \eqref{epsi} and \eqref{33} (in particular, $\eps<\rho_*$):
\be
\label{44}
\int_0^\pi |w_\eps|^2 d\theta \le \int_0^\pi \left(w_\eps(0)+\int_0^\theta \partial_\theta w_\eps(y) dy\right)^2 d\theta \le C_8\bigg(\delta+\sqrt{\frac\eps{\rho_*}}\, \bigg).
\ee

\noindent{\bf Step 3}. {\it We prove that}
\begin{equation}\label{eq:onehol}
\liminf_{\eps\to 0}\big(\hat F_\eps( \phi_\eps;B^+_{\rho_*}) -\pi \log\frac{\rho_*}{\eps}\big) \ge \gamma_0 -o_{\delta}(1).
\end{equation}
The idea is to estimate the energy of the interpolation between  $\phi_\eps$ and $\phi^*=\theta$ in a small annulus around $\partial B^+_{\rho_*}$. In the small annulus $B_{\rho_*+\eta}^+\setminus B_{\rho_*}$
with $\eta$ to be chosen later (see \eqref{eta_ici}), we set the interpolation function between $\phi_\eps(\rho_*e^{i\theta})$ and $\phi^*\big((\rho_*+\eta)e^{i\theta}\big)=\theta$:
\[
\hat \phi_\eps(r,\theta)= \theta + \frac{\rho_*+\eta-r}{\eta} w_\eps(\theta), \quad r\in (\rho_*, \rho_*+\eta), \, \theta\in (0, \pi).
\]
Then we estimate the energy of $\hat \phi_\eps$: 
\begin{align*}
\hat F_\eps(\hat \phi_\eps; B_{\rho_*+\eta}^+\setminus B_{\rho_*} ) 
&= \int_{\rho_*}^{\rho_*+\eta} 
\Biggl( 
\int_0^\pi 
\frac1r 
\left({1} +\frac{\rho_*+\eta-r}{\eta} \partial_\theta w_\eps\right)^2+ \frac{r}{\eta^2}|w_\eps|^2 \, d\theta
+\\
&\quad+\frac{1}{2\pi\eps} 
\left( \sin^2 \left(\frac{\rho_*+\eta-r}{\eta} w_\eps(0)\right)+\sin^2 \left(\frac{\rho_*+\eta-r}{\eta} w_\eps(\pi)\right)
\right)
\Biggr)
dr
\end{align*}
For the first term in the above RHS, we use \eqref{33} and \eqref{22} to estimate: 
 \begin{align*}
  \int_{\rho_*}^{\rho_*+\eta} 
\int_0^\pi &
\bigg( \frac1r +\frac{2(\rho_*+\eta-r)}{r\eta} \partial_\theta w_\eps+ \frac1r(\partial_\theta w_\eps)^2\bigg)\, d\theta dr \\
&\leq \pi \log(1+\frac{\eta}{\rho_*})+ \int_{\rho_*}^{\rho_*+\eta} \frac2r |w_\eps(\pi)-w_\eps(0)|+
\frac{1}{r} \big(\delta + 4C_7\sqrt{\frac\eps{\rho_*}}\, \big)\, dr\\
&\leq \log(1+\frac{\eta}{\rho_*}) \bigg(\pi + 8C_7\sqrt{\frac\eps{\rho_*}}+\delta\bigg).
 \end{align*}
Since $\rho_*\in A_\eps$, this estimate combined with \eqref{44} yield
  \begin{align*}
 \hat F_\eps(\hat \phi_\eps; B_{\rho_*+\eta}^+\setminus B_{\rho_*} ) 
& \le  \log(1+\frac{\eta}{\rho_*})\left( \pi+8C_7\sqrt{\frac\eps{\rho_*}}+\delta\right) +
    C_8(\delta+\sqrt{\frac\eps{\rho_*}})(\frac{\rho_*}{\eta}+1)
+  \frac{2\pi\eta}{\rho_*}.
\end{align*}
Letting $\eps\to 0$ and setting 
\be
\label{eta_ici}
\eta=\delta^{1/4}\rho_*,
\ee 
we obtain that 
\begin{equation}\label{eq:kldelta}
 \limsup_{\eps\to 0}\hat F_\eps(\hat \phi_\eps; B^+_{\rho_*+\eta}\setminus B_{\rho_*})
 \le (\pi+\delta)\log(1+\delta^{1/4}) + C_8 (\delta^{3/4}+\delta)+ 2\pi \delta^{1/4},
\end{equation}
which tends to $0$ as $\delta\to 0$. If we extend $\hat \phi_\eps$ in the ball $B^+_{\rho_*+\eta=\rho_*(1+\delta^{1/4})}$ by setting $\hat \phi_\eps:=\phi_\eps$ in $B^+_{\rho_*}$, we can now use the lower bounds from the definition of $\gamma_1$ in Lemma \ref{lem:cabre} (because $\hat \phi_\eps=\phi^*$ on $\partial B^+_{\rho_*+\eta}\setminus I_{\rho_*+\eta}$), giving us 
\[\liminf_{\eps\to 0} \big(\hat F_\eps(\hat \phi_\eps;B^+_{\rho_*(1+\delta^{1/4})}) -\pi \log\frac{\rho_*}{\eps}\big) \ge \gamma_0-o_\delta(1).
\]
Since $\phi_\eps=\hat \phi_\eps$ on $B_{\rho_*}^+$, we can use \eqref{eq:kldelta} and obtain \eqref{eq:onehol}
(recall that $\rho_*$ depends on $\delta$, that's why the last term $o_\delta(1)$ is needed in \eqref{eq:onehol}).

\medskip

\noindent {\bf Step 4}. {\it We prove the optimal lower bound in the outer annulus}
$$
\liminf_{\eps\to 0}\hat F_\eps(\phi_\eps;B^+_1\setminus B_{\rho_*}) \ge \pi \log\frac{1}{\rho_*}.
$$
In fact, we prove the following more general case that is needed in the proof of Theorem \ref{thm:GCforGeps}: 

\medskip

\nd {\bf Claim:} If $\ell\in \ZZ$ and $\phi_\eps\in H^1(B_1^+)$ with $\phi_\eps(x,0)\to \ell \phi^*(x,0)$ in $L^1((-1,1))$ as $\eps\to 0$, then   
\be
\label{estim_outer}
\liminf_{\eps\to 0}\hat F_\eps(\phi_\eps;B^+_1\setminus B_{\rho_*}) \ge \pi \ell^2\log\frac{1}{\rho_*},
\, \textrm{ for every } \rho_*\in (0,1).
\ee
For that, we start by fixing  $\rho_*\in (0,1)$ and focusing on the set 
$$S_\eps =\left \{ s\in (\rho_*, 1)\, : \,|\phi_\eps(-s,0)-\ell\pi|+ |\phi_\eps(s,0)|<\frac14\right\}.$$
It is clear that $|S_\eps|\to 1-\rho_*$ as $\eps\to 0$, since $\phi_\eps\to \ell \phi^*$ in $L^1(I_1)$. By H\"older's inequality,
\[
|\phi_\eps(-r, 0)-\phi_\eps(r,0)|\le \int_0^\pi |\partial_\theta \phi_\eps(re^{i\theta})|\, d\theta \le \left(\int_0^\pi |\partial_\theta \phi_\eps(re^{i\theta})|^2 \, d\theta\right)^{1/2} \pi^{1/2}, \quad r\in (0,1),
\]
so  using $f_\eps$ defined in \eqref{eq:kleinfeps}, we estimate 
\[
f_\eps(r) \ge \frac1 {\pi r} \bigl(\phi_\eps(r,0)-\phi_\eps(-r,0)\bigr)^2 + \frac{1}{2\pi\eps} \left(\sin^2 \phi_\eps(r,0)+\sin^2 \phi_\eps(-r,0)\right), \quad r\in (0,1).
\]
If we restrict to $r\in S_\eps$, there is a constant $C_9>0$ such that 
\begin{align*}
\sin^2 \phi_\eps(r,0) + \sin^2 \phi_\eps(-r,0) &\ge 2 C_9 \bigl((\ell\pi-\phi_\eps(-r,0))^2 + \phi_\eps(r,0)^2 \bigr)\\
& \ge C_9 \bigl(\ell\pi-\phi_\eps(-r,0)+\phi_\eps(r,0)\bigr)^2, \quad r\in (\rho_*,1)
\end{align*}
so
\[
f_\eps(r) \ge \inf_{s\in \R} \left( \frac{s^2}{\pi r} + C_9\frac{(\ell\pi-s)^2}{2\pi \eps}\right), \quad r\in S_\eps.
\]
Optimising over $s$, we obtain for a constant $C_{10}>0$: \footnote{The function $h(s)=\frac{s^2}{\pi r} + C_9\frac{(\ell\pi-s)^2}{2\pi \eps}$ is a parabola having the minimum $s_*$ satisfying $s_*=C_9r(\ell \pi-s_*)/(2\eps)$ which yields ${h}(s_*)=\ell s_*/r=\pi \ell^2/(r+\frac{2\eps}{C_9})$.}
\[
f_\eps(r)\ge \frac{\pi \ell^2}{r+C_{10}\eps}, \quad r\in S_\eps
\]
yielding
\begin{align*}
\hat F_\eps(\phi_\eps;B^+_1\setminus B_{\rho_*}) &\ge \int_{S_\eps} \frac{\pi \ell^2}{r+C_{10}\eps} \, dr\\
& \ge \int_{1-|S_\eps|}^1  \frac{\pi\ell^2}{r+C_{10}\eps}\,  dr
= \pi \ell^2\log \frac{1+C_{10}\eps}{1-|S_\eps|+C_{10}\eps}.
\end{align*}
Letting $\eps\to 0$, as $|S_\eps|\to 1-\rho_*$, this yields \eqref{estim_outer} and proves the claim.

Combining \eqref{estim_outer} and \eqref{eq:onehol} (in the case $\ell=1$), we obtain
\[
\liminf_{\eps\to 0} \left(\hat F_\eps(\phi_\eps;B_1^+) -\pi \log\frac1\eps \right)\ge \gamma_0-o_{\delta}(1),
\]
so letting $\delta\to 0$ we obtain the desired conclusion.
\qed

We need the following estimate, which is closely related to a result
from Struwe \cite{Struwe:1994a}.
\begin{lem}\label{lem:lpbd}
Let $f\in L^2(B^+_1)$ be a function on the unit half disk $B^+_1\subset\R^2$
with the following property: There exists $r_0<1$ and $A>0$ such that
for every $0<r\le r_0$,
\[
\|f\|^2_{L^2(B^+_1\setminus B_r)} \le A (1+\log\frac1r).
\]
Then for $1\le q<2$ we have
\[
\|f\|_{L^q(B^+_1)} \le C(A,q,r_0)<\infty,
\]
where $C(A,q,r_0)$ is independent of $f$.
\end{lem}
\proof{} Let $1\le q<2$. 
Using H\"older's inequality (as $q<2$), we calculate for $r_j=2^{-j}r_0$:
\begin{align*}
\int_{B^+_{r_0}} |f|^q dx 
& = \sum_{j=0}^\infty \int_{B^+_{r_j}\setminus B_{r_{j+1}}} (|f|^2)^{\frac{q}{2}}  \, dx
\\
&\le\sum_{j=0}^\infty \left(\int_{B^+_{r_j}\setminus B_{r_{j+1}}} |f|^2dx\right)^{\frac q2}
|{B^+_{r_j}\setminus B_{r_{j+1}}}|^{1-\frac q2}\\
&\le \sum_{j=0}^\infty
\left(\int_{B^+_1\setminus B_{r_{j+1}}} |f|^2dx\right)^{\frac q2} (\frac\pi 2)^{1-\frac q2}(2^{-j}r_0)^{2- q}\\
&\le C\sum_{j=0}^\infty (1+j\log 2-\log r_0)^{\frac q2} 2^{-(2-q)j}.
\end{align*}
The sum converges by the root test
so $ \|f\|_{L^q(B^+_{r_0})}\le C(A,q,r_0)$. We also clearly have that $\|f\|_{L^q(B^+_1\setminus B_{r_0})} \le C(A,q,r_0)$.\qed

\proof{ of Theorem~\ref{thm:GCforGeps}} We divide the proof in several steps:

\medskip

\noindent {\bf Step 1}. {\it Proof of point 1.} For small $r>0$, using Lemma \ref{lem:conformalflat} (and the notation therein), we can cover a neighbourhood of 
$\dOm$ with a finite number of patches $A_j=\Psi_{p_j}({B_{r(1-c_1 r\log\frac1r)}^+\cup I_{r(1-c_1 r\log\frac1r)}})\subset B_r(p_j)\cap \overline{\Omega}$ for a finite set of points $p_j\in \dOm$, such that $\cup A_j$ is relatively open 
in $\overline{\Om}$ and the functions $\psi_\eps^{(j)}=\phi_\eps\circ\Psi_{p_j}$ satisfy the energy estimate for the functionals \eqref{funct_fhat}:
\[
\limsup_{\eps\to 0} \frac1{|\log\eps|}\hat F^{(g^{(j)})}_{\eps}(\psi_\eps^{(j)};B^+_\rho) <\infty
\]
where we have denoted
 $g^{(j)}=g\circ \Psi_{p_j}$ for $g$ a lifting as given in \eqref{g} and $\rho=r(1-c_1r\log\frac1r)$. 
 On each patch, we arrange $g^{(j)}$ to be continuous.
From \eqref{eq:l7roughbd} in Lemma \ref{lem:lem7}, we find that the functions $w_\eps^{(j)}=\psi_\eps^{(j)}-\widehat{g^{(j)}}_\rho$ then satisfy
\[
\limsup_{\eps\to 0} \frac1{|\log\eps|}\hat F^{(0)}_{\eps}(w_\eps^{(j)};B^+_\rho) <\infty.
\]
We can now use Lemma~\ref{lem:ABS2d1d} to reduce $\hat F^{(0)}_\eps$ to $F_\eps$ defined at \eqref{def:feps} and apply Proposition~\ref{prop:orliczbd} to see that for a {sequence / family} $z^{(j)}_\eps\in \ZZ$, 
$w^{(j)}_\eps(0, \cdot)-\pi z^{(j)}_\eps$ are bounded in $L^p(I_{\rho})$ for every $p\in [1, \infty)$. By Proposition~\ref{prop:strcomp}, we have up to a subsequence the $L^p$ convergence
 $w_\eps^{(j)}(0, \cdot)-\pi z^{(j)}_\eps\to w_0^{(j)}\in BV(I_\rho;\pi\ZZ)$. 
Changing variables, we obtain convergence for $\phi_\eps-\pi z^{(j)}_\eps$ in $L^p(A_j\cap
\partial\Om)$.
 If $\partial\Omega \cap A_j\cap A_j\neq \emptyset$, it
follows that $z^{(j)}_\eps-z^{(j)}_\eps(\in \ZZ)$ converges as $\eps\to 0$, i.e., it is constant for small $\eps$; in particular, 
we may choose a subsequence $z_\eps\in \ZZ$ that works for all of the patches $A_j$. Adding up the results on the patches,
it follows that  $(\phi_\eps-{\pi z_\eps})_\eps$ is bounded and converges as claimed in $L^p(\dOm)$ for every $p\in [1, \infty)$ to a limit function $\phi_0$ on $\dOm$
that satisfies $\phi_0-g\in BV(\dOm;\pi\ZZ)$. Furthermore, $\de_\tau \phi_\eps\to \de_\tau \phi_0$ in $W^{-1,p}$. 
With $\ka=[\de_\tau g]_{ac}$   we obtain that $\de_\tau\phi_0-\ka = -\pi \sum_{j=1}^N{d_j \delta_{a_j}}$, where the $a_j$ can be chosen mutually distinct.
The measures $\de_\tau \phi_\e$ all average to zero, so $\de_\tau \phi_0$ does as well, and we must have that $\sum d_j =2$
(due to the Gau\ss{}-Bonnet theorem in the proof of Lemma~\ref{lem:lifting}).
To show the lower bound, we consider for small $r>0$ disjoint patches $A_j$ as above, centred at $a_j$.
Defining $w_\eps^{(j)}$ as above and setting $\rho=r(1-c_1r\log\frac1r)$ and $\tilde\eps=\frac{\eps}{1-c_1r\log\frac1r}$, 
the results of  Lemma~\ref{lem:conformalflat}, the convergence of $\phi_\e$ in $L^2(\dOm)$ and \eqref{eq:36} imply 
\begin{equation}\label{eq:s40}
\int_{B_r(a_j)\cap\Om} |\nabla \phi_\e|^2 \, dx + \frac1{2\pi \e} \int_{B_r(a_j)\cap\dOm} \sin^2(\phi_\e-g)\, d\h^1 \ge \hat F^{(0)}_{\tilde \eps}(w_\eps^{(j)};B^+_\rho) -Cr^{\frac12}.
\end{equation}
In $I_\rho$, we have $w_\eps^{(j)}(0, \cdot)\to w_*^{(j)}$, where $w_*^{(j)}$ is locally constant except for a single jump of height $d_j\pi$. Subtracting a suitable constant, we can
apply Corollary~\ref{cor:lob3} and obtain 
\begin{equation}\label{eq:stern41}
\hat F^{(0)}_{\tilde \eps}(w_\eps^{(j)};B^+_\rho) \ge \pi |d_j| \log\frac \rho{\tilde \eps}- M_2|d_j|=\pi |d_j| \log\frac r{\eps}- M_2|d_j|+2\pi |d_j| \log (1-c_1r\log\frac1r).
\end{equation}
Combining the results on each of the disjoint patches and dividing by $|\log\eps|$, it follows that
\[
\frac1{|\log\eps|}\Ge(\phi_\e) \ge\pi \sum_{k} |d_j| - \frac1{|\log\eps|} (\pi \sum_j |d_j| (|\log r|+M_2)+Cr^{\frac12}),
\]
and letting $\eps\to 0$ we obtain the first order lower bound as claimed.

\medskip

\noindent {\bf Step 2}. {\it Proof of point 2.} Assume now the stronger condition \eqref{eq:Gbetterub}. 
For small $r>0$, let $\rho=r(1-c_1r\log\frac1r)$. From  \eqref{eq:s40}, we then must have
\[
\sum_{j=1}^N \hat F^{(0)}_{\teps}(w_\eps^{(j)};B^+_\rho) \le \sum_{j=1}^N \pi |d_j| \log\frac\rho{\tilde\eps} + K_0,
\]
where $K_0=K_0(\rho)$ is independent of $\eps$.

For $\sigma<\rho$, we use Lemma \ref{lem:ABS2d1d} and Corollary~\ref{cor:lob3}, which shows for $\eps<\eps_0$
\[
\sum_{j=1}^N \hat F^{(0)}_{\teps}(w_{\eps}^{(j)};B^+_\sigma) \ge \sum_{j=1}^N \pi |d_j| \log\frac\sigma{\teps}- \sum_{j=1}^N |d_j| {M_2}
\]
so in $B^+_\rho\setminus B_\sigma$, we obtain
\be
\label{st1}
\limsup_{\eps\to 0}\sum_{j=1}^N \hat F^{(0)}_{\teps}(w_\eps^{(j)};B^+_\rho\setminus B_\sigma)\le  \sum_{j=1}^N \pi |d_j|\log\frac\rho\sigma + K_0+ \sum_{j=1}^N |d_j| {M_2}.
\ee
However, as $w_\eps^{(j)}\to w^{(j)}_*$ in $L^1(I_\rho)$
where $w^{(j)}_*$ is locally constant expect one jump point of size $d_j\pi$, 
by \eqref{estim_outer}, 
we get
\[
\liminf_{\eps\to 0}\sum_{j=1}^N \hat F^{(0)}_{\teps}(w_\eps^{(j)};B^+_\rho\setminus B_\sigma) \ge \sum_{j=1}^N \pi d_j^2 \log\frac\rho\sigma,
\]
so 
\[
\sum_{j=1}^N \pi (d_j^2-|d_j|) \log\frac\rho\sigma \le K_0+ \sum_{j=1}^N |d_j| {M_2}.
\]
Letting $\sigma\to 0$, we obtain that this is only possible if $\sum_{j=1}^N  (d_j^2-|d_j|) \le 0$, so $d_j=\pm 1$ as claimed.

From \eqref{eq:s40}, \eqref{eq:stern41} and \eqref{eq:Gbetterub}, we find 
the existence of a constant $K_1$ such that for every small $\eps>0$ and $r>0$, 
\[
 \int_{\Omega_r}|\nabla \phi_\eps|^2 dx 
 \le N\pi\log\frac1r + K_1,
\]
where $\Omega_r=\Om\setminus \bigcup_j B_r(a_j)$. We conclude using Lemma~\ref{lem:lpbd} that $\nabla \phi_\eps$ are uniformly bounded in $L^q(\Om)$ for every $q<2$.

It follows that there exists $\hat \phi_0\in W^{1,q}(\Om)$ such that 
for a subsequence, 
 $\phi_\eps\rightharpoonup \hat \phi_0$ weakly in $W^{1,q}(\Omega)$ and weakly  in $H^1(O)$ for any open  $O$ with $\overline{O}\subset \overline{\Om}\setminus \{a_1,\dots, a_N\}$. By the trace theorem, we deduce that $\hat \phi_0$ is an extension (in $\Om$) of the boundary limit $\phi_0$ found at point 1.  In order to prove the second order lower bound \eqref{eq:Gepslowerbd} for $\phi_\eps$, we replace $\phi_\eps$ by $\phi^*_\eps$ which is the harmonic extension of $\phi_\eps\big|_{\dOm}$ to $\Om$, i.e., $\phi^*_\eps$ is the minimiser of the Dirichlet energy in $\Om$ under the Dirichlet boundary condition $\phi_\eps\big|_{\dOm}$. Therefore, $\Ge(\phi_\eps)\geq 
 \Ge(\phi^*_\eps)$ and it is sufficient to prove \eqref{eq:Gepslowerbd} for $\phi^*_\eps$. 
By the above argument for the convergence of $\phi_\eps$, replacing $\phi_\eps$ by $\phi^*_\eps$, we know that  $\phi^*_\eps$ converges weakly in $W^{1,p}(\Omega)$ and weakly  in $H^1(O)$ for any open  $O$ with $\overline{O}\subset \overline{\Om}\setminus \{a_1,\dots, a_N\}$ to the harmonic extension $\phi_*$ to $\Om$ of $\phi_0:\dOm\to \R$. Using lower semicontinuity of the Dirichlet integral, we find by letting $\eps\to 0$:
\[
\int_{\Om_r} |\nabla \phi_*|^2\, dx \le \liminf_{\eps\to 0}\int_{\Omega_r}|\nabla \phi_\eps|^2 \, dx.
\] 

By definition of $W_\Omega$ we know that
\[
\int_{\Om_r} |\nabla \phi_*|^2dx = \pi N \log\frac1r + W_\Omega(\{(a_j,d_j)\})+o_{r}(1).
\]
From Proposition \ref{prop:onevortexGC} and \eqref{eq:s40}, we find
\be\label{eq:newlbba}
\liminf_{\eps\to 0}\left(\int_{\bigcup_j B_r(a_j)\cap \Om} |\nabla \phi_\eps|^2 dx + \frac{1}{2\pi\eps} \int_{\bigcup_j  B_r(a_j)\cap \dOm} \sin^2(\phi_\eps-g) d\h^1 - N(\pi\log\frac{r}\eps+\gamma_0)\right) \ge - CNr^{\frac12}.
\ee

Combining the last three relations, we see that
\[
\liminf_{\eps\to 0} \left(\Ge(\phi_\e)- \pi N \log\frac1\eps -N\gamma_0-W_\Omega(\{(a_j,d_j)\}) \right)  \ge o_r(1)
\]
Letting $r\to0$, we conclude  \eqref{eq:Gepslowerbd}.

\medskip

\noindent {\bf Step 3}. {\it Proof of point 3.} 
Let $\phi_*$ be the harmonic extension of $\phi_0$ given in Definition~\ref{defi:renen}, and let $r>0$ be a small radius. 
For each $j$ we use Lemma~\ref{lem:conformalflat} and find $\Psi_{a_j}:\ol{B^{+}_{2r}}\to\ol\Om$ as there.

Close to $a_j$, for a suitable choice of the argument function and arguing as in Lemma~\ref{lem:lem7}, $\phi_*=\pm d_j\arg(z-a_j)+h(z)$,
 for $h\in W^{1,p}$ in a neighbourhood of $a_j$ for all $p\in(1,\infty)$, with   bounds depending only {on $p$}, $\dOm$ and the choice of $\{(a_j,d_j)\}$ (since $g$ is Lipschitz). 
  Clearly $\sin(h-g)=0$ on $\dOm$.
  It follows using the Dini regularity of $\Psi_{a_j}$ that $\hat \phi_*=\phi_*\circ \Psi_{a_j}$ can be written as 
\[
\hat \phi_*(z) = \pm d_j\arg(z)+\hat h(z)
\]
{in a neighborhood of $0$} with $\hat h=h\circ \Psi_{a_j}$ {bounded in
${W^{1,p}}$ around the origin as above.}
We now define 
\[
\hat \psi_\eps = \begin{cases}
\phi_*(z) &\quad \text{{if }} |\Psi_{a_j}^{-1}(z)|>r\\
h(z)\pm\phi_\eps(\Psi_{a_j}^{-1}(z)) &\quad \text{elsewhere},
\end{cases}
\]
where $\phi_\eps$ is the function as defined in Step 1 of the proof of Lemma~\ref{lem:cabre} for $d_j=\pm 1$ 
and $\phi_\eps=\phi_{d,\eps}$ as in the proof of Lemma~\ref{lem:ubhigher} for $|d_j|>1$.
Then $\hat \psi_\eps$ is continuous in $\overline{\Omega}$. 
From our construction, it is clear that $\hat \psi_\e \to \phi_0$ in all $L^p(\dOm)$.

Using Lemma~\ref{lem:conformalflat} and the definition of $\phi_*$, denoting $E_r=\bigcup_{j=1}^N(B_r(a_j)\triangle \Psi_{a_j}(B_r^+))$, we have 
\[
\int_{E_r} |\nabla \hat \psi_\eps|^2\, dx \le C\int_{r(1-c_1r\log\frac1r)}^{r(1+c_1r\log\frac1r)} \frac1s\, ds =O(r\log\frac1r).
\] 
Inside $\Psi_{a_j}(B_r^+)$, we compute the energy of $\hat \psi_\eps$. By conformal mapping, we can compute it in $B_r^+$.
Note that 
\[
\int_{B_r^+} |\nabla \phi_\eps+{\nabla}\hat h|^2 \, dx = \int_{B_r^+} |\nabla \phi_\eps|^2\, dx + 2 \int_{B_r^+} \nabla \phi_\eps\cdot {\nabla}\hat h\, dx + 
\int_{B_r^+} |{\nabla}\hat h|^2 \, dx.
\]
For $d_j=\pm 1$, by  Lemma~\ref{lem:cabre}, 
\[
\hat F_\eps( \phi_\eps;B_r^+)\le \pi \log\frac r\eps+\gamma_0 +o_r(1).
\]

Using $\int_{B_r\cap \R} \sin^2 \phi_\e \,d\h^1 \le C$, we find using Lemma~\ref{lem:conformalflat} that
\[
 \int_{B_r(a_j)\cap \dOm} \sin^2(\hat \psi_\e - g) \, d\h^1 = \int_{B_r\cap \R} \sin^2 \phi_\eps \, d\h^1+ O(r\log\frac1r).
\] 

As $\hat h\in {W^{1,p}}$  and $|\nabla \phi_\eps(x,y)|\le \frac1{|(x,y)|}$, we can estimate {using H\"older's inequality}
\[
\left|\int_{B_r^+} \nabla \phi_\eps\cdot {\nabla}\hat h\, dx \right|
{ \le \left(\int_{B_r^+} |\nabla \phi_\eps|^{\frac32}\, dx \right)^{\frac23} \left(\int_{B_r^+} |\nabla \hat h|^3\, dx\right)^\frac13}
\le {o_r(1)} \quad \textrm{and}\quad 
\int_{B_r^+} |{\nabla}\hat h|^2 \, dx{\le o_r(1)}.
\]
As the Dirichlet energy of $\phi_*$ in $\Omega\setminus \bigcup_{j=1}^N B_r(a_j)$
is $W_\Om(\{a_j, d_j\})+N\pi |\log r|+o_r(1)$, we can thus establish that
 the upper bound \eqref{eq:gammaG_ub} holds {for $\hat \psi_\e$ with an error $o_r(1)$. Choosing $r$ sufficiently small, we see that \eqref{eq:gammaG_ub} must hold.
 Replacing $\hat \psi_\e$ by the harmonic function with the same boundary conditions, the energy does not increase.
As harmonic functions satisfy $\|f\|_{L^2(\Om)}\le C\|f\|_{L^2(\dOm)}$, 
 we obtain that $\hat \psi_\e\to \phi_0$ in $L^2(\Omega)$ and hence (by boundedness) in all $L^p(\Omega)$.
 }
For $|d_j|>1$, applying the result of Lemma~\ref{lem:ubhigher} similarly leads to \eqref{eq:gammaG_ub1}.\footnote{For a different proof of the upper bound construction in the case $g=0$, we refer to \cite{CCK19}.}

\qed

\section{Second order $\Gamma$-convergence for the full energy. Proof of Theorems \ref{thm:firstocomp}, \ref{thm:gammacforee} and \ref{thm:gcubintro}}
\label{sec:gammacee}
In this section we prove Theorems~\ref{thm:firstocomp}, \ref{thm:gammacforee} and \ref{thm:gcubintro}.
\begin{proof}{~of Theorem~\ref{thm:firstocomp}}
Let $(u_\e)$ be a {sequence / family} of maps with $\eee(u_\e)\le C|\log\e|$. Then we can use Theorem~\ref{lem_approx} to construct a {sequence / family} $U_\eps$ with the the following properties:
\begin{itemize}
\item $U_\eps\in H^1(\Om;\Ss^1)$;
\item ${\|U_\e-u_\e\|_{L^p(\Om)}\to 0}$, ${\|U_\e-u_\e\|_{L^p(\dOm)}\to 0}$ as $\e\to 0$ for every $p\geq 1$;
\item $\eee(U_\eps) \le \eee(u_\eps)+o_\e(1)$
\item  ${\jaco(U_\eps)-\jaco(u_\eps)\to 0}$ in $({\rm Lip}(\Om))^*$ as $\e\to 0$.
\end{itemize}
By Lemma~\ref{lem:lifting}, we find $\phi_\eps$ such that $U_\eps=e^{i\phi_\eps}$ and $\eee(U_\eps) = {\Ge}(\phi_\eps)$, with $\Ge$ defined in \eqref{eq:Geps}. 
The global Jacobian of $U_\eps$ is given by \eqref{jac_s1} as $\jaco(U_\eps) = -\partial_\tau \phi_\eps  \h^1\llcorner \partial\Omega$. 
By Theorem \ref{thm:GCforGeps}, {for a subsequence, there exist integers $z_\e\in \ZZ$ (either all of them are even, or all are odd) such that $\phi_\eps-\pi z_\e$ converges to a limit $\phi_0$ in any $L^p(\dOm)$ where $\phi_0-g \in BV(\dOm;\pi\ZZ)$ with $g$ given in \eqref{g}. As $|e^{is}-e^{it}|\leq \frac\pi 2|s-t|$ for every $s,t\in \R$, we deduce that $U_\e\to \pm e^{i\phi_0}$  in any $L^p(\dOm)$. Changing $\phi_0$ in $\phi_0-\pi$ (in the case where $z_\e$ are all odd), we obtain  the desired convergence $u_\e\to e^{i\phi_0}$  in any $L^p(\dOm)$. Moreover,}
the convergence $\partial_\tau\phi_\eps \to \partial_\tau\phi_0$ in $W^{-1,p}(\partial\Om)$ for any $p\in (1, \infty)$ directly induces the convergence of $\jaco(U_\eps)=\jacbd(U_\eps)$ as claimed in $({\rm Lip}(\Om))^*$.
As  $\| \jaco(U_\eps)-\jaco(u_\eps)\|\to 0$ in $({\rm Lip}(\Om))^*$, we obtain that $\jaco(u_\eps)$ tends to the same limit.
Since $\eee(u_\e)\ge \eee(U_\e)-o(1)=\Ge(\phi_\e)-o(1)$, the lower bounds for $\Ge(\phi_\eps)$ directly translate into the claimed lower bound for $\eee(u_\e)$ at the first order.
\qed
\end{proof}

\medskip

\begin{proof}{~of Theorem~\ref{thm:gammacforee}}
Continuing as in the previous proof (within the same notation), we note that the estimate \eqref{eq:sharpenbd} transfers to $\Ge(\phi_\eps)$ so that Theorem \ref{thm:GCforGeps} yields the claims about $|d_i|=\pm 1$  and again, 
$\eee(u_\e)\ge \eee(U_\e)-o(1)=\Ge(\phi_\e)-o(1)$ implies the desired lower bounds for $\eee$ at the second order. 

To show the $L^q(\Om)$ bound for $\nabla u_\eps$ for every $q<2$, we proceed as follows: Using the boundary vortices $a_j$ with their multiplicities  $d_j$ from Theorem~\ref{thm:GCforGeps} coming from the lifting $\phi_\eps$ of $U_\eps$, we have
by \eqref{eq:newlbba} that 
\[
{\liminf_{\rho\to 0}} \liminf_{\eps\to 0}\left( \int_{\bigcup B_\rho(a_j)\cap \Om} |\nabla U_\eps|^2 dx + \frac{1}{2\pi\eps} \int_{\bigcup B_\rho(a_j)\cap \dOm} (U_\eps\cdot\nu)^2 d\h^1 - N(\pi\log\frac\rho\eps+\gamma_0)\right) \ge 0.
\]

From \eqref{eq:localapprox} applied to 
{$G=B_{\rho(1+\rho)}(a_j)\cap \Om$ so that $G_\eta\supset B_\rho(a_j)\cap \Om$}, 
we now deduce (since {$\log\frac{\rho(1+\rho)}{\rho}\to 0$} as $\rho\to 0$)
\be
\label{dsta}
\liminf_{\rho\to 0} \liminf_{\eps\to 0}\Bigl( \int_{\bigcup B_\rho(a_j)\cap \Om} |\nabla u_\eps|^2+\frac1{\eta^2}(1-|u_\eps|^2)^2 \, dx
+ \frac{1}{2\pi\eps} \int_{\bigcup B_\rho(a_j)\cap \dOm} (u_\eps\cdot\nu)^2 d\h^1 - N(\pi\log\frac\rho\eps+\gamma_0)\Bigr) \ge 0.
\ee
Using Corollary~\ref{cor:lob3} and \eqref{eq:localapprox}, we also find $C$ such that for $\rho<\rho_0$, $\eps<\eps_0$: 
\be\label{dsta2}
\int_{\bigcup B_\rho(a_j)\cap \Om} |\nabla u_\eps|^2+\frac1{\eta^2}(1-|u_\eps|^2)^2 \, dx
+ \frac{1}{2\pi\eps} \int_{\bigcup B_\rho(a_j)\cap \dOm} (u_\eps\cdot\nu)^2 d\h^1 \ge N\pi \log\frac\rho\eps-C,
\ee
hence, by \eqref{eq:sharpenbd},
\[
\int_{\Omega \setminus \bigcup B_{\rho}(a_j)} |\nabla u_\e|^2 \, dx \le N\pi \log \frac1\rho+C.
\]
 so Lemma \ref{lem:lpbd} applies and shows that $\limsup_{\e\to 0}\|\nabla u_\e\|_{L^q(\Omega)}<\infty$ for every $q<2$.

Finally, we need to show  \eqref{eq:thpenub}, which clearly implies \eqref{eq:l2lob} via \eqref{dsta2}. 
However, \eqref{eq:thpenub} follows from 
the exact same argument as used in Lemma~\ref{lem:ubpen}: 
Let $\tilde u_{2\e}:=u_\e$, $\tilde U_{2\e}:=U_\e$ and as in \eqref{dsta2}, apply Corollary \ref{cor:lob3} for $(\tilde u_{2\e})$ to get
for $\eps$ sufficiently small, 
\[
E_{2\eps, 2\eta}(u_{\eps})=E_{2\eps, 2\eta}(\tilde u_{2\eps})\ge N\pi \log \frac1{2\eps} -C,
\]
while by the upper bound  for some fixed $\rho>0$:
\[
\eee(u_{\eps}) \le N\pi \log \frac1{\eps} +C,
\]
so 
\[
\frac3{4\eta^2} \int_\Om (1-|u_{\eps}|^2)^2 \, dx + \frac1{4\pi \eps } \int_{\dOm} (u_{\eps}\cdot \nu)^2\, d\h^1=
\eee(u_{\eps}) -E_{2\eps, 2\eta}(u_{\eps})\le \tilde C,
\]
which clearly implies \eqref{eq:thpenub}.

For point $iv)$, by Theorem \ref{thm:GCforGeps}, we know that up to a subsequence and an additive constant, $\phi_\eps\to \phi_0$ a.e. in $\dOm$ which by dominated convergence theorem implies that $U_\e=e^{i\phi_\e}\to e^{i\phi_0}$ in every $L^p(\dOm)$ for $p\geq 1$. By \eqref{aprox_bdry}, we know that $u_\e-U_\e\to 0$ in $L^p(\dOm)$, therefore $u_\e- e^{i\phi_0}\to 0$ in $L^p(\dOm)$ 
{As $(u_\eps)$ is bounded in $W^{1,q}(\Om)$ for every $q\in [1,2)$, by the trace theorem and Theorem \ref{thm:GCforGeps} point 2), for a subsequence, $u_\e$ converges weakly in $W^{1,q}(\Om)$ and strongly in $L^p(\Om)$ for every $p\geq 1$ to an $\Ss^1$-valued extension $e^{i\hat{\phi_0}}$ of $e^{i\phi_0}$ in $\Om$}.
\qed
\end{proof}

\medskip

\begin{proof}{~of Theorem~\ref{thm:gcubintro}}
The upper bound construction is a direct consequence of the corresponding construction for $\Ge$ in {Theorem~\ref{thm:GCforGeps}}. With $\hat \psi_\eps$ as constructed there, 
we set $u_\eps = e^{i\hat \psi_\eps}$, then $|u_\eps|=1$ and $\jaco(u_\eps) = -\partial_\tau \hat \psi_\eps  \h^1\llcorner \partial\Omega$, and then the convergence and 
energy bound results follow directly using $\eee(u_\e) = \Ge(\hat \psi_\e)$. 
\qed
\end{proof}

\paragraph{Acknowledgment.} R.I. acknowledges partial support by the ANR project ANR-14-CE25-0009-01.

\bibliographystyle{plain}
\bibliography{ignatkurzke_arxiv}

\begin{thebibliography}{10}

\bibitem{boojum}
Stan {Alama}, Lia {Bronsard}, and Dmitry {Golovaty}.
\newblock {Thin Film Liquid Crystals with Oblique Anchoring and Boojums}.
\newblock {\em arXiv:1907.04757}, Jul 2019.

\bibitem{AlbertiBouchittSeppeche:1994a}
Giovanni Alberti, Guy Bouchitt{\'e}, and Pierre Seppecher.
\newblock Un r\'esultat de perturbations singuli\`eres avec la norme {$H\sp
  {1/2}$}.
\newblock {\em C. R. Acad. Sci. Paris S\'er. I Math.}, 319(4):333--338, 1994.

\bibitem{AlbertiBouchittSeppeche:1998a}
Giovanni Alberti, Guy Bouchitt{\'e}, and Pierre Seppecher.
\newblock Phase transition with the line-tension effect.
\newblock {\em Arch. Rational Mech. Anal.}, 144(1):1--46, 1998.

\bibitem{Alicandro:2014aa}
Roberto Alicandro and Marcello Ponsiglione.
\newblock Ginzburg-{L}andau functionals and renormalized energy: a revised
  {$\Gamma$}-convergence approach.
\newblock {\em J. Funct. Anal.}, 266(8):4890--4907, 2014.

\bibitem{Axler:2001aa}
Sheldon Axler, Paul Bourdon, and Wade Ramey.
\newblock {\em Harmonic function theory}, volume 137 of {\em Graduate Texts in
  Mathematics}.
\newblock Springer-Verlag, New York, second edition, 2001.

\bibitem{BEK19}
Marco Baffetti, Timothy Espin, and Matthias Kurzke.
\newblock A single multiplicity result for boundary vortices.
\newblock In preparation, 2019.

\bibitem{Ball:1989aa}
J.~M. Ball.
\newblock A version of the fundamental theorem for {Y}oung measures.
\newblock In {\em P{DE}s and continuum models of phase transitions ({N}ice,
  1988)}, volume 344 of {\em Lecture Notes in Phys.}, pages 207--215. Springer,
  Berlin, 1989.

\bibitem{BethuelBrezisHelein:1994a}
Fabrice Bethuel, Ha{\"{\i}}m Brezis, and Fr{\'e}d{\'e}ric H{\'e}lein.
\newblock {\em Ginzburg-{L}andau vortices}.
\newblock Progress in Nonlinear Differential Equations and their Applications,
  13. Birkh\"auser Boston Inc., Boston, MA, 1994.

\bibitem{BethuelZheng:1988a}
Fabrice Bethuel and Xiao~Min Zheng.
\newblock Density of smooth functions between two manifolds in {S}obolev
  spaces.
\newblock {\em J. Funct. Anal.}, 80(1):60--75, 1988.

\bibitem{BreNgu}
Ha{\"{\i}}m Brezis and Hoai-Minh Nguyen.
\newblock The {J}acobian determinant revisited.
\newblock {\em Invent. Math.}, 185(1):17--54, 2011.

\bibitem{CCK19}
Xavier Cabr{\'e}, Neus {C{\'o}nsul}, and Matthias {Kurzke}.
\newblock Minimizers for boundary reactions: renormalized energy, location of
  singularities, and applications, 2019.
\newblock In preparation.

\bibitem{CabreSola-Mor:2005a}
Xavier Cabr{\'e} and Joan Sol{\`a}-Morales.
\newblock Layer solutions in a half-space for boundary reactions.
\newblock {\em Comm. Pure Appl. Math.}, 58(12):1678--1732, 2005.

\bibitem{Cantero:2009a}
Rub\'en Cantero-{\'A}lvarez.
\newblock {\em Pattern, Walls and Vortex: A micromagnetic excursion}.
\newblock PhD thesis, Universit\"at Bonn, 2009.

\bibitem{ColliandJerrard:1999a}
J.~E. Colliander and R.~L. Jerrard.
\newblock Ginzburg-{L}andau vortices: weak stability and {S}chr\"odinger
  equation dynamics.
\newblock {\em J. Anal. Math.}, 77:129--205, 1999.

\bibitem{Cote:2014aa}
Rapha\"el C\^ote, Radu Ignat, and Evelyne Miot.
\newblock A thin-film limit in the {L}andau-{L}ifshitz-{G}ilbert equation
  relevant for the formation of {N}\'eel walls.
\newblock {\em J. Fixed Point Theory Appl.}, 15(1):241--272, 2014.

\bibitem{Pino:1997aa}
Manuel del Pino and Patricio~L. Felmer.
\newblock Local minimizers for the {G}inzburg-{L}andau energy.
\newblock {\em Math. Z.}, 225(4):671--684, 1997.

\bibitem{GarroniMuller:2006a}
Adriana Garroni and Stefan M{\"u}ller.
\newblock A variational model for dislocations in the line tension limit.
\newblock {\em Arch. Ration. Mech. Anal.}, 181(3):535--578, 2006.

\bibitem{GarsiaRodemich:1974a}
A.~M. Garsia and E.~Rodemich.
\newblock Monotonicity of certain functionals under rearrangement.
\newblock {\em Ann. Inst. Fourier (Grenoble)}, 24(2):vi, 67--116, 1974.

\bibitem{GradshteRyzhik:2007a}
I.~S. Gradshteyn and I.~M. Ryzhik.
\newblock {\em Table of integrals, series, and products}.
\newblock Elsevier/Academic Press, Amsterdam, seventh edition, 2007.
\newblock Translated from the Russian, Translation edited and with a preface by
  Alan Jeffrey and Daniel Zwillinger, With one CD-ROM (Windows, Macintosh and
  UNIX).

\bibitem{Ig_CV}
Radu Ignat.
\newblock Optimal lifting for {${\rm BV}(S^1,S^1)$}.
\newblock {\em Calc. Var. Partial Differential Equations}, 23(1):83--96, 2005.

\bibitem{Ig_JacBV}
Radu Ignat.
\newblock The space {${\rm BV}(S^2,S^1)$}: minimal connection and optimal
  lifting.
\newblock {\em Ann. Inst. H. Poincar\'{e} Anal. Non Lin\'{e}aire},
  22(3):283--302, 2005.

\bibitem{Ignat_HDR}
Radu Ignat.
\newblock Singularities of divergence-free vector fields with values into
  {$S^1$} or {$S^2$}. {A}pplications to micromagnetics.
\newblock {\em Confluentes Math.}, 4(3):1230001, 80, 2012.

\bibitem{IgJe}
Radu Ignat and Robert~L. Jerrard.
\newblock Interaction energy between vortices of vector fields on {R}iemannian
  surfaces.
\newblock {\em C. R. Math. Acad. Sci. Paris}, 355(5):515--521, 2017.

\bibitem{IgJeP}
Radu Ignat and Robert~L. Jerrard.
\newblock Renormalized energy between vortices in some {G}inzburg-{L}andau
  models on $2$-dimensional {R}iemannian manifolds.
\newblock arXiv:1910.02921, 2019.

\bibitem{IK_bdv}
Radu Ignat and Matthias Kurzke.
\newblock An effective model for boundary vortices in thin-film micromagnetics.
\newblock Preprint, 2019.

\bibitem{IKL19}
Radu {Ignat}, Matthias {Kurzke}, and Xavier {Lamy}.
\newblock {Global uniform estimate for the modulus of $2D$ Ginzburg-Landau
  vortexless solutions with asymptotically infinite boundary energy}.
\newblock {\em arXiv:1904.00856}, Apr 2019.

\bibitem{IgnatOtto:2011a}
Radu Ignat and Felix Otto.
\newblock A compactness result for {L}andau state in thin-film micromagnetics.
\newblock {\em Ann. Inst. H. Poincar\'e Anal. Non Lin\'eaire}, 28(2):247--282,
  2011.

\bibitem{Jerrard}
Robert~L. Jerrard.
\newblock Lower bounds for generalized {G}inzburg-{L}andau functionals.
\newblock {\em SIAM J. Math. Anal.}, 30(4):721--746, 1999.

\bibitem{JerrardSoner:2002b}
Robert~L. Jerrard and Halil~Mete Soner.
\newblock The {J}acobian and the {G}inzburg-{L}andau energy.
\newblock {\em Calc. Var. Partial Differential Equations}, 14(2):151--191,
  2002.

\bibitem{KohnSlastiko:2005a}
Robert~V. Kohn and Valeriy~V. Slastikov.
\newblock Another thin-film limit of micromagnetics.
\newblock {\em Arch. Ration. Mech. Anal.}, 178(2):227--245, 2005.

\bibitem{Kurzke:2006b}
Matthias Kurzke.
\newblock Boundary vortices in thin magnetic films.
\newblock {\em Calc. Var. Partial Differential Equations}, 26(1):1--28, 2006.

\bibitem{Kurzke:2006a}
Matthias Kurzke.
\newblock A nonlocal singular perturbation problem with periodic well
  potential.
\newblock {\em ESAIM Control Optim. Calc. Var.}, 12(1):52--63, 2006.

\bibitem{Kurzke:2007a}
Matthias Kurzke.
\newblock The gradient flow motion of boundary vortices.
\newblock {\em Ann. Inst. H. Poincar\'e Anal. Non Lin\'eaire}, 24(1):91--112,
  2007.

\bibitem{Moser:2003a}
Roger Moser.
\newblock Ginzburg-{L}andau vortices for thin ferromagnetic films.
\newblock {\em AMRX Appl. Math. Res. Express}, (1):1--32, 2003.

\bibitem{Moser:2004a}
Roger Moser.
\newblock Boundary vortices for thin ferromagnetic films.
\newblock {\em Arch. Ration. Mech. Anal.}, 174(2):267--300, 2004.

\bibitem{Muller:1999a}
Stefan M{\"u}ller.
\newblock Variational models for microstructure and phase transitions.
\newblock In {\em Calculus of variations and geometric evolution problems
  (Cetraro, 1996)}, volume 1713 of {\em Lecture Notes in Math.}, pages 85--210.
  Springer, Berlin, 1999.

\bibitem{Pommerenke:1992aa}
Ch. Pommerenke.
\newblock {\em Boundary behaviour of conformal maps}, volume 299 of {\em
  Grundlehren der Mathematischen Wissenschaften [Fundamental Principles of
  Mathematical Sciences]}.
\newblock Springer-Verlag, Berlin, 1992.

\bibitem{Sandier:1998a}
Etienne Sandier.
\newblock Lower bounds for the energy of unit vector fields and applications.
\newblock {\em J. Funct. Anal.}, 152(2):379--403, 1998.

\bibitem{SandierSerfaty:2007a}
Etienne Sandier and Sylvia Serfaty.
\newblock {\em Vortices in the magnetic {G}inzburg-{L}andau model}.
\newblock Progress in Nonlinear Differential Equations and their Applications,
  70. Birkh\"auser Boston Inc., Boston, MA, 2007.

\bibitem{Struwe:1994a}
Michael Struwe.
\newblock On the asymptotic behavior of minimizers of the {G}inzburg-{L}andau
  model in {$2$} dimensions.
\newblock {\em Differential Integral Equations}, 7(5-6):1613--1624, 1994.

\bibitem{Taylor:1997a}
Michael~E. Taylor.
\newblock {\em Partial differential equations. {III}}, volume 117 of {\em
  Applied Mathematical Sciences}.
\newblock Springer-Verlag, New York, 1997.
\newblock Nonlinear equations, Corrected reprint of the 1996 original.

\bibitem{Toland:1997a}
J.~F. Toland.
\newblock The {P}eierls-{N}abarro and {B}enjamin-{O}no equations.
\newblock {\em J. Funct. Anal.}, 145(1):136--150, 1997.

\bibitem{Valadier:1994aa}
Michel Valadier.
\newblock A course on {Y}oung measures.
\newblock {\em Rend. Istit. Mat. Univ. Trieste}, 26(suppl.):349--394 (1995),
  1994.
\newblock Workshop on Measure Theory and Real Analysis (Italian) (Grado, 1993).

\end{thebibliography}
\end{document}